\documentclass{amsart}

\usepackage{color}
\usepackage{amsthm,amssymb,verbatim}
\usepackage{graphicx}
\usepackage{enumerate}
\DeclareRobustCommand{\SkipTocEntry}[5]{}

\newcommand{\nupt}{\nu_{\partial M}}

\newcommand{\Mm}{\mathcal{M}}

\newcommand{\length}{\operatorname{length}}

 \newcommand{\graph}{\operatorname{graph}}
 \newcommand{\Cc}{\mathcal C}

 \newcommand{\MM}{\mathcal{M}}
 
 \newcommand{\slope}{\operatorname{slope}}
 \newcommand{\FF}{\mathcal{F}}
 \newcommand{\sS}{\mathbf{S}}

 \newcommand{\RR}{\mathbf{R}}  
 \newcommand{\ZZ}{\mathbf{Z}}  
 \newcommand{\BB}{\mathbf{B}}  
 \newcommand{\dist}{\operatorname{dist}}
 
 \newcommand{\area}{\operatorname{area}}
 \newcommand{\eps}{\epsilon}
 \newcommand{\Tan}{\operatorname{Tan}}

 \newcommand{\genus}{\operatorname{genus}}

\newcommand{\interior}{\operatorname{interior}}

\newcommand{\vv}{\bold v}

\newcommand{\ee}{\mathbf{e}}

\newcommand{\grad}{\nabla}

 \newcommand{\Div}{\operatorname{div}}

\newcommand{\pdf}[2]{\frac{\partial #1}{\partial #2}}

\newcommand{\talpha}{\tilde\alpha}

\input epsf
\def\begfig {
\begin{figure}[b]
\small }
\def\endfig {
\normalsize
\end{figure}
}

    \newtheorem{THEOREM} {Theorem}         

\swapnumbers

    \newtheorem{theorem}    {Theorem}       [section]
    \newtheorem{lemma}      [theorem]       {Lemma}
    \newtheorem{corollary}  [theorem]     {Corollary}
    \newtheorem{proposition}       [theorem]       {Proposition}
    
    \newtheorem{alt-claim}    [theorem]       {Claim}                     
    \newtheorem*{claim}{Claim}
    \newtheorem*{theorem*}{Theorem}
    \newtheorem*{corollary*}{Corollary}

    \theoremstyle{definition}
    \newtheorem{definition}  [theorem] {Definition}

    \theoremstyle{definition}
    \newtheorem{remark}   [theorem]       {Remark}

\usepackage{hyperref} 
\usepackage[alphabetic, msc-links, backrefs]{amsrefs}

\begin{document}

\renewcommand{\thesubsection}{\thetheorem}

\title{Helicoidal minimal surfaces of prescribed genus, I}
\subjclass[2000]{Primary: 53A10; Secondary: 49Q05, 53C42}
\author{David Hoffman}
\address{Department of Mathematics\\ Stanford University\\ Stanford, CA 94305}
\email{hoffman@math.stanford.edu}
\author{Martin Traizet}
\address{Laboratoire de Math\'{e}matiques et Physique Th\'{e}orique,Universit\'{e} Fran\c cois Rabelais, 37200 Tours, France}
\email{Martin.Traizet@lmpt.univ-tours.fr}
\author{Brian White}
\address{Department of Mathematics\\ Stanford University\\ Stanford, CA 94305}
\thanks{The research of the third author was supported by NSF
   grants~DMS--1105330  }
\email{white@math.stanford.edu}
\date{April 6, 2013.}
\begin{abstract}
For every genus $g$, we prove that $\sS^2\times\RR$
contains complete, properly embedded, genus-$g$ minimal
surfaces whose two ends are asymptotic to helicoids of any prescribed pitch.   
We also show that as the radius
of the $\sS^2$ tends to infinity, these examples converge smoothly to complete, properly embedded minimal surfaces in
$\RR^3$ that are helicoidal at infinity.   
In a companion paper, we prove that helicoidal surfaces in $\RR^3$ of every prescribed genus occur as such limits of examples in $\sS^2\times\RR$.
\end{abstract}

\maketitle

\tableofcontents

\section{Introduction}\label{section:intro}
The study of complete, properly embedded minimal surfaces in 
  $\Sigma\times \RR$, where $\Sigma$ is complete Riemannian $2$-manifold, was
initiated by Harold Rosenberg in \cite{rosenberg2002}.
The general theory of such surfaces was further developed by Meeks and Rosenberg
 in~\cite{MeeksRosenberg2005}.
In the case of $\sS^2\times\RR$, 
 if such a surface has finite topology, then
 either it is a union of horizontal spheres $\sS^2\times\{t\}$, 
or else it is conformally a connected, twice-punctured, 
compact Riemann surface,
with one end going up and the other end going down~\cite{rosenberg2002}*{Theorems~3.3,~4.2,~5.2}.
In that same paper, Rosenberg described a class of such
surfaces in $\sS^2\times \RR$ 
that are very similar to helicoids in $\RR^3$, and hence
are also called helicoids.   
They may be characterized as the complete, non-flat minimal
surfaces in $\sS^2\times\RR$ whose 
horizontal slices are all great circles. 
(See Section~\ref{helicoid-guide-section} for a more explicit
description of helicoids in $\sS^2\times\RR$ and for a 
discussion of their basic properties.) He went on to construct  for each genus a minimal surface
 in $\sS^2\times\RR$
 whose two ends are asymptotic to a totally geodesic cylinder.

In this paper we prove the existence of properly embedded minimal
surfaces in $\sS ^2\times \RR$  of  prescribed finite topology, with top and
bottom ends asymptotic to  an end of a helicoid of any
prescribed pitch. (The pitch of a helicoid in $\sS ^2\times \RR$   is defined in section~\ref{helicoid-guide-section}. The absolute value of the pitch
is twice the vertical distance between successive sheets of the helicoid. The sign of the pitch depends on the sense in which the helicoid winds about its axes.) 
Although  the pitch of the helicoid  to which the top end is asymptotic equals 
 pitch of the helicoid to which the bottom is asymptotic, we do not know if these two helicoids coincide;
one might conceivably be a vertical translate of the other.  Each of the surfaces
we produce contains a pair of antipodal vertical lines $Z$ and $Z^*$ (called axes of the surface)
and a horizontal great circle $X$ that intersects each of the axes.  Indeed, for each  our surfaces, there is a helicoid whose intersection with the surface is precisely 
$X\cup Z \cup Z^*$.

 For every genus, our method produces two examples that are not
 congruent to each other by any orientation-preserving isometry 
 of $\sS^2\times \RR$. 
 The two examples are distinguished by their behavior at the origin $O$:
 one is ``positive" at $O$ and the other is ``negative" at $O$. (The
 positive/negative terminology is explained in Section~\ref{sign-section}.) 
 If the genus is odd, the two examples
 are congruent to each other by reflection $\mu_E$ in the totally geodesic cylinder
 consisting of all points equidistant from the two axes.
 If the genus is even, the two examples are not congruent to each other
 by any isometry of $\sS^2\times \RR$, but each one is invariant under the
 reflection $\mu_E$.
 The examples of even genus $2g$ are also invariant under 
 $(p,z)\mapsto (\tilde p, z)$, where $p$ and $\tilde p$ are antipodal points 
  in $\sS^2$, so their quotients under this involution are
 genus-$g$ minimal surfaces in $\RR \mathbf{P}^2\times\RR$ with helicoidal ends. 

For each genus $g$ and for each helicoidal pitch, we prove that as the radius
of the~$\sS^2$ tends to infinity, our examples converge subsequentially
to complete, properly embedded minimal surfaces in $\RR^3$ that are asymptotic
to helicoids at infinity.   In general, some of the handles  will drift off to infinity.
Indeed, a priori, the limiting surface might have genus $0$ even if $g$ is large.
However, in the companion paper \cite{hoffman-traizet-white-2}, we show that when $g$ is even, we can ensure
that exactly half of the handles drift away.  It follows that $\RR^3$ contains
properly embedded, helicoidal minimal surfaces of every genus.

The arguments required to keep too many handles from drifting away
are rather delicate.  It is much easier to control whether the 
limiting surface has odd or even genus: a limit (as the radius of $\sS^2$
tends to infinity) of ``positive'' examples must have even genus and a limit
of ``negative'' examples must have odd genus.
Such parity control   is sufficient (without the delicate arguments of~\cite{hoffman-traizet-white-2})
to give a new  proof of the existence of a
 genus-one helicoid in $\RR^3$. Previous proofs were given in \cite{hoffman-white-genus-one} and 
 \cite{weber-hoffman-wolf}.
 See the corollary to theorem~\ref{theorem2} in section~\ref{section:main-theorems} for
 details.
 
Returning to our discussion of examples in $\sS^2\times \RR$, we also prove existence of what might be
called periodic genus-$g$ helicoids.  They are properly embedded minimal
surfaces  that are invariant under a screw motion of $\sS^2\times\RR$
and that have fundamental domains of genus $g$.  Indeed, our nonperiodic examples in  $\sS^2\times\RR$ are obtained as limits of the periodic examples as the period  tends to infinity.

As mentioned above, all of our examples contain two vertical axes $Z\cup Z^*$
and a horizontal great circle $X\subset \sS^2\times\{0\}$ at height $0$.   Let $Y$ be the great circle at height $0$ such that $X$, $Y$, and $Z$ meet
orthogonally at a pair of points $O\in Z$ and $O^*\in Z^*$.
All of our examples are invariant
under $180^\circ$ rotation about $X$, $Y$, and $Z$
 (or, equivalently,  about $Z^*$: rotations about $Z$ are also rotations about $Z^*$).
In addition, the nonperiodic examples (and suitable fundamental domains
of the periodic examples) are what we call ``$Y$-surfaces".  Intuitively, this
means that they are $\rho_Y$-invariant (where $\rho_Y$ is $180^\circ$
rotation about $Y$) and that the handles (if there are any)  occur along $Y$. 
The precise definition is that $\rho_Y$ acts by multiplication by $-1$ on
the first homology group of the surface.
This property is very useful because it means that when we
let the period of the periodic examples tend to infinity, the handles cannot drift
away: they are trapped along $Y$, which is compact.  
In the companion paper~\cite{hoffman-traizet-white-2}, when we need to control handles drifting off to infinity
as we let the radius of the $\sS^2$ tend to infinity, 
the $Y$-surface property means that 
the handles can only drift off in one direction (namely along $Y$).
 
This paper is organized as follows.
In section~\ref{helicoid-guide-section}, we present the basic facts about helicoids
in $\sS^2\times\RR$.
In section~\ref{section:main-theorems}, we state
the main results of this paper and the companion paper~\cite{hoffman-traizet-white-2}. 
In section~\ref{sign-section}, we describe what it means
for a surface to be positive or negative at $O$ with respect to $H$.
In section~\ref{section:$Y$-surfaces}, we describe the general properties
of $Y$-surfaces.
In sections~\ref{construction-outline-section}\,--\ref{smooth-count-section},
 we prove existence of periodic genus-$g$ helicoids in 
  $\sS^2\times \RR$.
In sections~\ref{general-results-section} and~\ref{area-bounds-section} 
we present general results we will use in order to establish the existence of limits. 
In section~\ref{nonperiodic-section}, 
we get nonperiodic genus $g$ helicoids as limit of periodic
examples by letting the period tend to infinity.
In section~\ref{R3-section} and \ref{Proof_of_theorem_2}, we prove that as the radius of the $\sS^2$ tends to infinity,
our nonperiodic genus-$g$ helicoids in $\sS^2\times\RR$ converge
to properly embedded, minimal surfaces in $\RR^3$ with helicoidal ends.
\stepcounter{theorem}
 \section{Helicoidal minimal surfaces in $\sS ^2\times \RR$}\label{helicoid-guide-section}

\stepcounter{theorem}
\addtocontents{toc}{\SkipTocEntry}
\subsection*{Helicoids in  $\sS ^2\times \RR$}
\label{mainresults}

In $\sS^2\times\RR$, 
 fix a point $O$ in $\sS ^2\times\{0\}$ and let
$Z$ be the vertical line through $O$.  
Define $\sigma_{\theta,v}$ to be the screw motion
of $\sS^2\times\RR$ given by clockwise rotation through angle $\theta$ about
$Z$ followed by vertical translation by $v$. If $C$ is a great
circle in $\sS^2\times\{0\}$ passing through   $O$ and its
antipodal point $O^*$, define for any constant $\kappa \neq 0$,
\begin{equation}\label{HCkappa}
   H_{C,\kappa}=\bigcup_{t\in \RR} \sigma_{2\pi t, \kappa t}(C).
\end{equation}
We say that $H_{C,\kappa}$ is {\em the helicoid of pitch $\kappa$
generated by $C$}.  The absolute value of the  pitch is twice the vertical distance between successive sheets of the helicoid.  Any two choices of $C$ (and $O$)   produce  congruent 
helicoids in  $\sS^2\times\RR$ if and only  their pitches have the same absolute value.
{ In the rest of the paper, unless otherwise specified, we will assume that $\kappa > 0$.}

\addtocontents{toc}{\SkipTocEntry}
\subsection*{Helicoids in $\sS^2\times\RR$ are minimal surfaces} 
The helicoid $H_{C,\kappa}$ contains two {\em axes},  the vertical
lines $Z$ and $Z^*=\{O^*\}\times \RR$. Like helicoids in $\RR ^3$,
the helicoids $H_{C,\kappa}$ are fibred by horizontal geodesics,
about each of which  they are invariant under $180^\circ$ rotation, an isometry of $\sS^2\times\RR$. If $p$
is a point in the helicoid, then $180^\circ$ rotation of $\sS^2\times\RR$ about
the horizontal great circle in $H_{C,\kappa}$ through $p$
interchanges the two components of $\sS^2\times\RR\setminus H_{C,\kappa} $
but leaves $p$ fixed. If that symmetry is denoted by $\rho$, then
we have $\rho H_{C,\kappa}=H_{C,\kappa}$. If $\vec{h}(p)$ is the
mean curvature of $H_{C,\kappa}$ at $p$, then $\rho
_*\vec{h}(p)=\vec{h}(\rho(p))=\vec{h}(p)$. But since $\rho$
interchanges the components of $\sS^2\times\RR\setminus H_{C,\kappa}$,  $\rho
_*\vec{h}(p)=-\vec{h}(p)$.  Hence  $\vec{h}(p)=0$. Since $p$ is
 an arbitrary point on the helicoid $H_{X,\kappa}$, it is a minimal surface.

\addtocontents{toc}{\SkipTocEntry}
\subsection*{Similarities with helicoids in $\RR^3$} 
The helicoids $H_{C,\kappa}$\,, like helicoids
in $\RR ^3$, are fibred by horizontal geodesics, about each of
which  they are invariant under $180^\circ$ rotation, an isometry
of $\sS^2\times\RR$.  Also, composition of reflection in $Z$ and
$180^\circ$ rotation about $\sigma_{2\pi t, \kappa t}(C)$, produces a $180^\circ$ rotation 
about the great circle $C^{'}\subset S^2\times
\{t\}$ that passes through $Z$ and $Z^*$ and is
orthogonal to $\sigma_{2\pi t,\kappa t}(C)$. These great
circles $C^{'}$ are orthogonal to $H_{C,\kappa}$ and the
rotational symmetries around them are called {\em normal
symmetries of $H_{C,\kappa}$}. In $\RR^3$ the horizontal lines
through the axis $Z$ that meet a helicoid orthogonally are also lines
of rotational symmetry.

 As $\kappa\rightarrow 0$, the helicoid  $H_{C,\kappa}$ converges  to a lamination of
$\sS^2\times\RR$, with leaves  $S^2\times\{t\}$, $t\in \RR$,  and
singular set $Z\cup Z^*$.  Rewriting \eqref{HCkappa} as  
\[
H_{C,\kappa}= \bigcup_{t\in \RR} \sigma_{2\pi t/\kappa ,t}(C),
\]
 it is easy to see that as $\kappa \to \infty$, 
 the helicoid $H_{C,\kappa}$ converges to the cylinder $C\times\RR$. (We
consider such cylinders to be degenerate helicoids.)
This is directly analogous to what happens in $\RR^3$, where the
convergence as $\kappa \to 0$ is to a lamination by
horizontal planes with the $z$-axis as singular set and 
the convergence as $\kappa\to \infty$ is to a vertical plane.

\addtocontents{toc}{\SkipTocEntry}
\subsection*{Differences between helicoids in $\sS^2\times\RR$ and  helicoids in $\RR^3$} 
There are some important differences. First,
as we have seen, helicoids in $\sS^2\times\RR$ have {\em two} axes, $Z$ and $Z^*$,
not just one. Second, any two helicoids in $\RR ^3$ with positive pitch differ by a
rigid motion followed by a homothety. But in $\sS^2\times\RR$ there are no
homotheties; helicoids with different pitches are essentially
different surfaces. Third, helicoids in $\sS^2\times\RR$ have a symmetry of
reflection not possessed by helicoids in $\RR^3$.  Let $E \subset
\sS^2\times \{0\} $ be the great circle of points equidistant from
$O$ and $O^*$. (Thus $E$ is the equator with poles $O$ and $O^*$.) Then
it is easy to see that any  helicoid with axes $Z$ and $Z^*$ is invariant under the
reflection $\mu_E$ in the flat totally geodesic cylinder $E \times \RR$.
Finally, note that a helicoid in $\sS^2\times\RR$  is topologically an annulus
and conformally a  twice
punctured sphere, while a helicoid in $\RR ^3$ is simply
connected and conformally a once-punctured sphere.

\addtocontents{toc}{\SkipTocEntry}
\subsection*{Normalization and notation}
As noted above, if  $C$ and $C^{'}$ are two great circles in $S^2\times\{0\}$, it
easy to see that $H_{C,\kappa}$ and $H_{C^{'},\kappa}$ differ by
an ambient isometry of $\sS^2\times\RR$. We will fix our choice of base circles
for the remainder of this paper: Let X and Y be two great circles
in $S^2\times\{0\}$ that meet orthogonally at points $O$ and
$O^*$, and let $Z$ and $Z^{'}$ be the vertical geodesics through
$O$ and $O^*$, respectively. Then
\[
H_{\kappa}=H_{X,\kappa}
\]
is the helicoid with pitch $\kappa$ and axes $Z\cup Z^*$ that contains $X$. It is
invariant under rotation $\rho _Y$ around the geodesic $Y$.
It is also invariant under reflection in the cylinder $E\times\RR$.
Unless stated otherwise, a helicoid in this paper will be an $H_\kappa$ 

\addtocontents{toc}{\SkipTocEntry}
\subsection*{Symmetries and scaling}
The symmetries of the helicoid $H_\kappa$ in $\sS^2\times\RR$  are
generated by: 
\begin{enumerate} 
\item 
The  screw motions
$\sigma_{2\pi t,\kappa t}$\,; \item $\rho _X$, $\rho _Y$, and
$\rho _Z=\rho_{Z^*}$, the $180^\circ$ rotations about the geodesics
$X$, $Y$, and $Z$; 
\item 
$\mu _E$, reflection in the geodesic
cylinder $E\times \RR$.
\end{enumerate}
When restricted to  $H_{\kappa}$, the symmetries $\sigma_{2\pi t, \kappa t}$ and 
 $\rho_Y$ are orientation-preserving,  while $\rho_X$ and $\rho_Z$ 
 are orientation-reversing.

Helicoids in $\sS^2\times\RR$ whose pitches differ in absolute value
are neither congruent nor related by a homothety. (There are no
homotheties in $\sS^2\times\RR$.)
However, if $H$ is a helicoid of pitch $\kappa$ in $\sS^2(r)\times\RR$,
then dilating $H$ by $\lambda$ produces a helicoid of pitch $\lambda\kappa$
in $\sS^2(\lambda r)\times\RR$.

\section{The Main Theorems}\label{section:main-theorems}

 We now state our first main result in a form that
includes the periodic case ($h<\infty$) and the nonperiodic case ($h=\infty$.)
The reader may wish initially to ignore the periodic case.
Here $X$ and $Y$ are horizontal great circles at height $z=0$ that intersect
each other orthogonally at points $O$ and $O^*$, and $Z$ and $Z^*$ are the
vertical lines passing through $O$ and $O^*$.  

\begin{THEOREM}\label{theorem1}
Let $H$ be a helicoid in $\sS^2\times \RR$ that has vertical axes $Z\cup Z^*$
and that contains the horizontal great circle $X$.
For each genus $g\ge 1$ and each height $h\in (0,\infty]$, there
exists a pair $M_{+}$ and $M_{-}$ of embedded
minimal surfaces in $\sS^2\times \RR$ of genus $g$ with the following
properties (where $s\in \{+, -\}$):
\begin{enumerate}[\upshape (1)]
\item If $h=\infty$, then $M_s$ has no boundary, it is properly embedded in $\sS^2\times\RR$, 
  and each of its two ends
  is asymptotic to $H$ or to a vertical translate of $H$.
\item\label{1:boundary-circles} If $h<\infty$, then $M_s$ is a smooth, compact 
  surface-with-boundary in $\sS^2\times[-h,h]$.  Its boundary consists of the 
  two great circles at heights $h$  and $-h$ that intersect $H$ orthogonally at points in $Z$ and in $Z^*$.
\item If $h=\infty$, then
     \[  
      M_s\cap H = Z\cup Z^*\cup X.
     \]
    If $h<\infty$, then 
    \[
       {\rm interior}(M_s)\cap H = Z_h \cup Z_h^* \cup X,
    \]
    where $Z_h$, and $Z^*_h$
   are the portions of $Z$ and $Z^*$ with $|z|<h$.
\item $M_s$ is a $Y$-surface. 
\item $M_s \cap Y$ contains exactly $2g+2$ points.
\item\label{1:sign}
   $M_{+}$ and $M_{-}$ are positive and negative, 
   respectively, with respect to $H$ at $O$.
 \item
If $g$ is odd, then $M_{+}$ and $M_{-}$ 
are congruent to each other by reflection $\mu_E$ in the cylinder $E\times \RR$.
They are not congruent to each other by any orientation-preserving isometry of $\sS^2\times\RR$.
\item\label{even-genus-congruence-assertion}
 If $g$ is even, then
 $M_{+}$ and $M_{-}$ are each invariant under reflection $\mu_E$  in the cylinder 
   $E\times\RR$.  They are not congruent to each other by any isometry of 
   $\sS^2\times\RR$.
\end{enumerate}
\end{THEOREM}

The positive/negative terminology in assertion~\eqref{1:sign} is explained
in section~\ref{sign-section}, and $Y$-surfaces are defined and discussed
in section~\ref{section:$Y$-surfaces}.

Note that if $h<\infty$, we can extend $M_s$
 by repeated Schwarz reflections to
get a complete, properly embedded minimal surface $\widehat{M}_s$ that
is invariant under the screw motion $\sigma$ that takes $H$ to $H$ (preserving its orientation) and $\{z=0\}$
to $\{z=2h\}$.   The intersection $\widehat{M}_s\cap H$ consists of $Z$, $Z^*$,
and the horizontal circles $H\cap \{z=2nh\}$, $n\in \ZZ$. 
The surfaces $\widehat{M}_s$ are the periodic genus-$g$ helicoids mentioned
in the introduction.

\begin{remark}\label{other-circles-remark}
Assertion~\eqref{1:boundary-circles} states (for $h<\infty$) 
that the boundary $\partial M_s$ consists of two great circles that meet
$H$ orthogonally.   Actually, we could allow $\partial M_s$ to be any $\rho_Y$-invariant pair of great circles
at heights $h$ and $-h$ that intersect $Z$ and $Z^*$.
We have chosen to state theorem~\ref{theorem1} for circles that meet
the helicoid $H$ orthogonally because when we extend $M_s$ by
repeated Schwarz reflection to get a complete, properly embedded surface $\widehat M$, that choice
makes the intersection set $\widehat M\cap H$ particularly simple.
In section~\ref{adjusting-pitch-section}, we explain why the choice does not matter: 
if the theorem is true for one choice, it is also true for any other choice.  
Indeed, in {\em proving} the $h<\infty$ case of theorem~\ref{theorem1}, it will be more convenient
to let $\partial M_s$ be the great circles $H\cap \{z=\pm h\}$ that lie in $H$.
(Later, when we let $h\to \infty$ to get nonperiodic genus-$g$ helicoids in $\sS^2\times\RR$, the
choice of great circles $\partial M_s$ plays no role in the proofs.)  
\end{remark}

\begin{remark}\label{non-round-remark}
Theorem~\ref{theorem1} remains true if the round metric on $\sS^2$ is replaced
by any metric that has positive curvature, that is rotationally symmetric about the poles $O$ and $O^*$,
and that is symmetric with respect to reflection in the equator of points equidistant from $O$ and $O^*$.
(In fact the last symmetry is required only for the assertions about $\mu_E$ symmetry.)
No changes are required in any of the proofs. 
\end{remark}

In the nonperiodic case ($h=\infty$) of theorem~\ref{theorem1}, 
we do not know whether the two ends of $M_s$ are asymptotic to
opposite ends of the same helicoid. Indeed, it is possible that the top end
is asymptotic to a $H$ shifted vertically by some amount $v\ne 0$; the bottom end would then be asymptotic to $H$ shifted vertically by $-v$.
 Also, we do not know whether
 $M_{+}$ and $M_{-}$ must be asymptotic to each other, or
 to what extent the pair $\{M_{+}, M_{-}\}$
is unique.

Except for the noncongruence assertions, the proof of theorem~\ref{theorem1}
 holds  for all helicoids $H$ including
    $H=X\times\RR$, which may be regarded as a helicoid of infinite pitch.
(When $H=X\times \RR$ and $h=\infty$, theorem~\ref{theorem1} was  proved
 by Rosenberg in section~4 of~\cite{rosenberg2002}  by completely different methods.)  
 When $X=H\times\RR$, the noncongruence
 assertions break down: see appendix~\ref{noncongruence-appendix}. 
The periodic (i.e., $h<\infty$) case
of theorem~\ref{theorem1} is proved at the end of section~\ref{construction-outline-section},
assuming theorem~\ref{special-existence-theorem}, whose proof is a consequence of the
material in subsequent sections.  
The nonperiodic ($h=\infty$) case is proved in section~\ref{nonperiodic-section}.

Our second main result  lets us take limits as the radius of
the $\sS^2$ tends to infinity. For simplicity we only deal with the
nonperiodic case  ($h=\infty$) here.\footnote{An analogous theorem is
true for the periodic case ($h<\infty$).}

\begin{THEOREM}\label{theorem2}
Let $R_n$ be a sequence of radii tending to
infinity.   For each $n$, let $M_{+}(R_n)$ and $M_{-}(R_n)$ be 
genus-$g$ surfaces in $\sS^2(R_n)\times \RR$ satisfying the list
of properties in theorem~\ref{theorem1}, where $H$ is the helicoid of pitch $1$
and $h=\infty$.
 Then, after passing to
a subsequence, the  $M_+(R_n)$ and $M_{-}(R_n)$ converge smoothly on
compact sets to limits $M_+$ and $M_{-}$ with the
following properties:
\begin{enumerate}
\item\label{2:helicoidlike} $M_{+}$ and $M_{-}$ are complete, properly
embedded minimal surfaces in $\RR^3$ that are asymptotic to the
standard helicoid $H\subset\RR^3$. 
\item\label{2:intersection-property} If $M_s \ne H$, then $M_s\cap H=X\cup Z$ and
  $M_s$ has sign $s$ at $O$ with respect to $H$.
\item\label{2:Y-surfacenish}  $M_s$ is a $Y$-surface.
\item\label{2:point-count} 
  $\|M_s\cap Y\| = 2\, \|M_s\cap Y^+\| +1 = 2\,\genus(M_s)+1$.
 \item\label{2:genus-bound} If $g$ is even, then
      $M_{+}$ and $M_{-}$ each have genus at most $g/2$.
      If $g$ is odd, then $\genus(M_{+})+\genus(M_{-})$
       is at most $g$.
 \item\label{2:genus-parity} The genus of $M_{+}$ is even. 
          The genus of $M_{-}$ is odd.
\end{enumerate}
\end{THEOREM}
Here if $A$ is a set, then $\|A\|$ denotes the number of elements of $A$.

Theorem~\ref{theorem2} is proved in section~\ref{Proof_of_theorem_2}.

As mentioned earlier, theorem~\ref{theorem2} gives a new proof of the existence of genus-one helicoids
in $\RR^3$:

\begin{corollary*}
If $g=1$ or $2$, then $M_{+}$ has genus $0$ and $M_{-}$ has genus $1$.
\end{corollary*}
The corollary follows immediately from statements~\eqref{2:genus-bound}
and~\eqref{2:genus-parity} of theorem~\ref{theorem2}.

In the companion paper~\cite{hoffman-traizet-white-2}, we prove existence of helicoidal surfaces
of arbitrary genus in $\RR^3$:

\begin{THEOREM}\label{theorem3}
Let $M_{+}$ and $M_{-}$ be the limit minimal surfaces in $\RR^3$
described in theorem~\ref{theorem2}, and suppose that $g$ is even.
\begin{enumerate}
\item If $g/2$ is even, then $M_{+}$ has genus $g/2$.
\item If $g/2$ is odd, then $M_{-}$ has genus $g/2$.
\end{enumerate}
\end{THEOREM}

The sign here is crucial: if $g/2$ is even, then $M_{-}$ has
genus strictly less than $g/2$, 
and if $g/2$ is odd, then $M_{+}$ has genus strictly less than $g/2$.
(These inequalities follow immediately from Statements~\eqref{2:genus-bound} 
and~\eqref{2:genus-parity} of theorem~\ref{theorem2}.)

\section{Positivity/Negativity of Surfaces at $O$}\label{sign-section}

In this section, we explain the positive/negative terminology
used in theorem~\ref{theorem1}. 
Let $H$ be a helicoid that has axes $Z\cup Z^*$ and that contains $X$.
 The set 
 \[
     H\setminus (X \cup Z\cup Z^*)
\]
consists of four
components that we will call quadrants. The axes $Z$ and $Z^*$ are
naturally oriented, and we choose an orientation of $X$ allowing
us to label the components of $X\setminus\{O,O^*\}$ as $X^+$ and
$X^-$.   
We will refer to the  quadrant bounded by $X^+$, $Z^+$ and $(Z^*)^+$ and the
quadrant bounded by $X^-$, $Z^-$, and $(Z^*)^-$ as
the {\em positive quadrants} of $H$.  
The other two quadrants are called the {\em negative quadrants}.
 We orient $Y$ so
that the triple $(X,Y,Z)$ is positively oriented at $O$, and let
 $H^+$ denote the the component of the complement of $H$ that contains
$Y^+$. 

Consider an embedded minimal surface $S$  in $\sS^2\times\RR$ such that
in some open set $U$ containing $O$,
\begin{equation}\label{S-sign-definition}
     (\partial S)\cap U = (X\cup Z) \cap U.
\end{equation}
If $S$ and the two positive quadrants of $H\setminus (X\cup Z)$ are tangent to each other at $O$,
 we say that $S$ is {\em positive} at $O$.
If $S$ and the two negative quadrants of $H\setminus(X\cup Z)$ are tangent 
to each other at $O$, we say that $S$ is {\em negative} at $O$.
(Otherwise the sign of $S$ at $O$ with respect to $H$ is not defined.)

Now consider an embedded minimal surface $M$ in $\sS^2\times\RR$ such that 
\begin{equation}\label{M-sign-definition}
\begin{aligned}
   &\text{The origin $O$ is an interior point of $M$, and} \\
   &\text{$M\cap H$ coincides with $X\cup Z$ in some neighborhood of $O$.}
\end{aligned}
\end{equation}
We say that $M$ is positive or negative at $O$ with respect to $H$ according
to whether $M\cap H^+$ is positive or negative at $O$.

Positivity and negativity at $O^*$ is defined in exactly the same way.

\begin{remark}
A surface $S$ satisfying~\eqref{S-sign-definition} is positive (or negative) at $O$ 
if and only if $\mu_ES$ is positive (or negative)
at $O^*$, where $\mu_E$ denotes reflection in the totally geodesic cylinder
consisting of all points equidistant from $Z$ and $Z^*$.  
Similarly, a surface $M$ satisfying~\eqref{M-sign-definition} is positive (or negative) at $O$ with respect to $H$ if
and only $\mu_EM$ is positive (or negative) at $O^*$ with respect to $H$.
(If this is not clear, note that $\mu_E(H^+)=H^+$ and that $\mu_E(Q)=Q$ for each quadrant $Q$ of $H$.)
\end{remark}



\section{$Y$-surfaces}\label{section:$Y$-surfaces}

As discussed in the introduction, the surfaces we construct
will be $Y$-surfaces.  In this section, we define ``$Y$-surface"
and prove basic properties of $Y$-surfaces.

\begin{definition}\label{Y-Surface}
Suppose $N$ is a Riemannian $3$-manifold that admits an order-two rotation
 $\rho_Y$
about a geodesic $Y$.  An orientable surface $S$ in $N$ 
is called a {\em $Y$-surface}
if $\rho_Y$ restricts to an orientation-preserving isometry of $S$
and if 
\begin{equation}
 \label{-1homology}
  \mbox{ $\rho _Y$ acts  on $H_1(S,\ZZ)$ by
multiplication by $-1$.}
\end{equation}
\end{definition}

 The following proposition shows that the definition of a $Y$-surface is equivalent to two other topological conditions.
 
\begfig
 \hspace{4.0in}
  \vspace{.2in}
  \centerline{
\includegraphics[width=5.05in]{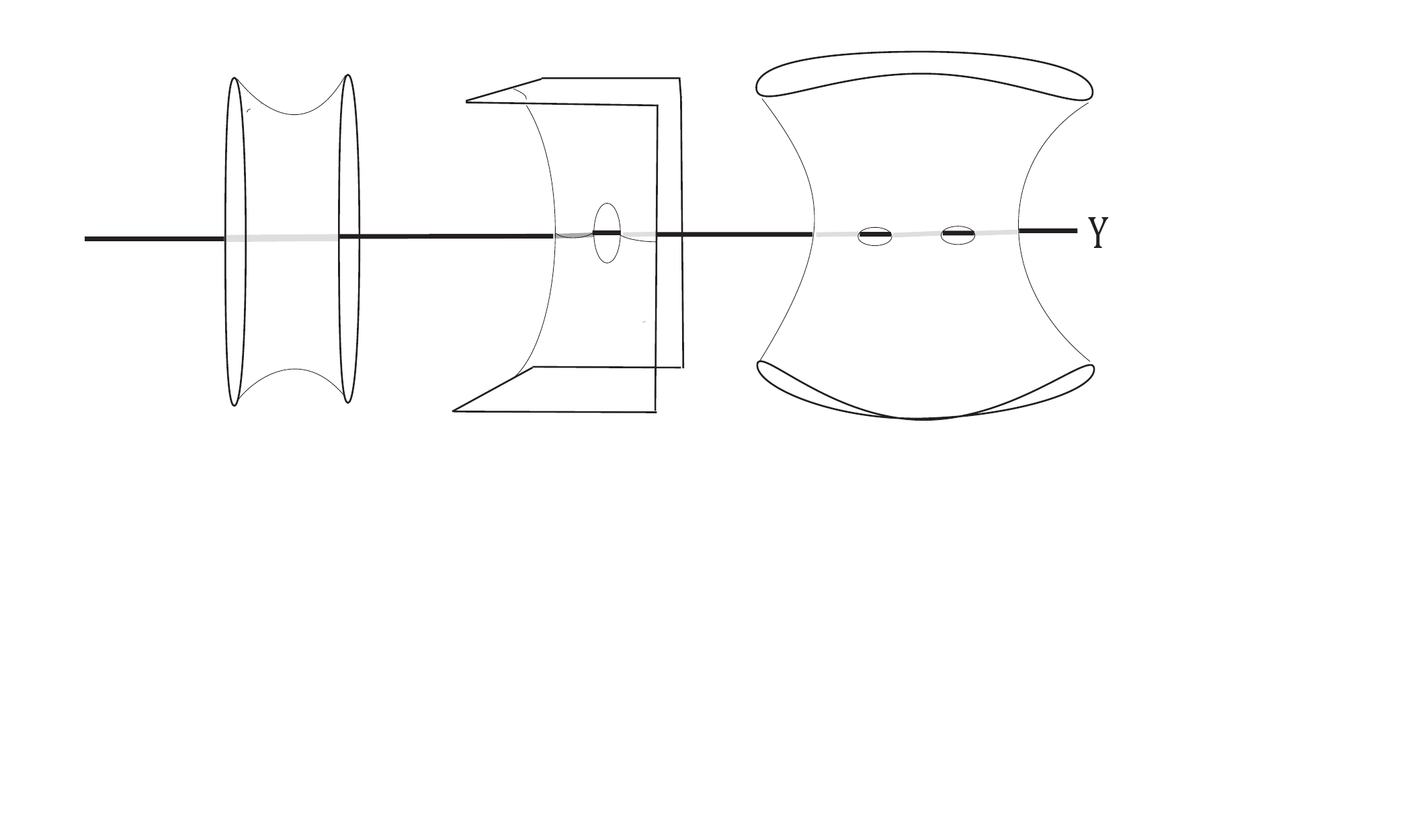}
   }
   \vspace{-2.0 in}
 \begin{center}
 \parbox{5.1in}{
\caption{
{\bf Right:} A $Y$-surface of genus two. The number of fixed points of $\rho _Y$ ($180$-degree rotation around $Y$) is even (equal to six)
and the number of  boundary components is two.
{\bf Center:} A $Y$-surface of genus one. The number of fixed points of $\rho _Y$ is odd (equal to three) and there is a single boundary component.
{\bf Left:} This annular surface $A$ is not a $Y$-surface. The rotation $\rho_Y$ acts  as the identity  on $H_1(A,\ZZ)$, not as multiplication by $-1$.
 }\label{Y-surface-figure} 
 }
 \end{center}
 \endfig

\newcommand{\Surf}{S}  

\begin{proposition}\label{Y-surface-topology-propostion}
Suppose that $\Surf$ is an open,  orientable Riemannian 
$2$-manifold of finite topology,
that $\rho: \Surf \to \Surf$ is an
 orientation-preserving
 isometry of order two, and that $\Surf/\rho$ is connected.
Then the following are equivalent:
\begin{enumerate}[\upshape(a)]
\item $\rho$ acts by multiplication by $-1$ on the first homology group 
    $H_1(\Surf,\ZZ)$.
\item\label{fixed-points-euler} the quotient $\Surf/\rho$ is topologically a disk. \item
$\Surf$ has exactly $2-\chi(\Surf)$ fixed points of $\rho$,
where $\chi(\Surf)$ is the Euler
  characteristic of $\Surf$.
\end{enumerate}
\end{proposition}

Note that Proposition~\ref{Y-surface-topology-propostion} is intrinsic in nature.
It does not require that the orientation-preserving automorphism $\rho$ be a reflection in an ambient geodesic $Y$.   
Proposition~\ref{Y-surface-topology-propostion} is easily proved 
using a $\rho$-invariant triangulation of $S$ whose vertices include
the fixed points of $\rho$; details may be found in \cite{hoffman-white-number}.

\begin{corollary}\label{$Y$-corollary}
Let $S$ be an open, orientable $Y$-surface such that $S/\rho_Y$ is connected.
Let $k$ be the number of fixed points of $\rho_Y:S\to S$.
\begin{enumerate}[\upshape(i)]
\item\label{$Y$-corollary-ends}
           The surface $\Surf$ has either one or two ends, according to
          whether $k$ is odd or even.
\item\label{$Y$-corollary-disks}
       If $k=0$, then $S$ is the union of two disks.
\item\label{$Y$-corollary-genus}
       If $k>0$, then $S$ is connected, and the genus of  $S$ is 
      $(k-2)/2$ if $k$ is even and $(k-1)/2$ if $k$ is odd. 
\end{enumerate}
\end{corollary}  

In particular, $S$ is a single disk if and only if $k=1$.

\begin{proof}[Proof of Corollary~\ref{$Y$-corollary}]
Since $S/\rho_Y$ is a disk, it has one end, and thus $S$ has either one
or two ends.  The Euler Characteristic of $S$ is $2c-2g-e$, where $c$
is the number of connected components, $g$ is the genus, and $e$
is the number of ends.  
Thus by Proposition~\ref{Y-surface-topology-propostion}\eqref{fixed-points-euler},
\begin{equation}\label{genus-fixed-points}
   2 - k = 2c - 2g - e.
\end{equation}
Hence $k$ and $e$ are congruent are congruent modulo $2$.  
Assertion~\eqref{$Y$-corollary-ends} follows immediately.  
(Figure~\ref{Y-surface-figure} shows two examples of assertion~\eqref{$Y$-corollary-ends}.)

Note that if $S$ has more than one component, then since $S/\rho_Y$
is a disk, in fact $S$ must have exactly two components, each of which
must be a disk.  Furthermore, $\rho_Y$ interchanges the two disks,
so that $\rho_Y$ has no fixed points in $S$, i.e., $k=0$.

Conversely, suppose $k=0$.  Then $e=2$ by Assertion~\eqref{$Y$-corollary-ends},
so from~\eqref{genus-fixed-points} we see that
\[
   2c = 2g + 4.
\]
Hence $2c\ge 4$ and therefore $c\ge 2$, i.e., $S$ has two or more components.
But we have just shown that in that case $S$ has exactly two components,
each of which is a disk. 
This completes the proof of Assertion~\eqref{$Y$-corollary-disks}.

Now suppose that $k>0$.  Then as we have just shown, $S$ is connected,
so~\eqref{genus-fixed-points} becomes $k=2g+e$, or
\begin{equation}\label{genus-vs-fixed-points}
    g =  \frac{k-e}2
\end{equation}
This together with Assertion~\eqref{$Y$-corollary-ends} gives
Assertion~\eqref{$Y$-corollary-genus}.
\end{proof}

\begin{remark}
To apply proposition~\ref{Y-surface-topology-propostion} and corollary~\ref{$Y$-corollary}
 to a compact manifold $M$ with non-empty boundary, one
lets $\Surf=M\setminus \partial M$.  The number of ends of $\Surf$ is equal to the
number of boundary components of $M$.
\end{remark}

\begin{proposition}\label{genus-0-proposition}
 If $S$ is a $Y$-surface in $N$ and if $U$ is an open subset of $S$
such that $U$ and $\rho_YU$ are disjoint, then $U$ has genus $0$.
\end{proposition}

\begin{proof}
Note that we can identify $U$ with a subset of $S/\rho_Y$.  Since $S$ is a $Y$-surface,
$S/\rho_Y$ has genus $0$ (by proposition~\ref{Y-surface-topology-propostion})
and therefore $U$ has genus $0$.
\end{proof}

\section{Periodic genus-$g$ helicoids in $\sS^2\times\RR$:
   theorem~\ref{theorem1} for $h<\infty$}\label{construction-outline-section}  

Let $0<h<\infty$.
Recall that we are trying to construct a minimal surface $M$
in $\sS^2\times [-h,h]$ such that
\[
  {\rm interior}(M) \cap H = Z_h \cup Z_h^* \cup X
\]
(where $Z_h$ and $Z_h^*$ are the portions of  $Z$ and $Z^*$ 
where $|z|< h$)
and such that $\partial M$ is a certain pair of circles at heights $h$
and $-h$. 
Since such an $M$ contains $Z_h$, $Z_h^*$, and $X$, it must
(by the Schwarz reflection principle)
be invariant under $\rho_Z$ (which is the same as $\rho_{Z^*}$)
and under $\rho_X$, the $180^\circ$ rotations about $Z$ and about $X$.
It follows that $M$ is invariant under $\rho_Y$, the composition of $\rho_Z$
and $\rho_X$.
In particular, if we let $S=\text{interior}(M)\cap H^+$ be the portion of 
the interior\footnote{It will be convenient for us to have $S$ be an open
manifold, because although $S$ is a smooth surface, its closure has has corners.}
 of $M$ in $H^+$,
then 
\[
    M = \overline{S \cup \rho_Z S} = \overline{S\cup \rho_X S}. 
\]
Thus to construct $M$, it suffices to construct $S$. 
 Note that the boundary
of $S$ is $Z_h \cup Z_h^* \cup X$ together with a great semicircle $C$ in $\overline{H^+}\cap\{z=h\}$
and its image $\rho_YC$ under $\rho_Y$.
Let us call that boundary $\Gamma_C$.
Thus we wish to construct embedded minimal surface $S$ in $H^+$ having
specified topology and having boundary $\partial S=\Gamma_C$.
Note we need $S$ to be $\rho_Y$-invariant; otherwise Schwarz reflection in $Z$
and Schwarz reflection in $X$ would not produce the same surface.

We will  prove existence by counting surfaces mod $2$. 
 Suppose for the moment
that the curve $\Gamma_C$ is nondegenerate in the following sense: if $S$ is a smooth
embedded, minimal, $Y$-surface in $H^+$ with boundary $\Gamma_C$, then $S$
has no nonzero $\rho_Y$-invariant jacobi fields that vanish on $\Gamma_C$.
For each $g\ge 0$, the number of such surfaces $S$ of genus $g$
turns out to be even.   Of course, for the purposes of proving existence, this fact is not
 useful, since $0$ is an even number.   However, if instead of considering
all $Y$-surfaces of genus $g$, we consider only that that are positive (or those that
are negative) at $O$, then the number of such surfaces turns out to be odd, and therefore
existence follows.

For the next few sections, we fix a helicoid $H$ and we fix an $h$ with $0<h<\infty$.
Our goal is to prove the following theorem:

\begfig
 \hspace{5.0in}
  \vspace{.2in}
  \centerline{
\includegraphics[width=4.05in]{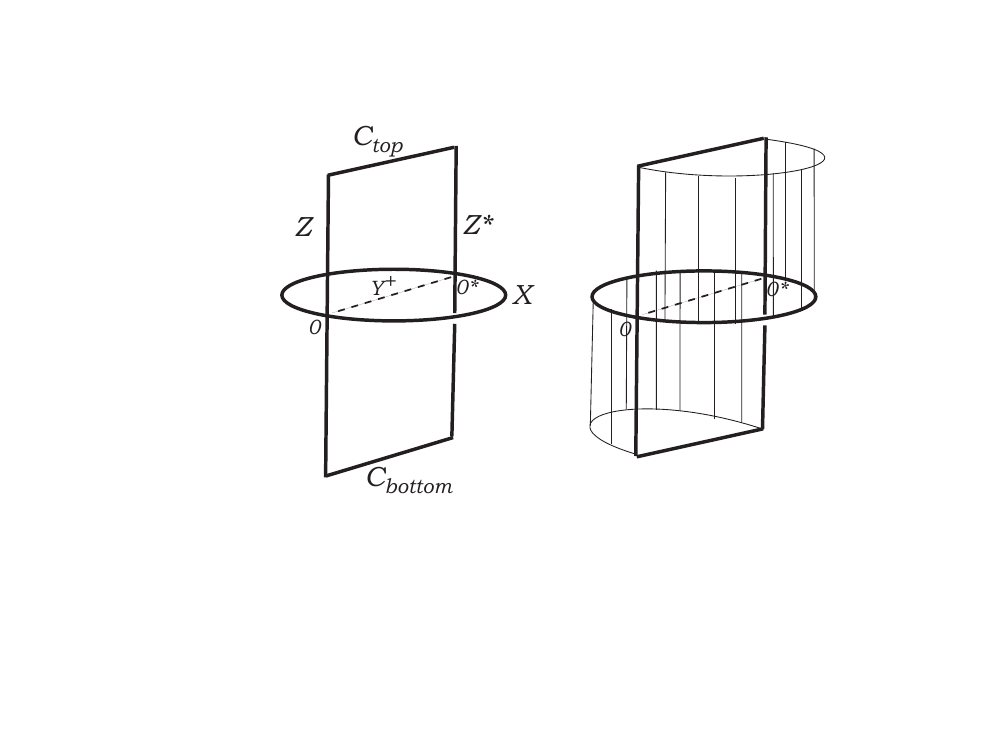}
   }
   \vspace{-1.3 in}
 \begin{center}
 \parbox{5.1in}{
\caption{\label{GammaFigure} {\bf The boundary curve $\Gamma_C$}.
We depict $\sS^2\times\RR$  in these illustrations as $\RR^3$ with  
each horizontal $\sS^2\times\{z\}$ represented as horizontal plane via stereographic projection, 
with one point of the sphere at infinity. Here, that point is the 
antipodal point of the midpoint of the semicircle $Y^+$.
{\bf Right:} For ease of illustration, we have 
chosen the reference helicoid $H$ to be  the vertical cylinder $X\times\RR$, and the semicircle 
$C=C_\text{top}$ to meet $H$ orthogonally.  The geodesics
$X$, $Z$ and $Z^*$ divide $H$ into four components, two of which are shaded. 
The helicoid $H$ divides $\sS^2\times\RR$ into two
components. The component $H^+$ is the interior of the solid cylinder bounded by $H$. 
{\bf Left:} The boundary curve $\Gamma=\Gamma _C$  consists of the great circle  $X$, 
two vertical line segments on the axes $Z\cup Z^*$ of height
$2h$ and two semicircles in $(\sS^2\times\{\pm h\})\cap H^+$. 
Note that $\Gamma$ has $\rho_Y$ symmetry. 
We seek a $\rho_Y$-invariant minimal surface in $H^+$  that has boundary $\Gamma_C$  
and has all of its topology
concentrated along $Y^+$. That is,   
we want a $Y$-surface as defined Section~\ref{section:$Y$-surfaces} with the properties    established in Proposition~\ref{Y-surface-topology-propostion}.  According to  theorem~\ref{special-existence-theorem}, there are in fact two such surfaces for every positive genus.
}
 }
 \end{center}
 \endfig

\begin{theorem}\label{special-existence-theorem}
Let $0<h<\infty$, let $C$ be a great semicircle in $\overline{H^+}\cap\{z=h\}$
joining $Z$ to $Z^*$, and let $\Gamma_C$ be the curve given by
\[
   \Gamma_C = Z_h\cup Z_h^*\cup C\cup \rho_YC
\]
where $Z_h=Z\cap \{|z|\le h\}$ and $Z^*=Z^*\cap\{|z|\le h\}$.

For each sign $s\in\{+,-\}$ and for each $n\ge 1$, there exists
an open, embedded 
minimal  $Y$-surface $S=S_s$ in $H^+\cap \{|z|<h\}$ such that $\partial S=\Gamma_C$,
such that $Y^+\cap S$ contains exactly $n$ points,
and such that $S$ is 
positive\footnote{Positivity and negativity of $S$ at $O$ were defined in section~\ref{sign-section}.}
 at $O$ if $s=+$ and negative at $O$
if $s=-$.

If $n$ is even, there is such a surface that is invariant under
reflection $\mu_E$ in the totally geodesic cylinder $E\times\RR$.
\end{theorem}

Before proving theorem~\ref{special-existence-theorem}, let us show that it implies the periodic case
of theorem~\ref{theorem1} of section~\ref{section:main-theorems}:

\begin{proposition}\label{reduction-to-H^+-proposition}
Theorem~\ref{special-existence-theorem} implies theorem~\ref{theorem1} in the 
periodic case $h<\infty$.
\end{proposition}

\begin{proof}
Let $C$ be the great semicircle in $\overline{H^+}\cap \{z=h\}$ that has endpoints on
$Z\cup Z^*$ and that meets $H$ orthogonally at those endpoints.
First suppose $n$ is even, and let $S_s$ for $s\in\{+,-\}$ be the surfaces
given by theorem~\ref{special-existence-theorem}.
Let $M_s$ be the surface obtained by Schwarz
reflection from $S_s$:
\[
   M_s = \overline{S_s \cup \rho_Z S_s} = \overline{S_s\cup \rho_X S_s}, 
\]
(The second equality holds because $S_s$ is $\rho_Y$-invariant and $\rho _Z\circ \rho_Y =\rho_X$.)

By lemma~\ref{Schwarz-regularity-lemma} below, 
  $M_s$ is a smoothly embedded minimal surface.  
  Clearly it is $\rho_Y$-invariant, it lies
   in  $\sS^2\times [-h,h]$,
its interior has the desired intersection with $H$, it has the indicated sign at $O$,
it has $\mu_E$ symmetry, 
and its boundary is the desired pair of horizontal circles.
We claim that $M_s$ is a $Y$-surface.  To see this, note that 
 since $S_s$
is a $Y$-surface, the quotient $S_s/\rho_Y$ is 
topologically a disk by proposition~\ref{Y-surface-topology-propostion}.
The interior of $M_s/\rho_Y$ is two copies of $S_s/\rho_Y$ glued along a common
boundary segment.  Thus the interior of $M_s/\rho_Y$ is also topologically a disk, and therefore
$M_s$ is a $Y$-surface by proposition~\ref{Y-surface-topology-propostion}.

Note that $M_s\cap Y$ has $2n+2$ points: the $n$ points in $S_s\cap Y^+$,
an equal number of points in $\rho_ZS_s\cap Y^-$, and the two
points $O$ and $O^*$.  Thus by corollary~\ref{$Y$-corollary}, 
$M_s$ has genus $n$.  Since $n$
is an arbitrary even number, this completes the proof for even genus, except
for assertion~\eqref{even-genus-congruence-assertion}, 
the assertion that $M_{+}$ and $M_{-}$ are not congruent.

Now let $n$ be odd, and let $S_{+}$ be the surface given by 
theorem~\ref{special-existence-theorem}.
By lemma~\ref{parity-sign-lemma} below, $S_{+}$ is negative at $O^*$,
which  implies that $\mu_E(S_{+})$ is 
negative at $O$.  In this case, we choose our $S_{-}$ to be $\mu_E(S_{+})$.
Exactly as when $n$ is even, we extend $S_{\pm }$ by Schwarz reflection
to get $M_{\pm }$.   As before, the $M_{\pm }$ are $Y$-surfaces
of genus $n$.  The proof that they have the required properties is exactly
as in the case of even $n$, except for the statement that $M_{+}$ and $M_{-}$
are not congruent by any orientation-preserving isometry of $\sS^2\times \RR$.

It remains only to prove the statements about noncongruence of $M_{+}$ and 
 $M_{-}$.  Those statements (which we never actually use) 
 are proved in appendix~\ref{noncongruence-appendix}.
\end{proof}

The proof above used the following two lemmas:

\begin{lemma}\label{Schwarz-regularity-lemma}
 If $S$ is a $\rho_Y$-invariant embedded minimal surface in $H^+$  with boundary
$\Gamma$ and with $Y\cap S$ a finite set,  then the Schwarz-extended surface
\[
      M = \overline{S \cup \rho_Z S} = \overline{S \cup \rho_X S}. 
\]
is smoothly embedded everywhere.
\end{lemma}

\begin{proof}
One easily checks that if $q$ is a corner of $\Gamma$ other than $O$ or $O^*$, 
then the tangent cone to $S$ at $q$ is a multiplicity-one quarter plane.  
 Thus the tangent
cone to $M$ at $q$ is a multiplicity-one halfplane, which implies that $M$ is smooth at $q$
by Allard's boundary regularity theorem.

Let $B$ be an open ball centered at $O$ small enough that $B$ contains no points of $Y\cap S$.
Now $S\cap B$ is a $Y$-surface, so by 
corollary~\ref{$Y$-corollary}\eqref{$Y$-corollary-disks}, it is topologically the union of two disks.  It follows
that $M\cap B$ is a disk, 
so $M$ is a branched minimal immersion at $O$ by~\cite{gulliver-removability}.
But since $M$ is embedded, in fact $M$ is unbranched.
\end{proof}

\begin{lemma}\label{parity-sign-lemma}
 Let $S\subset H^+$ be  a $Y$-surface with $\partial S=\Gamma$. 
 Then the signs of $S$ at $O$ and $O^*$ agree or disagree according to whether
 the number of points of $Y\cap S$ is even or odd.
\end{lemma}

\begin{proof}
Let $\widehat{S}$ be the geodesic completion of $S$.
We can identity $\widehat{S}$ with $\overline{S}=S\cup\partial S$, except
that $O\in \overline{S}$ corresponds to two points in $\widehat{S}$,
and similarly for $O^*$.
Note that the number of ends of $S$ is equal to the number of boundary components of 
 $\partial \widehat{S}$.

By symmetry, we may assume that the sign of $S$ at $O$ is $+$. 
Then at $O$, $Z^+$ is joined in $\partial \widehat{S}$ to $X^+$ and $Z^-$ is joined to $X^-$.
 If the sign of $S$ at $O^*$ is also $+$, then  the same pairing occurs at $O^*$, 
 from which it follows that 
 $\partial \widehat{S}$ has two components and therefore that $S$ has two ends. 
 If the sign of $S$ at $O^*$ is $-$, then  the pairings are crossed, so that
$\partial \widehat{S}$ has  only one component and therefore $S$ has only one end.
We have shown that $S$ has two ends or one end
according to whether the signs of $S$ at $O$ and $O^*$ are equal or not.
The lemma now follows from corollary~\ref{$Y$-corollary},
 according to which the number of ends of $S$
is two or one according to whether the number of points of $Y\cap S$ is even or odd.
\end{proof}






\section{Adjusting the pitch of the helicoid}\label{adjusting-pitch-section}

Theorem~\ref{special-existence-theorem}  of section~\ref{construction-outline-section} asserts that
the curve $\Gamma_C$ 
bounds various minimal surfaces in $H^+$.  In that theorem, $C=\Gamma\cap \{z=h\}$ is allowed
to be any semicircle in $\overline{H^+}\cap \{z=h\}$ with endpoints in $Z\cup Z^*$.  
In this section, we will show that in order to prove 
Theorem~\ref{special-existence-theorem}, it is sufficient to prove it for the special case where
 $C$ is a semicircle in the helicoid $H$.

\begin{theorem}\label{special-existence-theorem-tilted}[Special case of 
theorem~\ref{special-existence-theorem}]
Let $0<h<\infty$ and let $C$ be one of the two great semicircles in $H\cap\{z=h\}$
joining $Z$ to $Z^*$.
For each sign $s\in\{+,-\}$ and for each $n\ge 1$, there exists
an open embedded 
minimal  $Y$-surface $S=S_s$ in $H^+\cap \{|z|<h\}$ such that 
\[
   \partial S=\Gamma_C :=  Z_h \cup Z_h^* \cup C \cup \rho_YC,
\]
such that $Y^+\cap S$ contains exactly $n$ points,
and such that $S$ is positive at $O$ if $s=+$ and negative at $O$
if $s=-$.

If $n$ is even, there is such a surface that is invariant under
reflection $\mu_E$ in the totally geodesic cylinder $E\times\RR$.
\end{theorem}
 We will prove that theorem~\ref{special-existence-theorem-tilted}, a special case of theorem~\ref{special-existence-theorem}, is in fact equivalent to it:
 
 \begin{proposition} Theorem~\ref{special-existence-theorem-tilted} implies theorem~\ref{special-existence-theorem}. \label{tilted-suffices-proposition}
 \end{proposition}

 \begin{proof}[Proof of Proposition~\ref{tilted-suffices-proposition}]
Let $H$ be a helicoid and let $C$ be a great semicircle in
   $\overline{H^+}\cap\{z=h\}$.
We may assume that $C$ does not lie in $H$, as otherwise there is nothing to prove.
Therefore the interior of the semicircle $C$ lies in $H^+$.
Now increase (or decrease) the pitch of $H$  to get a one-parameter family of helicoids
$H(t)$ with $0\le t\le 1$ such that
\begin{enumerate}
\item $H(1)=H$,
\item $C \subset \overline{H(t)^+}$ for all $t\in [0,1]$,
\item $C \subset H(0)$.
\end{enumerate}

\begin{claim} Suppose $S$ is an open, $\rho_Y$-invariant, embedded minimal surface bounded by $\Gamma_C$
with $S\cap Y^+$ nonempty.
If $S$
 is contained in $H(0)^+$, then it is contained in $H(t)^+$ for all $t\in[0,1]$.
Furthermore, in that case
 the sign of $S$ at $O$ with respect to $H(t)$ does not depend on $t$.
\end{claim}

\begin{proof}[Proof of claim] 
Let $T$ be the set of $t\in [0,1]$ for which $S$ is contained in $\overline{H(t)^+}$. 
Clearly $T$ is a closed set.
We claim that $T$ is also open relative to $[0,1]$.
For suppose that $t\in T$, and thus that $S\subset \overline{H(t)^+}$.
Now $S$ is not contained in $H(t)$ since $S\cap Y^+$ is nonempty.
Thus by the strong maximum principle and the strong boundary maximum
principle, $S$ cannot touch $H(t)$, nor is $\overline{S}$ tangent
to $H(t)$ at any points of $\Gamma_C$ other than its corners.

At the corners $O$ and $O^*$, $S$ and $H(\tau)$ are tangent.  However, the curvatures of $H$ and
$
    M:=\overline{S\cup \rho_YS}
$
differ from each other\footnote{Recall that if two minimal surfaces in a $3$-manifold 
are tangent at a point, then the intersection set
near the point is like the zero set of a homogeneous harmonic polynomial.  In particular, it consists
of $(n+1)$ curves crossing through the point, where $n$ is the degree of contact of the two surfaces at the point.
 Near $O$, the intersection of $M$ and $H$ coincides with $X\cup Z$, so their order of contact  
at $O$ is  
exactly one.} 
 at $O$, and also at $O^*$. 
It follows readily that $t$ is in the interior of $T$ relative to $[0,1]$.  
Since $T$ is open and closed in $[0,1]$, either $=\emptyset$ or $T=[0,1]$.
This proves the first assertion of the claim.  The second follows by continuity.
\end{proof}

By the claim, if theorem~\ref{special-existence-theorem-tilted} 
is true for $\Gamma_C$ and $H(0)$, then theorem~\ref{special-existence-theorem}
 is true for $H=H(1)$ and $\Gamma_C$. 
 This completes the proof of proposition~\ref{tilted-suffices-proposition}.
 \end{proof}


\section{Eliminating jacobi fields by perturbing the metric}\label{bumpy-section}

Our proof involves counting minimal surfaces mod $2$.  Minimal surfaces
with nontrivial jacobi fields tend to throw off such counts.  (A nontrivial jacobi field
is a nonzero normal jacobi field that vanishes on the boundary.)
Fortunately, if we fix a curve $\Gamma$ in a $3$-manifold, then 
a generic Riemannian metric on the $3$-manifold
 will be  ``bumpy" (with respect to $\Gamma$)
 in the following sense: $\Gamma$ will not bound any 
minimal surfaces with
nontrivial jacobi fields.  
 Thus instead of working with the 
standard product metric on
$\sS^2\times\RR$, we will use
 a slightly perturbed bumpy metric and prove
 theorem~\ref{special-existence-theorem-tilted}
for that perturbed metric.  By taking a limit of surfaces as the perturbation
goes to $0$, we get the surfaces whose existence is asserted in
 theorem~\ref{special-existence-theorem-tilted} for the standard metric.
In this section, we explain how to perturb the metric to make
it bumpy, and how to take the limit as the perturbation goes to $0$.

In what class of metrics should we make our perturbations?
The metrics should have $\rho_X$ and $\rho_Z$ symmetry so that we can do Schwarz reflection,
$\rho_Y$ symmetry so that the notion of $Y$-surface makes sense,
and $\mu_E$-symmetry so that the conclusion of theorem~\ref{special-existence-theorem-tilted} makes sense.
It is convenient to use metrics for which the helicoid $H$ and the spheres $\{z=\pm h\}$ are minimal, 
because we will need  
the region $N=\overline{H^+}\cap \{|z|\le h\}$ to be weakly mean-convex.
We will also need to have an isoperimetric inequality hold for minimal surfaces in $N$,
which is equivalent (see remark~\ref{isoperimetric-equivalence-remark}) 
to the nonexistence of any smooth, closed minimal surfaces
in $N$.  Finally, at one point (see the last sentence in section~\ref{smooth-count-section}) 
we will need the two bounded components of $H\setminus\Gamma$
to be strictly stable, so we restrict ourselves to metrics for which they are strictly stable.

The following theorem (together with its corollary) is theorem~\ref{special-existence-theorem-tilted}
with the standard metric on $\sS^2\times\RR$ replaced by 
a suitably bumpy metric in the class of metrics described above,
and with the conclusion strengthened to say that $\Gamma_C$
bounds an odd number of surfaces with the desired properties:

\begin{theorem}\label{main-bumpy-theorem}
Let $H$ be a helicoid in $\sS^2\times\RR$, let $0<h<\infty$,
and let $\Gamma=\Gamma_C$ be the curve in theorem~\ref{special-existence-theorem-tilted}:
\[
 \Gamma= Z_h \cup Z_h^* \cup C \cup \rho_YC,
\]
where $C$ is one of the semicircles in $H\cap\{z=h\}$ joining $Z$ to $Z^*$.
Let $G$ be the group of isometries of $\sS^2\times\RR$ generated
by $\rho_X$, $\rho_Y$, $\rho_Z=\rho_{Z^*}$, and $\mu_E$.
 Let $\gamma$ be a smooth, $G$-invariant Riemannian metric on $\sS^2\times\RR$
such that 
\begin{enumerate}
\item\label{minimal-anchors-hypothesis}
        the helicoid $H$ and the horizontal spheres $\{z=\pm h\}$ are $\gamma$-minimal surfaces.
\item\label{strictly-stable-hypothesis} the two bounded components of $H\setminus \Gamma$ are strictly stable
     (as $\gamma$-minimal surfaces).
\item\label{no-closed-minimal-surface-hypothesis} 
    the region $N:=\overline{H^+}\cap \{|z|\le h\}$ contains no smooth, closed, embedded $\gamma$-minimal surface, 
\item\label{bumpy-hypothesis}
    the curve $\Gamma$ does not bound any embedded $\gamma$-minimal $Y$-surfaces
   in $\overline{H^+}\cap \{|z|\le h\}$ with nontrivial $\rho_Y$-invariant jacobi fields.
\end{enumerate}
For each nonnegative
 integer $n$ and each sign $s\in \{+,-\}$, let
\[
    \MM(\Gamma,n,s) = \MM_\gamma(\Gamma,n,s)
\]
denote the set of 
 embedded,  $\gamma$-minimal $Y$-surfaces $S$ in 
  $\overline{H^+}\cap\{|z|\le h\}$ bounded by $\Gamma$ such that $S\cap Y^+$ has exactly $n$
  points and such that $S$ has sign $s$ at $O$.
Then the number of surfaces in $\MM(\Gamma,n,s)$ is odd.
\end{theorem}

\begin{corollary}\label{extra-symmetry-corollary}
Under the hypotheses of the theorem, if $n$ is even, then the number of 
$\mu_E$-invariant surfaces in $\MM(\Gamma,n,s)$ is odd.
\end{corollary}

\begin{proof}[Proof of corollary]
Let $n$ be even.  By lemma~\ref{parity-sign-lemma},
 if $S\in \MM(\Gamma,n,s)$, then $S$ also has sign $s$ at $O^*$,
from which it follows that $\mu_E(S)\in \MM(\Gamma,n,s)$.
Thus the number of non-$\mu_E$-invariant surfaces in $\MM(\Gamma,n,s)$ is even because
such surfaces come in pairs ($S$ being paired with $\mu_ES$.)  
By the theorem, the total number of surfaces in $\MM(\Gamma,n,s)$ is odd, so therefore
the number of $\mu_E$-invariant surfaces must also be odd.
\end{proof}

\begin{remark}\label{isoperimetric-equivalence-remark}
Hypothesis~\eqref{minimal-anchors-hypothesis} of theorem~\ref{main-bumpy-theorem}
 implies that the compact region $N:=\overline{H^+}\cap\{|z|\le h\}$ is $\gamma$-mean-convex.
It follows (see~\cite{white-isoperimetric}*{\S2.1 and \S5}) that 
condition~\eqref{no-closed-minimal-surface-hypothesis}
is equivalent to the following condition:
\begin{enumerate}
\item[(3$'$)] There is a finite constant $c$ such that $\area(\Sigma)\le c \length(\partial \Sigma)$
                for every $\gamma$-minimal surface $\Sigma$ in $N$.
\end{enumerate}
Furthermore, the proof of theorem~2.3 in~\cite{white-isoperimetric}
 shows that for any compact set $N$, the set of Riemannian metrics  
 satisfying~\thetag{3$'$} 
is open, with a constant $c=c_{\gamma}$ that depends 
upper-semicontinuously on the 
metric\footnote{As explained in \cite{white-isoperimetric}, for any metric $\gamma$,
we can let $c_\gamma$ be the supremum (possibly infinite)
of $|V|/|\delta V|$ among all $2$-dimensional varifolds $V$ in $N$ with $|\delta V|<\infty$,
 where $|V|$ is the mass of $V$ and $|\delta V|$
is its total first variation measure.
 The supremum is attained by a varifold $V_\gamma$
with mass $|V_\gamma|=1$.  Suppose $\gamma(i)\to\gamma$.  By passing to a subsequence, we may
assume that the $V_{\gamma(i)}$ converge weakly to a varifold $V$.
Under weak convergence, mass is continuous and total first variation measure is lower
semicontinuous.  Thus 
\[
   c_\gamma \ge \frac{ |V| }{|\delta V|} 
   \ge \limsup \frac{|V_{\gamma(i)}|}{|\delta V_{\gamma(i)}|}
   = \limsup c_{\gamma(i)}.
\]
This proves that the map $\gamma\mapsto c_\gamma\in (0,\infty]$ is uppersemicontinuous,
and therefore also that the set of metrics $\gamma$ for which $c_\gamma<\infty$ is 
an open set.  (The compactness, continuity, and lower-semicontinuity results used here
are easy and standard, and are explained in the appendix to~\cite{white-isoperimetric}.
See in particular~\cite{white-isoperimetric}*{\S7.5}.)
}.   
\end{remark}

\begin{proposition}\label{bumpy-suffices-proposition}
Suppose theorem~\ref{main-bumpy-theorem} is true.  
Then theorem~\ref{special-existence-theorem-tilted} is true.
\end{proposition}

\newcommand{\Gbig}{\mathcal{G}_1}
\newcommand{\Gmedium}{\widehat{\mathcal{G}}}
\newcommand{\Gsmall}{\mathcal{G}}

\begin{proof}
Let $\Gbig$ be the space of all smooth, $G$-invariant Riemannian
metrics $\gamma$ on $\sS^2\times\RR$
that satisfy hypothesis~\eqref{minimal-anchors-hypothesis} of 
the theorem.  
Let $\Gmedium$ be the subset consisting of those metrics $\gamma\in \Gbig$ such that
also satisfy hypotheses~\eqref{strictly-stable-hypothesis} 
and~\eqref{no-closed-minimal-surface-hypothesis} 
of  theorem~\ref{main-bumpy-theorem}, and let $\Gsmall$ be the set of metrics 
that satisfy all the hypotheses of the theorem.

We claim that the standard product metric $\gamma$ belongs to $\Gmedium$. 
Clearly it is $G$-invariant and satisfies hypothesis~\eqref{minimal-anchors-hypothesis}.
Note that each bounded component of $H\setminus \Gamma$ is strictly stable, because it is contained
in one of the half-helicoidal components of $H\setminus(Z\cup Z^*)$ and those half-helicoids
are stable (vertical translation induces  a positive jacobi field). 
Thus $\gamma$ satisfies the strict stability
hypothesis~\eqref{strictly-stable-hypothesis}.  
It also satisfies hypothesis~\eqref{no-closed-minimal-surface-hypothesis}
because if $\Sigma$ were a closed minimal surface in $N$, then the height function $z$ would attain
a maximum value, say $a$, on $\Sigma$, which implies by the strong maximum principle that
 the sphere $\{z=a\}$ would be contained in $\Sigma$, contradicting the fact that 
    $\Sigma\subset N\subset \overline{H^+}$.
 This completes the proof that the standard product metric $\gamma$ belongs to $\Gmedium$.
 
By lemma~\ref{bumpy-lemma} below, a generic metric in $\Gbig$ satisfies the bumpiness
hypothesis~\eqref{bumpy-hypothesis} of theorem~\ref{main-bumpy-theorem}.
Since $\Gmedium$ is an open subset of $\Gbig$ (see  remark~\ref{isoperimetric-equivalence-remark}),
it follows that a generic metric in $\Gmedium$ satisfies the bumpiness hypothesis.  
In particular, this means that $\Gsmall$ is a dense subset of $\Gmedium$.

Since the standard metric $\gamma$ is in $\Gmedium$, 
there is a sequence $\gamma_i$ of metrics in $\Gsmall$
that converge smoothly to $\gamma$. 
Fix a nonnegative integer $n$ and a sign $s$.
By theorem~\ref{main-bumpy-theorem},  $\MM_{\gamma_i}(\Gamma,n,s)$
contains at least one surface $S_i$.  If $n$ is even, we choose $S_i$ to
be $\mu_E$-invariant, which is possible by corollary~\ref{extra-symmetry-corollary}.

By remark~\ref{isoperimetric-equivalence-remark},
\begin{equation}\label{isoperimetric-ratio}
  \limsup_i \frac{\area_{\gamma_i}(S_i)}{\length_{\gamma_i}(\partial S_i)} \le c_\gamma
\end{equation}
where $c_\gamma$ is the constant 
 in remark~\ref{isoperimetric-equivalence-remark}
for the standard product metric $\gamma$.   Since
\[
   \length_{\gamma_i}(\partial S_i) = \length_{\gamma_i}(\Gamma)\to \length_\gamma(\Gamma)<\infty,
\]
we see from~\eqref{isoperimetric-ratio} that the areas of the $S_i$ are uniformly bounded.


Let
\[
  M_i = \overline{S_i\cup \rho_Z S_i}.
\]
be obtained from $S_i$ by Schwarz reflection.
Of course the areas of the $M_i$ are also uniformly bounded.
Using  the Gauss-Bonnet theorem, the minimality of the $M_i$, and the fact that the sectional curvatures 
of $\sS^2\times\RR$ are bounded, it follows that 
\begin{equation}\label{total-curvature-bound}
  \sup_i \int_{M_i} \beta(M_i,\cdot)\,dA < \infty ,
\end{equation}
where $\beta(M_i,p)$ is the square of the norm of the 
second fundamental  form of $M_i$ at the point $p$.

The total curvature bound~\eqref{total-curvature-bound} implies (see~\cite{white-curvature-estimates}*{theorem~3})
that after passing to a further subsequence, the $M_i$ converge 
smoothly to an embedded minimal surface $M$, which implies that the $S_i$ converge uniformly smoothly to 
a surface $S$ in $N$ with $\partial S=\Gamma$ and with $M=\overline{S\cup\rho_Y S}$.
The smooth convergence $M_i\to M$ implies
that $S\in \MM_\gamma(\Gamma,n,s)$, where $\gamma$ is the standard product
metric.  Furthermore, if $n$ is even, then $S$ is $\mu_E$-invariant.
This completes the proof of theorem~\ref{special-existence-theorem-tilted} (assuming 
theorem~\ref{main-bumpy-theorem}).
\end{proof}

\begin{lemma}\label{bumpy-lemma}
Let $\Gbig$ be the set of smooth, $G$-invariant metrics $\gamma$ on $\sS^2\times \RR$ such
that the helicoid $H$ and the spheres $\{z=\pm h\}$ are $\gamma$-minimal.
For a generic metric $\gamma$ in $\Gbig$, the curve $\Gamma$ bounds no 
embedded, $\rho_Y$-invariant, $\gamma$-minimal surfaces with nontrivial $\rho_Y$-invariant
jacobi fields.
\end{lemma}

\begin{proof}
By the bumpy metrics theorem~\cite{white-bumpy}, a generic metric $\gamma$ in $\Gbig$ has the property
\begin{enumerate}
\item[(*)]\label{circles-bumpy-item}
The pair of circles $H\cap\{z=\pm h\}$ bounds no embedded $\gamma$-minimal
 surface in $H\cap\{|z|\le h\}$ with a nontrivial jacobi field.
\end{enumerate}
Thus it suffices to prove that if $\gamma$ has the property~\thetag{*},
and if $S\subset N$ is an embedded, $\rho_Y$-invariant, $\gamma$-minimal
surface with boundary $\Gamma$, then $S$ has no nontrivial $\rho_Y$-invariant jacobi field.

Suppose to the contrary that $S$ had such a nontrivial jacobi field $v$.  Then
$v$ would extend by Schwarz reflection to a nontrivial jacobi field on 
$
  M:= \overline{S\cup \rho_Y S},
$
contradicting~\thetag{*}.
\end{proof}


\section{Rounding the curve $\Gamma$ and the family of surfaces $t\mapsto S(t)$}\label{rounding-section}

Our goal for the next few sections is to prove theorem~\ref{main-bumpy-theorem}. 
The proof is somewhat involved.  
It will be completed in section~\ref{smooth-count-section}.
From now until the end of section~\ref{smooth-count-section}, we fix a helicoid $H$ in $\sS^2\times\RR$
    and a height $h$ with $0<h<\infty$.
We let $\Gamma=\Gamma_C$ be the curve in theorem~\ref{main-bumpy-theorem}.
We also fix a  Riemannian metric on $\sS^2\times\RR$ that satisfies the
hypotheses of theorem~\ref{main-bumpy-theorem}.
In particular, in sections~\ref{rounding-section}\,--\,\ref{smooth-count-section}, 
 every result is with respect to that Riemannian metric.
In reading those sections, it may be helpful to imagine that the metric
is the standard product metric.  (In fact, for the purposes of proving theorem~\ref{theorem1},
the metric may as well be arbitrarily close to the standard product metric.)
Of course, in carrying out the proofs in sections~\ref{rounding-section}\,--\,\ref{smooth-count-section}, 
we must take care to use no property of the metric
other than those enumerated in theorem~\ref{main-bumpy-theorem}.

Note that theorem~\ref{main-bumpy-theorem} is about counting minimal surfaces mod $2$.
The  mod $2$ number of embedded minimal surfaces of a given topological type bounded by a smoothly 
embedded, suitably  bumpy curve is rather well understood. 
 For example, if the curve lies on the boundary
of a strictly convex set in $\RR^3$, the number is $1$ is the surface is a disk and is $0$ if not.
Of course the curve $\Gamma$ in theorem~\ref{main-bumpy-theorem} 
is neither smooth nor embedded, so to take advantage of such results,
we will round the corners of $\Gamma$ to make a smooth embedded curve, and we will 
use information
about the mod $2$ number of various surfaces bounded by the rounded curve to deduce
information about mod $2$ numbers of various surfaces bounded by the original curve $\Gamma$.

In this section, we define the notion of a ``rounding''. 
A rounding of $\Gamma$ is a one-parameter family 
  $t\in (0,\tau]\mapsto \Gamma(t)$ of smooth embedded curves (with
certain properties) that converge to $\Gamma$ as $t\to 0$.
Now if $\Gamma$ were smooth and bumpy, then by the implicit function theorem, any smooth
minimal surface $S(0)$ bounded by $\Gamma$ would extend uniquely to a one-parameter family
$t\in [0,\tau']\mapsto S(t)$ of minimal surfaces with $\partial S(t)\equiv \Gamma(t)$ (for some possibly smaller 
  $\tau'\in (0,\tau]$.)

It is natural to guess that this is also the case even in our situation, when $\Gamma$ is neither smooth
nor embedded.   
In fact, we prove that the guess is 
correct\footnote{The correctness of the guess can be viewed as a kind of bridge theorem. 
Though it does not quite follow from the bridge theorems in~\cite{smale-bridge} or  
in~\cite{white-stable-bridge,white-unstable-bridge}, we believe the proofs
there could be adapted to our situation.  However, the proof here is shorter and more elementary than
those proofs.
(It takes advantage of special properties of our surfaces.)}.
The proof is  still based on the implicit function
theorem, but the corners make the proof significantly more complicated.  However, the idea of the proof
is simple: we project the rounded curve $\Gamma(t)$ to a curve in the 
surface
\[  
  M:=\overline{S\cup \rho_ZS}
\] 
by the nearest point projection. We already have a minimal surface bounded by that projected curve:
it bounds a portion $\Omega(t)$ of $M$.  Now we smoothly isotope the projected curve back to $\Gamma(t)$,
and use the implicit function theorem to make a corresponding isotopy through minimal surfaces of $\Omega(t)$
to the surface $S(t)$ we want.  Of course we have to be careful to verify that we do not encounter
nontrivial jacobi fields on the way.

We also prove that, roughly speaking, the surfaces $S(t)$  (for the various $S$'s bounded by $\Gamma$)
account for {\em all} the minimal $Y$-surfaces bounded by
$\Gamma(t)$ when $t$ is sufficiently small.  
The precise statement (theorem~\ref{all-accounted-for-theorem}) is 
slightly more complicated because the larger the genus of the surfaces, the 
smaller one has to choose $t$.

Defining roundings, proving the existence of the associated one-parameter families $t\mapsto S(t)$
of minimal surfaces as described above, and proving basic properties of such families take up
the rest of this section and the following section.  Once we have those tools, 
the proof of theorem~\ref{main-bumpy-theorem}
 is not so hard: it is carried out
in section~\ref{counting-section}.  

To avoid losing track of the big picture, 
the reader may find it helpful initially to skip 
sections~\ref{Gamma(s,t)-definition}--\ref{eigenfunction-claim} 
(the proof of theorem~\ref{bridged-approximations-theorem})
as well as the proofs in section~\ref{additional-properties-section}, 
and then to read section~\ref{counting-section}, which contains the heart
of the proof of theorem~\ref{main-bumpy-theorem} and therefore also
(see remark~\ref{periodic-case-done-remark}) of the periodic case of theorem~\ref{theorem1}.

\begin{lemma}\label{TubularNeighborhoodLemma}
Suppose that  $S$   a minimal embedded $Y$-surface in $N=\overline{H^+}\cap\{|z|\le h\}$ with
 $\partial S=\Gamma$.   
Let 
\[
    V(S,\eps) = \{p\in \sS^2\times\RR: \dist(p,S) <\eps\}.
\]
For all sufficiently small $\eps>0$, the following hold:
\begin{enumerate}
\item\label{nearest-point-item} if $p\in V(S,\eps)$, then there is a unique point $\pi(p)$ in $\overline{S\cup\rho_Z S}$ nearest to $p$.
\item\label{unique-in-V-item}
if $S'$ is a $\rho_Y$-invariant minimal surface in $\overline{V(S,\eps)}$ with $\partial S'=\Gamma$,
  and if $S'$ is smooth except possibly at the corners of $\Gamma$,  then $S'=S$.
\end{enumerate}
\end{lemma}

\begin{proof}[Proof of Lemma~\ref{TubularNeighborhoodLemma}]
Assertion~\eqref{nearest-point-item} holds (for sufficiently small $\eps$) 
because $M:=\overline{S\cup \rho_ZS}$ is a
smooth embedded manifold-with-boundary.

Suppose assertion~\eqref{unique-in-V-item} fails.
Then there is a sequence of minimal $Y$-surfaces $S_n\subset \overline{V(S,\eps_n)}$ 
with $\partial S_n=\Gamma$ such that $S_n\ne S$ and such that $\eps_n\to 0$.
 Let $M_n$ be the closure of $S_n\cap \rho_Z S_n$ or (equivalently) of $S_n\cap \rho_X S_n$. 
 (Note that $\rho_Z S_n=\rho_X S_n$ by the $\rho_Y$-invariance of $S_n$.)
Then $M_n$ is a 
minimal surface with boundary $\partial M_n=\partial M$, $M_n$ is smooth away from $Y$ and from the corners of $\Gamma$, and 
\[
   \max_{p\in M_n}\dist(p,M)\to 0.
\]
Since $M$ is a smooth, embedded manifold with nonempty boundary,
 this implies 
that the convergence $M_n\to M$ is smooth 
by~\cite{white-controlling-area}*{6.1}.

A  {\em normal graph} of $f:S\to \RR$ over a hypersurface $S$ in a Riemannian manifold is  the hypersurface
$\{\exp_p(f(p)n(p))\, |\,\, p\in S\}$, where n(p) is a unit normal
vector field on $S\subset N$ and $\exp _p$ is the exponential mapping  at 
$p$.
From the previous paragraph, it follows that for all sufficiently large $n$,  $M_n$ is the 
    normal graph of a function $f_n:M\to \RR$ 
with $f_n\vert\Gamma=0$ such that $f_n\to 0$
smoothly.  
But then 
\[
  \frac{f_n}{\|f_n\|_0}
\]
converges (after passing to a subsequence) to a nonzero jacobi field on $S$ that vanishes
on $\partial S=\Gamma$, contradicting the assumption (hypothesis~\eqref{bumpy-hypothesis}
of theorem~\ref{main-bumpy-theorem}) that 
the Riemannian metric is bumpy with respect to $\Gamma$.
\end{proof}



\stepcounter{theorem}

\addtocontents{toc}{\SkipTocEntry}
\subsection{Roundings of $\Gamma$}\label{roundings-subsection}
 Let $t_0>0$ be less than half the
distance between any two corners of $\Gamma$. For $t$ satisfying
$0<t\leq t_0$, we  can form from $\Gamma$ a smoothly embedded $\rho_Y$-invariant
curve $\Gamma (t)$ in the portion of $H$ with $|z|\le h$ as follows:
\begin{enumerate}
\item If $q$ is a corner of $\Gamma$ other than $O$ or $O^*$, we
replace $\Gamma\cap\BB(q,t)$ by  a smooth curve in $H\cap
\BB(q,t)$ that has the same endpoints as $\Gamma\cap \BB(q,t)$ but
that is otherwise disjoint from $\Gamma\cap\BB(q,t)$. 
\item 
  If $q=O$
or $q=O^*$ we replace $\Gamma\cap \BB(q,t)$ by two smoothly embedded
curves in $H$ that have the same endpoints as $\Gamma\cap\BB(q, t)$ but that
are otherwise disjoint from $\Gamma\cap \BB(q,t)$. See
Figures~\ref{rounding-corners-figure} and \ref{signs-figure}.
\end{enumerate}
Note that $\Gamma(t)$ lies in the boundary of $\partial N$ of the region $N=\overline{H^+}\cap\{|z|\le h\}$.


\begfig
 \hspace{-1.0in}
  \vspace{-2.0in}
  \centerline{
\includegraphics[width=4.05in]{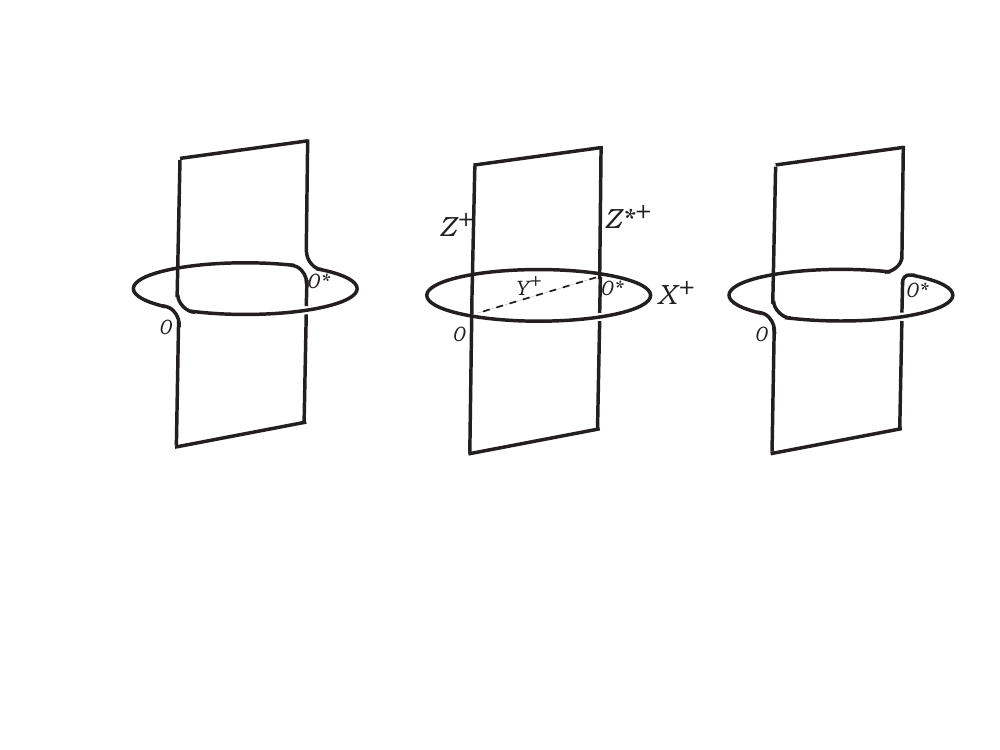}
   }
   \vspace{0.60 in}
 \begin{center}
 \parbox{5.1in}{
\caption{\label{rounding-corners-figure}  {\bf Rounding the corners  of $\Gamma$.}
{\bf Center:} The boundary curve $\Gamma$ as illustrated in Figure~\ref{GammaFigure}.
{\bf Left and Right:} Desingularizations of $\Gamma$. The corners at $O$ and $O^*$ are removed, following the conditions $(1)$ and $(2)$ of \ref{roundings-subsection}. In both cases we have desingularized near $O$ by joining $X^+$ to $Z^+$ and
$X^-$ to $Z^-$. In the language of Definition~\ref{PosNegRounding}, both desingularizations are {\em positive at $O$}.
On the left, the rounding is also positive  at $O^*$. On the right, the rounding is
{\em negative at $O^*$}. Note that when the signs of the rounding agree at $O$ and $O^*$, as they do on the left, 
the rounded curve  has two components; 
when the signs are different, as on the right, the rounded curve is  connected.
 }
 }
 \end{center}
 \endfig

\begin{definition}\label{Rounding1} Suppose $\Gamma(t)\subset H$ is a family of smooth embedded
$\rho _Y$-invariant curves created from $\Gamma$ according to the
recipe above. Suppose  we do this in such a way that that  for each
corner $q$  of $\Gamma$, the curve
\begin{equation}\label{expression}
       (1/t) ( \Gamma (t) - q)
\end{equation}
converges smoothly to a smooth, embedded planar curve $\Gamma'$ as $t\to 0$.
Then we say that the family $\Gamma (t)$ is a {\em rounding of}
$\Gamma$.   
\end{definition}

\begin{remark}\label{meaning-of-translation-remark} 
Since we are working in $\sS^2\times\RR$
 with some Riemannian metric, 
 it may not be immediately obvious what we mean by 
  translation and  by scaling in definition~\ref{Rounding1}. 
 However, there are various ways to make sense of it.  For example, by the Nash embedding
 theorem, we can regard $\sS^2\times\RR$ with the given Riemannian metric as 
 embedded isometrically 
 in some Euclidean space.  In that Euclidean space, the expression~\eqref{expression}
  is well defined, 
and its limit as $t\to 0$ lies in the $3$-dimensional tangent space (at $q$) to $\sS^2\times\RR$,
which is of course linearly isometric to $\RR^3$.
\end{remark}

\begin{remark}
In definition~\ref{Rounding1}, note that if the corner $q$ is $O$ or $O^*$, then $\Gamma'$ consists
of  two components, and $\Gamma'$ coincides with a pair of perpendicular lines outside a disk of radius $1$ about
the intersection of those lines.   In this case, $\Gamma'$ is the boundary of two regions in the plane: one
region is connected, and the other region (the complement of the connected region) consists of two connected components. 
 We refer to each of these regions as a {\bf rounded quadrant pair}.
 If $q$ is a corner other than $O$ or $O^*$, then $\Gamma'$ consists
of a single curve.  In this case, $\Gamma'$ bounds a planar region which, outside of a disk, coincides with a quadrant
of the plane.  We call such a region a {\bf rounded quadrant}.
\end{remark}

\stepcounter{theorem}
\addtocontents{toc}{\SkipTocEntry}
\subsection{The existence of bridged approximations to $S$}
We  will assume until further notice that $\Gamma\subset H$ bounds an embedded
minimal $Y$-surface $S$ in $N=\overline{H^+}\cap\{|z|\le h\}$. As in the previous section we
define  $M= \overline {S\cup\rho _Z S}$. 
For $p\in \sS^2\times\RR$, let $\pi(p)=\pi_M(p)$ be the point in $M$ closest to $p$,
provided that point is unique.   Thus the domain of $\pi$ is the set of all points in $\sS^2\times\RR$
such that there is a unique nearest point in $M$.
Since $M$ is a smooth embedded 
manifold-with-boundary, the domain of $\pi$ contains $M$ in its interior.

Consider a rounding 
 $\Gamma (t)$ of $\Gamma$ with $t\in [0,t_0]$.
 By replacing $t_0$ by a smaller value, we may assume that for all $t\in [0,t_0]$,
 the curve $\Gamma(t)$ is in the interior of the domain of $\pi$ and
 $\pi(\Gamma(t))$ is a smooth embedded curve in $M$.  It follows that $\Gamma(t)$ is the 
 normal graph of a function
 \[
     \phi_t: \pi(\Gamma(t))\to \RR.
 \]
We let $\Omega(t)$ be the domain in $M$ bounded by $\pi(\Gamma(t))$.

\begin{remark}\label{where-is-O-remark}
Suppose that $S$ is positive at $O$, i.e., that it is tangent to the positive quadrants of $H$
(namely the quadrant bounded by $X^+$ and $Z^+$ and the quadrant bounded by $X^-$ and $Z^-$.)
Note that $O$ is in $\Omega(t)$ if and only if
   $\Gamma(t)\cap \BB(O,t)$ lies in the {\em negative} quadrants of $H$, or, equivalently, 
 if and only if $\Gamma(t)\cap \BB(O,t)$ connects $Z^+$ to $X^-$ and $Z^-$ to $X^+$. 
 See figure~\ref{signs-figure}.
\end{remark}

\begin{theorem}\label{bridged-approximations-theorem}
There exists a $\tau>0$ and a smooth one-parameter family $t\in (0,\tau]\mapsto f_t$ of functions
\[  
    f_t: \Omega(t)\to \RR
\]
with the following properties:
\begin{enumerate}
\item The normal graph $S(t)$ of $f_t$ is a $Y$-nongenerate,
minimal embedded $Y$-surface with boundary $\Gamma(t)$,
\item $\|f_t\|_0 + \|Df_t\|_0\to 0$ as $t\to 0$, 
\item $S(t)$ converges smoothly to $S$ as $t\to 0$ except possibly at the corners of $S$,
\item\label{in-H^+-item} $S(t)$ lies in $\overline{H^+}$.
\end{enumerate}
\end{theorem}

Later (see theorem~\ref{strong-uniqueness-theorem}) we will prove that for small $t$, the surfaces $S(t)$ have a very strong uniqueness property.
In particular, given $S$, the rounding $t\mapsto \Gamma(t)$, and any sufficiently small if $\tau>0$,  
there is a unique family $t\in(0,\tau]\to S(t)$
having the indicated properties.

\begin{remark}\label{in-H^+-remark}
Assertion~\eqref{in-H^+-item} of the theorem follows easily from the preceding assertions, provided
we replace $\tau$ by a suitable smaller number.
To see this, 
note by the smooth convergence $S(t)\to S$ away from corners, each point of $S(t)\cap H^-$  must 
lie within distance $\eps_n$ of the corners of $S$, where $\eps_n\to 0$.
By the implicit function theorem,
each corner $q$ of $S$ has a neighborhood $U\subset \sS\times[-h,h]$
that is foliated by minimal surfaces, one of which is $\overline{M}\cap U$.  For $t$ sufficiently small, the set of
points of $S(t)\cap H^-$ that are near $q$ will be contained entirely in $U$, which violates the maximum
principle unless $S(t)\cap H^-$ is empty.   
\end{remark}

{\bf Idea of the proof of theorem~\ref{bridged-approximations-theorem}}. 
(The details will take up the rest of section~\ref{rounding-section}.)
   The rounding is a one-parameter family of curves $\Gamma(t)$.
We extend the one-parameter family to a two-parameter family $\Gamma(t,s)$ (with $0\le s\le1$)
in such a way 
that $\Gamma(t,1)=\Gamma(t)$ and $\Gamma(t,0)=\pi(\Gamma(t))$.  Now $\Gamma(t,0)$
trivially bounds a minimal $Y$-surface that is a normal graph over $\Omega(t)$, namely $\Omega(t)$
itself (which is the normal graph of the zero function).  We then use the implicit function
theorem to get existence for all $(t,s)$ with $t$ sufficiently small of a minimal embedded $Y$-surface
$S(t,s)$ with boundary $\Gamma(t,s)$.  Then $t\mapsto S(t,1)$ will be the desired one-parameter family
of surfaces.

  \begin{definition}\label{Gamma(s,t)-definition}
For $0\leq t< t_0$, each $\Gamma(t)$  is the normal graph over $\pi(\Gamma (t))$ of a function $\phi _t : \pi(\Gamma (t))\rightarrow \RR$.  For  $0\leq s\leq 1$, define
 \begin{equation}
 \Gamma (t,s): =\,\,\graph \,\, s\phi _t.
  \end{equation}
  
  Note that $\pi (\Gamma (t,s)) =\Gamma (t,0)$.
  \end{definition}

\begin{proposition}\label{two-parameter-proposition}
There is a $\tau>0$ and a smooth two-parameter family
\[
   (t,s)\in (0,\tau]\times[0,1]\mapsto S(t,s)
\]
of $Y$-nondegenerate, minimal embedded $Y$-surfaces such that each $S(t,s)$ has boundary $\Gamma(t,s)$
and is the normal graph of a function $f_{t,s}:\Omega(t)\to \RR$  such that
\[
   \|f_{t,s}\|_0 + \|Df_{t,s}\|_0 \to 0
\]
as $t\to 0$.  The convergence $f_{t,s}\to 0$ as $t\to 0$ is smooth away from the corners of $S$.
\end{proposition}

Theorem~\ref{bridged-approximations-theorem} follows from proposition~\ref{two-parameter-proposition} by setting $S(t):=S(t,1)$.
(See remark~\ref{in-H^+-remark}.)

\begin{proof}[Proof of proposition~\ref{two-parameter-proposition}]
Fix a $\eta>0$ and a $\tau>0$ and consider the following subsets of the domain $D:=(0,\tau]\times[0,1]$:
\begin{enumerate}
\item the relatively closed set $A$ of all $(t,s)\in D$ such that $\Gamma(t,s)$ bounds a minimal embedded $Y$-surface
 that is the normal graph of a function from $\Omega(t)\to \RR$ with Lipschitz constant $\le \eta$.
\item the subset $B$ of $A$ consisting
    of all $(t,s)\in D$ such that $\Gamma(t,s)$ bounds a minimal embedded $Y$-surface
 that is $Y$-nondegenerate and that is the normal graph of a function from $\Omega(t)$ to $\RR$
 with Lipschitz constant $< \eta$.
\item the subset $C$ of $A$ consisting of all 
   $(t,s)\in D$ such that there is exactly one function whose Lipschitz constant is $\le \eta$
  and whose normal graph is a minimal embedded $Y$-surface with boundary $\Gamma(t,s)$.
\end{enumerate}
By proposition~\ref{rounding-sequence-proposition} below, 
we can choose $\eta$ and $\tau$ so that these three sets are equal: $A=B=C$.
Clearly the set $A$ is a relatively open closed subset of $(0,\tau]\times[0,1]$.  Also, $A$ is nonempty
since it contains $(0,\tau]\times\{0\}$. (This is because $\Gamma(t,0)$ is the boundary of the minimal
$Y$-surface $\Omega(t)$, which is the normal graph of the zero function on $\Omega(t)$).
By the implicit function theorem, the set $B$ is a relatively open subset of $(0,\tau]\times[0,1]$.

Since $A=B=C$ is nonempty and since it is both relatively closed and relatively open in $(0,\tau]\times [0,1]$,
we must have
\[
   A=B=C=(0,\tau]\times [0,1].
\]
For each $(t,s)\in (0,\tau]\times [0,1]=C$, let $f_{t,s}:\Omega(t)\to \RR$  be the unique function with Lipschitz
constant $\le \eta$ whose normal graph is a minimal embedded $Y$-surface $S(t,s)$ with boundary $\Gamma(t,s)$.
Since $B=C$, in fact $f_{t,s}$ has Lipschitz constant $<\eta$ and $S(s,t)$ is $Y$-nondegenerate.
By the $Y$-nondegeneracy and the implicit function theorem, $S(t,s)$ depends smoothly on $(t,s)$.
Also, 
\begin{equation}\label{C^1-convergence-to-0}
    \|f_{t,s}\|_0 + \|Df_{t,s}\|_0 \to 0
\end{equation}
as $t\to 0$ by proposition~\ref{rounding-sequence-proposition} below.  
Finally, the smooth convergence $S(t,s)\to S$ away from corners 
follows from~\eqref{C^1-convergence-to-0} by standard elliptic PDE.
\end{proof}

\begin{proposition}\label{rounding-sequence-proposition}
There is an $\eta>0$ with the following property.
Suppose $S_n$ is a sequence of minimal embedded $Y$-surfaces with $\partial S_n=\Gamma(t_n,s_n)$
where $t_n\to 0$ and $s_n\in [0,1]$.  Suppose also that each $S_n$ is the normal
graph of a function 
\[
   f_n: \Omega(t_n)\to\RR
\]
with Lipschitz constant $\le \eta$.  
Then 
\begin{enumerate}
\item\label{C^1-item} $\|f_n\|_0 + \|Df_n\|_0\to 0$. (In particular, $Lip(f_n)<\eta$ for all sufficiently large $n$.)
\item\label{nondegenerate-item} $S_n$ is $Y$-nondegenerate for all sufficiently large $n$,
\item\label{unique-function-item} If $g_n$ is a function with Lipschitz constant $\le \eta$
  and if the graph of $g_n$ is a minimal embedded $Y$-surface, then $g_n=f_n$ for all sufficiently large $n$.
\item\label{unique-surface-item} 
If $\Sigma_n$ is a sequence of minimal embedded $Y$-surfaces such that $\partial \Sigma_n = \partial S_n$
and such that $\Sigma_n\subset V(S,\eps_n)$ where $\eps_n\to 0$, then $\Sigma_n=S_n$ for all
sufficiently large $n$.
\end{enumerate} 
\end{proposition}

\begin{proof}[Proof of proposition~\ref{rounding-sequence-proposition}.]
By lemma~\ref{TubularNeighborhoodLemma}, there is an $\eps>0$
 be such that $S$ is the only embedded minimal $Y$-surface in $V(S,\eps)$
with boundary $\Gamma$.
Choose $\eta>0$ small enough that if $f:S\to \RR$ is Lipschitz with Lipschitz constant $\le \eta$ and 
if $f|\Gamma=0$, then the normal graph of $f$ lies in $V(S,\eps)$.  In particular, if the graph of $f$
is a minimal embedded $Y$-surface, then $f=0$.

Since the $f_n$ have a common lipschitz bound $\eta$, they converge subsequentially
to a lipschitz function $f:S\to \RR$.  By the Schauder estimates, the convergence is smooth
away from the corners of $\Gamma$, so the normal graph of $f$ is minimal.  Thus by
choice of $\eta$, $f=0$.   This proves that
\[
   \|f_n\|_0 \to 0.
\]

Let
\[
  L = \limsup \|Df_n\|_0.
\]
We must show that $L=0$. By passing to a subsequence, we can
assume that the $\limsup$ is a limit, and we can choose a sequence
of points
 $p_n\in S_n\setminus\partial M_n=S_n\setminus\Gamma(t_n,s_n)$ such that
\[
    \lim | Df_n(p_n)| = L.
\]
By passing to a further subsequence, we can assume that the $p_n$
converge to a point $q\in S$.  If $q$ is not a corner of $S$, then
$f_n\to 0$ smoothly near $q$, which implies that $L=0$.

Thus suppose $q$ is a corner point of $S$, that is, one of the corners of $\Gamma$.  Let $R_n =
\dist(p_n,q)$.
Now translate $S_n$, $\Omega_n$, $Y$, and $p_n$ by $-q$ and dilate
by $1/R_n$ to get $S (t)'$, $\Omega_n'$, $Y_n'$ and $p_n'$.  Note
that $S(n')$ is the normal graph over $\Omega_n'$ of a function
$f_n'$ where the $\|Df_n'\|_0$ are bounded (independently of $n$).

By passing to a subsequence, we may assume that the $\Omega_n'$
converges to a planar region $\Omega'$, which must be one of the
following:
\begin{enumerate}
\item A quadrant 
\item a rounded quadrant. 
\item a quadrant pair.
\item a rounded quadrant pair. 
\item an entire plane.
\end{enumerate}
(If $q$ is $O$ or $O^*$, then (3), (4), and (5) occurs according
to whether $t_n/R_n$ tends to $0$, to a finite nonzero limit, or
to infinity.  If $q$ is one of the other corners, then (1) or (2)
occurs according to whether $t_n/R_n$ tends to $0$ or not.)  We
may also assume that the $f_n'$ converge to a lipschitz function
$f:\Omega'\to \RR$ and that the convergence is smooth away from
the origin. Furthermore, there is a point $p\in \Omega'$ with
\begin{equation}\label{e:Lpoint}
  \text{$|p|=1$ and $|Df(p)|=L$.}
\end{equation}

Suppose first that $\Omega'$ is a plane, which means that $q$ is
$O$ or  $O^*$, and thus that $Y'$ is the line that intersects the
plane of $\Omega'$ orthogonally. Since $S'$ is a minimal graph
over $\Omega'$,    $S'$ must also be a plane (by Bernstein's
theorem).  Since $Y$ interesects each $S(t)$ perpendicularly, $Y'$
must intersect $S'$ perpendicularly.  Thus $S'$ is a plane
parallel to $\Omega'$, so $Df'\equiv 0$. In particular, $L=0$ as
asserted.

Thus we may suppose that $\partial \Omega'$ (which is also
$\partial S'$) is nonempty.

By Schwartz reflection, we can extend $S'$ to a surface
$S^\dagger$ such that $\partial S^\dagger$ is a compact subset of
the plane $P$ containing $\partial S'$ and such that $S^\dagger$
has only one end, which is a lipschitz graph over that plane. Thus
the end is either planar or catenoidal.  It cannot be catenoidal
since it contains rays.  Hence the end is planar, which implies
that
\[
   \lim_{x\to\infty} f(x) = 0.
\]
But then $f\equiv 0$ by the maximum principle, so $Df\equiv 0$, and
therefore $L=0$ by \eqref{e:Lpoint}. This completes the proof that $\|Df_n\|_0\to 0$ and thus
the proof of assertion~\eqref{C^1-item}.

For the proofs of assertions~\eqref{nondegenerate-item}--\eqref{unique-surface-item}, 
it is convenient to make the following observation:

\begin{alt-claim}\label{planar-limit-domain-claim}
Suppose that $p_n\in S_n\setminus \partial S_n$ and that $\dist(p_n,\partial S_n)\to 0$.
Translate $S_n$ by $-p_n$ and dilate by $1/\dist(p_n, \partial S_n)$
to get a surface $S_n'$.  Then a subsequence of the $S_n'$ converges to one of the following
planar regions:
\begin{enumerate}[\upshape $\bullet$]
 \item a quadrant,
  \item a rounded quadrant,
\item a quadrant pair, 
 \item a rounded quadrant pair, or
 \item a halfplane.
\end{enumerate}
\end{alt-claim}

The claim follows immediately from the definitions (and the fact that $\|Df_n\|\to 0$) so we omit the proof.

Next we show assertion~\eqref{nondegenerate-item} of proposition~\ref{rounding-sequence-proposition}: 
 that $S_n$ is $Y$-nondegenerate for all sufficiently large $n$.
In fact, we prove somewhat more:

\begin{alt-claim}\label{eigenfunction-claim}
Suppose $u_n$ is an eigenfunction of the Jacobi operator on $S_n$ with eigenvalue $\lambda_n$,
normalized so that 
\[
    \|u_n\|_0 = \max  |u_n(\cdot)| = \max u_n(\cdot) = 1.
\]
Suppose also that the $\lambda_n$ are bounded.  Then (after passing to a subsequence)
the $S_n$ converge smoothly on compact sets to an eigenfunction $u$ on $S$ with
eigenvalue $\lambda=\lim_n\lambda_n$.
\end{alt-claim}

(With slightly more work, one could prove that for every $k$, the $k$th eigenvalue 
 of the jacobi operator on $S_n$ converges to the $k$th eigenvalue
of the jacobi operator on $S$.  However, we do not need that result.)

\begin{proof}
By passing to a subsequence, we can assume that 
the $\lambda_n$ converge to a limit $\lambda$, 
and that the $u_n$ converge smoothly
away from the corners of $S$ to a solution of
\[
    Ju = -\lambda u
\]
where $J$ is the jacobi operator on $S$. To prove the claim, it suffices to show
that $u$ does not vanish everywhere, and that $u$ extends continuously to the corners
of $S$.

Since $u$ is bounded, that $u$ extends continuously to the corners is a standard
removal-of-singularities result.
(One way to see it is as follows.  Extend $u$ by reflection to
the the smooth manifold-with-boundary $M= \overline{S\cup \rho_ZS}$.  Now $u$
solves $\Delta u = \phi u$ for a certain smooth function $\phi$ on $M$.  Let $v$ be the solution
of $\Delta v=\phi u$ on $M$ with $v|\partial M=0$ given by the Poisson formula.  Then $v$
is continuous on $M$ and smooth away from a finite set (the corners of $\Gamma$).
Away from the corners of $M$,  $u-v$ is a bounded harmonic function that vanishes
on $\partial M$.  But isolated singularities of bounded harmonic functions are removable,
so $u-v\equiv 0$.)

To prove that $u$ does not vanish everywhere, let $p_n$ be a point at which $u_n$
attains its maximum:
\[
   u_n(p_n) = 1 = \max_{S_n}|u_n(\cdot)|.
\]
By passing to a subsequence, we can assume that the $p_n$ converge to a point 
 $p\in \overline{S}$.
 We assert that $p\notin \partial S$.  For suppose $p\in\partial S$.
Translate\footnote{See Remark~\ref{meaning-of-translation-remark}.}
  $S_n$  by $-p_n$ and dilate by 
\[
  c_n := \frac1{\dist(p_n, \partial S_n)}= \frac1{\dist(p_n, \Gamma (t_n,s_n))}
\]
to get  $S_n'$.
Let $u_n'$ be the eigenfunction on $S_n'$ corresponding to $u_n$.
Note that $u_n'$ has eigenvalue $\lambda_n/c_n^2$.

We may assume (after passing to a subsequence) that the $S_n'$ converge to one of the 
planar regions $S'$ listed in lemma~\ref{planar-limit-domain-claim}. 
The convergence $S_n'\to S_n$ is smooth except possibly at the corner (if there
is one) of $S'$. 

By the smooth convergence of $S_n'$ to $S'$, the $u_n'$ converge
subsequentially to a jacobi field $u'$ on $S'$ that is smooth except
possibly at the corner (if there is one) of $S'$.   Since $S'$ is flat, $u'$ is a harmonic function.
Note that $u'(\cdot)$ attains its maximum
value of $1$ at $O$.  By the strong maximum principle for harmonic functions, $u'\equiv 1$ on the connected component of $S'\setminus \partial S'$
containing $O$.  But $u'\equiv 0$ on $\partial S'$, a contradiction.
Thus $p$ is in the interior of $S$, where the smooth convergence $u_n\to u$
implies that $u(p)=\lim u(p_n)=1$.

This completes
the proof of claim~\ref{eigenfunction-claim}
(and therefore also the proof of assertion~\eqref{nondegenerate-item} 
in proposition~\ref{rounding-sequence-proposition}.)
\end{proof}

To prove assertion~\eqref{unique-function-item} of proposition~\ref{rounding-sequence-proposition}, 
note that by assertion~\eqref{C^1-item} of the proposition applied to the $g_n$, 
\[
  \|g_n\|_0 + \|Dg_n\|\to 0.
\]
Thus if $\Sigma_n$ is the normal graph of $g_n$, then $\Sigma_n\subset V(S,\eps_n)$ for $\eps_n\to 0$.
Hence assertion~\eqref{unique-function-item} of the proposition is a special case 
of assertion~\eqref{unique-surface-item}.

Thus it remain only to prove assertion~\eqref{unique-surface-item}.  Suppose it is false.  
Then (after passing to a subsequence) there exist embedded minimal $Y$-surfaces
$\Sigma_n\ne S_n$ such that $\partial \Sigma_n=\partial S_n$ and such that
\[
  \text{$ \Sigma_n\subset \overline{V(S,\eps_n)}$ with $\eps_n\to 0$}. \tag{*}
\]
Now~\thetag{*} implies, by the extension of 
Allard's boundary regularity theorem in \cite{white-controlling-area},
that the $\Sigma_n$ converge smoothly to $S$ away from the corners of $S$.
(We apply theorem~6.1 of \cite{white-controlling-area} in 
the ambient space obtained by removing the 
corners of $S$ from $\sS^2\times \RR$.)

  Choose a point $q_n\in \Sigma_n$
that maximizes $\dist(\cdot, S_n)$.  Let $p_n$ be the point in $S_n$ closest to $q_n$.
Since $\partial S_n=\partial \Sigma_n$,
\begin{equation}\label{relative-distances}
    \dist(p_n,q_n) \le \dist(p_n, \partial \Sigma_n) = \dist(p_n, \partial S_n).
\end{equation}
By passing to a subsequence, we may assume that $p_n$ converges to a limit $p\in \overline{S}$.

If $p$ is not a corner of $S$,  then the smooth convergence $\Sigma_n\to S$ away
from the corners implies that there is a bounded $Y$-invariant jacobi field $u$ on
$S\setminus C$ such that $u$ vanishes on $\partial S\setminus C=\Gamma \setminus C$
and such that
\[
   \max |u(\cdot)| = u(p) = 1.
\]
By standard removal of singularities (see the second paragraph
of the proof of claim~\ref{eigenfunction-claim}), 
the function $u$ extends continuously to the corners.
By hypothesis, there is no such $u$. 
Thus $p$ must be one of the corners of $S$ (i.e., one of the corners of $\Gamma$.)
  Translate $S_n$, $\Sigma_n$, and $q_n$ by
$-p_n$ and dilate by $1/\dist(p_n, \partial S_n)$ to get $S_n'$, $\Sigma_n'$, and $q_n'$.

By passing to a subsequence, we can assume the the $S_n'$ converge to one
of the planar regions $S'$ listed in the statement of lemma~\ref{planar-limit-domain-claim}.
We can also assume that the $\Sigma'_n$ converge as sets to a limit set $\Sigma'$,
and that the points $q_n'$ converge to a limit point $q'$.
Note that 
\begin{equation}\label{bounded-distance-away}
  \sup_{\Sigma'} \dist(\cdot, S') = \dist(q',S') \le \dist(O, S') = 1
\end{equation}
by~\eqref{relative-distances}.

We claim that $\Sigma'\subset S'$.  
We prove this using catenoid barriers as follows.
Let $P$ be the plane containing $S'$ 
and consider a connected component $\mathcal{C}$ of the set of catenoids
whose waists are circles in $P\setminus S'$.
(There are either one or two such components according to whether
$P\setminus S'$ has one or two components.)   Note that the ends of each such
catenoid are disjoint from $\Sigma'$ since $\Sigma'$ lies with a bounded distance of $S'$.
By the strong maximum principle,
the catenoids in $\mathcal{C}$ either all intersect $S'$ or or all disjoint from $S'$.  
Now $\mathcal{C}$ contains catenoids whose waists are unit circles that arbitrarily far from $S'$.
Such a catenoid (if its waist is sufficiently far from $S'$) is disjoint from $\Sigma'$.  Thus all the
catenoids in $\mathcal{C}$ are disjoint from $\Sigma'$. We have shown that if the waist of catenoid
is a circle in $P\setminus S'$, then the catenoid is disjoint from $\Sigma'$.  The union of all such catenoids
is $\RR^3\setminus S'$, so $\Sigma'\subset S'$ as claimed.

Again by the extension of Allard's boundary regularity theorem
in \cite{white-controlling-area}*{theorem~6.1},
the $\Sigma_n'$ must converge smoothly to $S'$ except at the corner (if there
is one) of $S'$.

The smooth convergence of $\Sigma_n'$ and $S_n'$ to $S'$ implies
existence of a bounded jacobi field $u'$ on $S'$ that is smooth
except at the corner, that takes its maximum value of $1$ at $O$,
and that vanishes on $\partial S'$. 
Since $S'$ is flat, $u'$ is a harmonic function.  By the maximum principle, $u'\equiv 1$
on the connected component of $S'\setminus \partial S'$ containing $O$.
But that is a contradiction since $u'$ vanishes on $\partial S'$.
\end{proof}

 \section{Additional properties of the family $t\mapsto S(t)$}\label{additional-properties-section}
 
 We now prove that the surfaces $S(t)$ of theorem~\ref{bridged-approximations-theorem} 
 have a strong uniqueness property for small $t$:

\begin{theorem}\label{strong-uniqueness-theorem}
Let $t\in (0,\tau]\mapsto S(t)$ be the one-parameter family of minimal $Y$-surfaces given by 
theorem~\ref{bridged-approximations-theorem}.
For every sufficiently small $\eps>0$, there is a $\tau'>0$ with the following property.
For every $t\in (0,\tau']$, the surface $S(t)$ lies in $V(S,\eps)$ 
(the small neighborhood of $S$ defined in section~\ref{rounding-section}) 
and is the unique minimal embedded 
 $Y$-surface in $V(S,\eps)$ with boundary $\Gamma(t)$.
 \end{theorem}
 
\begin{proof}
Suppose the theorem is false.  Then there is a sequence of $\eps_n\to 0$
such that, for each $n$, either
\begin{enumerate}
\item there are arbitrarily large $t$ for which $S(t)$ is not contained in $V(S,\eps_n)$, or
\item there is a $t_n$ for which $S(t_n)$ is contained in $V(S,\eps_n)$ but such that
 $V(S,\eps_n)$ contains a second embedded minimal $Y$-surface $\Sigma_n$ with boundary $\Gamma(t_n)$.
 \end{enumerate}
 The first is impossible since $S(t)\to S$ as $t\to 0$.
 Thus the second holds for each $n$.  But (2) contradicts assertion~\eqref{unique-surface-item} of proposition~\ref{rounding-sequence-proposition}. 
\end{proof}

According to theorem~\ref{bridged-approximations-theorem}, 
for each embedded minimal $Y$-surface $S$ bounded by $\Gamma$,
we get a family of minimal surfaces $t\mapsto S(t)$ with $\partial S(t)=\Gamma(t)$.   The following theorem
says, roughly speaking, that as $t\to 0$, then the those surfaces account for all
minimal embedded $Y$-surfaces bounded by $\Gamma(t)$.

\begin{theorem}\label{all-accounted-for-theorem}
Let $t\mapsto \Gamma(t)$ be a rounding of $\Gamma$.  Let $S_n$ be a sequence of embedded minimal $Y$-surfaces
in $H^+\cap \{|z|\le h\}$ such that $\partial S_n = \Gamma(t_n)$ where $t_n\to 0$.  Suppose the number of points in
$S_n\cap Y^+$ is bounded independent of $n$.  Then, after passing to a subsequence, the $S_n$ converge
to smooth minimal embedded $Y$-surface $S$ bounded by $\Gamma$, and $S_n =S(t_n)$ for all sufficiently large $n$,
where $t\mapsto S(t)$ is the one-parameter family given by theorem~\ref{bridged-approximations-theorem}.
\end{theorem}

\begin{proof}
The areas of the $S_n$ are uniformly bounded by hypothesis 
on the Riemannian metric on $\sS^2\times\RR$:
see~($3''$) in remark~\ref{isoperimetric-equivalence-remark}. 
Using  the Gauss-Bonnet theorem, the minimality of the $S_n$, 
and the fact that the sectional curvatures 
of $\sS^2\times\RR$ are bounded, it follows that 
\[
  \int_{S_n} \beta(S_n, \cdot)\,dA
\]
is uniformly bounded, where $\beta(S_n,x)$ is the square of the norm of the second fundamental form of $S_n$ at $x$.
It follows (see~\cite{white-curvature-estimates}*{theorem~3}) that after passing to a subsequence, 
the $S_n$ converge smoothly
(away from the corners of $\Gamma$)
 to an minimal embedded $Y$-surface $S$ with boundary $\Gamma$. 
By the uniqueness theorem~\ref{strong-uniqueness-theorem}, $S_n=S(t_n)$ for all sufficiently large $n$.
\end{proof}



\section{Counting the number of points in $Y\cap S(t)$}\label{counting-section}

Consider a rounding $t\rightarrow\Gamma(t)$ of a boundary curve $\Gamma$, as specified in 
definition~\ref{Rounding1}. 
There are two qualitatively different ways to do the rounding at the crossings $O$ and $O^*$. We describe
what can happen at $O$ (the same description holds at $O^*$):
\begin{enumerate}
\item Near $O$, each  $\Gamma(t)$ connects points of $Z^+$ to points
of $X^+$ (and therefore points of $Z^-$ to points of $X^-$), or
\item the curve $\Gamma(t)$ connects points of $Z^+$ to points of $X^-$
(and therefore points of $Z^-$ to points of $X^+$.)
\end{enumerate}
\begin{definition} \label{PosNegRounding}
 In case (1),  the rounding  $t\rightarrow\Gamma(t)$ is {\em positive} at $O$.
In case (2),   the rounding $t\rightarrow\Gamma(t)$ is {\em negative} at $O$. 
Similar statements hold at $O^*$.
\end{definition}
In what follows, we will use the notation $\|A\|$ to denote the number of elements in a finite set $A$.

\begin{proposition}\label{disagreements-proposition}
Let $S$ be an open minimal embedded $Y$-surface in 
$N:=\overline{H^+} \cap \{|z|\le h\}$ bounded by $\Gamma$.
Let $t\mapsto S(t)$ be the family given by theorem~ \ref{bridged-approximations-theorem}, and suppose $S\cap Y$ has
exactly $n$ points.  Then 
\[
     \|S(t)\cap Y\|  =  \|S\cap Y\|  + \delta(S, \Gamma(t))
\]
where $\delta(S,\Gamma(t))$ is $0$, $1$, or $2$ according to whether
according to whether the signs of $S$ and $\Gamma(t)$ agree at both $O$ and $O^*$, 
at one but not both of $O$ and $O^*$, or at neither $O$ nor $O^*$.
(In other words, $\delta(S,\Gamma(t))$ is the number of sign disagreements of $S$ and $\Gamma(t)$.
 See Figure~\ref{signs-figure}.)
\end{proposition}

\begin{proof}
Recall that $S(t)$ is normal graph over $\Omega(t)$, the region in $M$ bounded by
the image of $\Gamma(t)$ under the nearest point projection from a neighborhood of $M$ to $M$.
It follow immediately that 
\[
   \| Y\cap S(t)\| =  \|Y\cap \Omega(t)\|.
\]
Note that $Y\cap \Omega(t)$ consists of $Y\cap S$ 
together with one or both of  the points $O$ and $O^*$.  
(The points $O$ and $O^*$ in $\partial S=\Gamma$ do not belong to $S$ because $S$ is open.)
Recall also (see remark~\ref{where-is-O-remark}) that $O\in \Omega(t)$ if and only if $S$ and $\Gamma(t)$
have the same sign at $O$.  Likewise, $O^*\in \Omega(t)$ if and only if $S$ and $\Gamma(t)$ have
the same sign at $O^*$.  The result follows immediately.
\end{proof}

\begin{definition} 
Let $\Mm(\Gamma)$ be the set of all open, minimal embedded $Y$-surfaces $S\subset N$
  such that $\partial S=\Gamma$.
  (Here $\Gamma=\Gamma_C$ is the curve in the statement of theorem~\ref{main-bumpy-theorem}.)

Let $\Mm(\Gamma,n,s)$ be the set surfaces $S$ in $\Mm(\Gamma)$ 
  such that $S\cap Y=n$ and such that $S$ has sign $s$ at $O$.

 If $\Gamma'$ is a smooth, $\rho_Y$-invariant curve (e.g., one of the rounded curves $\Gamma(t)$)
 in $\overline{H^+}$ such 
 that $\Gamma'/\rho_Y$ has
  exactly one component, we let $\Mm(\Gamma',n)$ be the set of 
  embedded minimal $Y$-surfaces $S$ in 
    $\overline{H^+}$ such that $\partial S=\Gamma'$ 
    and such that $S\cap Y$ has exactly $n$ points.
\end{definition}

\begin{proposition}\label{double-positive-rounding-proposition}
Suppose the rounding $t\mapsto \Gamma(t)$ is positive at $O$ and at $O^*$.
\begin{enumerate}[\upshape (1)]
\item If $n$ is even and $S\in \Mm(\Gamma,n,+)$, then $S(t)\in \Mm(\Gamma(t),n)$.
\item If $n$ is odd and $S\in \Mm(\Gamma,n,s)$, then $S(t)\in \Mm(\Gamma(t),n+1)$.
\item If $n$ is even and $S\in \Mm(\Gamma,n,-)$, then $S(t)\in \Mm(\Gamma(t),n+2)$.
\end{enumerate}
\end{proposition}

\begin{remark}\label{sign-switch-remark}
Of course, the statement remains true if we switch all the signs.
\end{remark}

\begin{proof}
If $n$ is odd and  $S\in \Mm(\Gamma,n,s)$, then 
 $S$ has different signs at $O$ and $O^*$ by lemma~\ref{parity-sign-lemma}, 
and thus $\delta(S,\Gamma(t))=1$.

Now suppose that $n$ is even and that $S\in \Mm(\Gamma,n,s)$.  
Then by lemma~\ref{parity-sign-lemma}, the surface $S$ has the same
sign at $O^*$ as at $O$, namely $s$.  Thus $\delta(S,\Gamma)$ is $0$ if $s=+$ and is
$2$ if $s$ is negative. Proposition~\ref{double-positive-rounding-proposition} now follows
immediately from proposition~\ref{disagreements-proposition}.
\end{proof}

\begin{proposition}\label{mixed-sign-rounding-proposition}
Suppose the rounding $t\mapsto \Gamma(t)$ has sign $s$ at $O$ and $-s$ at $O^*$.
\begin{enumerate}[\upshape (1)]
\item If $n$ is odd and $S\in\Mm(\Gamma,n,s)$, then $S(t)\in \Mm(\Gamma(t),n)$.
\item If $n$ is even and $S$ is in $\Mm(\Gamma,n,+)$ or $\Mm(\Gamma,n,-)$, then $S(t)\in \Mm(\Gamma(t),n+1)$.
\item If $n$ is odd and $S\in \Mm(\Gamma,n,-s)$, then $S(t)\in \Mm(\Gamma(t),n+2)$.
\end{enumerate}
\end{proposition}

The proof is almost identical to the proof of proposition~\ref{double-positive-rounding-proposition}.

\begin{theorem}\label{main-count-theorem}
For every nonnegative integer $n$ and for each sign $s$, the set $\Mm(\Gamma,n,s)$ has
an odd number of surfaces.
\end{theorem}


\begfig
\centerline{
         \includegraphics[width=3.0in]{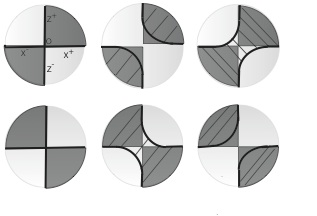}
}
 \caption{\label{signs-figure} 
  {\bf The sign of $S$ and $\Gamma (t)$ at $O$.}\newline 
   The behavior near $O$ of a surface $S\subset H^+$ with boundary $\Gamma$.\newline  {\bf First Column:}
  The surface $S$, here illustrated by
 the darker shading,   is tangent at
 $O$ to either the positive quadrants of $H$ (as illustrated on
 top) or the negative quadrants (on the bottom).  In the sense of section~\ref{sign-section},
  $S$ is positive at $O$ in the top illustration and negative  in the bottom illustration.  \newline {\bf Second column:} A curve $\Gamma (t)$ in a  positive rounding $t\rightarrow \Gamma(t)$ of $\Gamma$.   The  striped regions lie in the projections
 $\Omega (t)$ defined in Theorem~\ref{bridged-approximations-theorem}. Note that on the top $O\not\in \Omega (t) $.
  On the bottom, $O\in \Omega (t)$.   
 \newline  {\bf Third Column:} A curve $\Gamma (t)$ in a negative rounding of $\Gamma$. The striped regions lie in
 $\Omega (t)$.  Note that on top we have  $O\in \Omega (t)$.
  On the bottom, $O\not \in \Omega(t)$. 
  }
\endfig


\begin{remark}\label{periodic-case-done-remark}
Note that theorem~\ref{main-count-theorem} is the same as theorem~\ref{main-bumpy-theorem},
because ever since section~\ref{bumpy-section}, we have been working with an arbitrary Riemannian metric
on $\sS^2\times\RR$ that satisfies the hypotheses of theorem~\ref{main-bumpy-theorem}.
By proposition~\ref{bumpy-suffices-proposition}, 
theorem~\ref{main-bumpy-theorem} implies 
theorem~\ref{special-existence-theorem-tilted},
 which by proposition~\ref{tilted-suffices-proposition}
  implies theorem~\ref{special-existence-theorem}, 
  which by proposition~\ref{reduction-to-H^+-proposition} 
  implies theorem~\ref{theorem1} for $h<\infty$.  
  Thus in proving theorem~\ref{main-count-theorem}, we complete the proof of the periodic case of 
  theorem~\ref{theorem1}.
\end{remark}

\begin{proof}
Let $f(n,s)$ denote the mod $2$ number of surfaces in $\Mm(\Gamma,n,s)$.
Note that $f(n,s)=0$ for $n<0$ since $Y\cap S$ cannot have a negative number of points.
The theorem asserts that $f(n,s)=1$ for every $n\ge 0$.

We prove that the theorem by induction.  Thus we let $n$ be a nonnegative integer,
we assume that $f(k,s)=1$ for all nonnegative $k< n$ and $s=\pm$, and we must prove
that $f(n,s)=1$.

{\bf Case 1}: $n$ is even and $s=+$.

To prove that $f(n,+)=1$, we choose a rounding
 $t\mapsto \Gamma(t)$ that is positive at both $O$ and $O^*$.

We choose $\tau$ sufficiently small that for every $S\in\Mm(\Gamma)$ with $\|Y\cap S\|\le n$,
the family $t\mapsto S(t)$ is defined for all $t\in (0,\tau]$.   We may also choose $\tau$ small
enough that if $S$ and $S'$ are two distinct such surfaces, then $S(t)\ne S'(t)$ for $t\le\tau$.
(This is possible since $S(t)\to S$ and $S'(t)\to S'$ as $t\to 0$.)

By theorem~\ref{all-accounted-for-theorem}, we can fix a $t$ sufficiently small that
 for each surface $\Sigma\in \Mm(\Gamma(t),n)$, there is a surface 
$S=S_\Sigma\in \Mm(\Gamma)$
such that $\Sigma=S_\Sigma(t)$.  Since all such $S(t)$ are $\rho_Y$-nondegenerate,
this implies
\begin{equation}\label{suitably-bumpy}
\text{The surfaces in $\Mm(\Gamma(t), n)$ are all $\rho_Y$-nondegenerate.}
\end{equation}

By proposition~\ref{double-positive-rounding-proposition}, $S_\Sigma$ belongs to the union $U$ of
\begin{equation}\label{gang-of-four}
 \text{$\Mm(\Gamma,n,+)$,  $\Mm(\Gamma,(n-1),+)$, $\Mm(\Gamma,(n-1),-)$, and $\Mm(\Gamma, (n-2),-)$.}
\end{equation}
By the same proposition, if $S$ belongs to the union $U$, then $S(t)\in \Mm(\Gamma(t),n)$.
Thus $\Sigma\mapsto S_\Sigma$ gives a bijection from $\Mm(\Gamma(t),n)$ to $U$,
so the number of surfaces in $\Mm(\Gamma(t),n)$ is equal to the sum of the numbers of surfaces
in the four sets in~\eqref{gang-of-four}.  Reducing mod $2$ gives
\begin{equation}\label{recursion-formula-even}
  \| \Mm(\Gamma(t), n) \|_\text{mod $2$} = f(n,+) + f((n-1),+) + f((n-1),-) + f((n-2),-).
\end{equation}
By induction, $f(n-1,+)=f(n-1,-)$ (it is $0$ for $n=0$ and $1$ if $n\ge 2$), so
\begin{equation}\label{recursion-formula-even}
  \| \Mm(\Gamma(t), n) \|_\text{mod $2$} = f(n,+)  + f((n-2),-).
\end{equation}

As mentioned earlier, we have good knowledge about the mod $2$ number of minimal surfaces
bounded by suitably bumpy smooth embedded curves.  In particular, $\Gamma(t)$ is smooth
and embedded and has the bumpiness property~\eqref{suitably-bumpy}, which implies
that (see theorem~\ref{Y-degree-theorem})
\begin{equation}\label{rounded-curve-count}
\|\Mm(\Gamma(t),n)\|_\text{mod $2$}
=
\begin{cases}
1 &\text{if $n=1$ and $\Gamma(t)$ is connected}, \\
1 &\text{if $n=0$ and $\Gamma(t)$ is not connected, and} \\
0 &\text{in all other cases.}
\end{cases}
\end{equation}
Combining~\eqref{recursion-formula-even} and~\eqref{rounded-curve-count} gives
    $f(n,+)=1$.

{\bf Case 2}: $n$ is even and $s$ is $-$.  The proof is exactly like the proof of case 1, except
that we use a rounding that is negative at $O$ and at $O^*$.  (See remark~{sign-switch-remark}.)

{\bf Cases 3 and 4}: $n$ is odd and $s$ is $+$ or $-$.

The proof is almost identical to the proof in the even case, except that we use
a rounding $t\mapsto \Gamma(t)$ that has sign $s$ at $O$ and $-s$ at $O^*$.
In this case we still get a bijection $\Sigma\mapsto S_\Sigma$, but it is a bijection
from $\Mm(\Gamma(t),n)$ to the union $U$ of the sets
\begin{equation}\label{odd-gang-of-four}
 \text{$\Mm(\Gamma,n,s)$, $\Mm(\Gamma,(n-1),+)$, $\Mm(\Gamma,(n-1),-)$, and $\Mm(\Gamma,(n-2),-s)$.}
\end{equation}
Thus $\Mm(\Gamma(t),n)$ and $U$ have the same number of elements mod $2$:
\[
  \|\Mm(\Gamma(t),n)\|_ \text{mod $2$} =  f(n,s) + f((n-1),+) + f((n-1),-) + f((n-2),-s).
\]
As in case 1, $f((n-1),+)=f((n-1),-)$ by induction, so their sum is $0$:
\[
  \|\Mm(\Gamma(t),n)\|_ \text{mod $2$} =  f(n,s) + f(n-2,-s).
\] 
Combining this with~\eqref{rounded-curve-count} gives
$f(n,s)=1$.
\end{proof}


\renewcommand{\thesubsection}{\thetheorem}
   
\section{Counting minimal surfaces bounded by smooth curves}\label{smooth-count-section}

In the previous section, we used certain the facts about mod $2$ numbers of minimal surfaces
bounded by smooth curves.  In this section we state those facts, and show that they apply
in our situation.   The actual result we need is theorem~\ref{Y-degree-theorem} below,
and the reader may go directly to that result.  However, we believe it may be helpful to 
first state a simpler result that has the main idea of theorem~\ref{Y-degree-theorem}:

\begin{theorem}\label{degree-theorem}
Suppose $N$ is compact, smooth, strictly mean convex Riemannian $3$-manifold diffeomorphic to the
a ball.
Suppose also that $N$ contains no smooth, closed minimal surfaces.  
Let $\Sigma$ be any compact $2$-manifold
with boundary.   Let $\Gamma$ be a smooth embedded curve in $\partial N$, and
let $\MM(\Gamma, \Sigma)$ be the set of embedded minimal surfaces in $N$ that have boundary $\Gamma$ and
that are diffeomorphic to $\Sigma$.   Suppose all the surfaces in $\MM(\Gamma,\Sigma)$ are nondegenerate.
Then the number of those surfaces is odd if $\Sigma$ is a disk or union of disks, and is even if not.
\end{theorem}

See~\cite{hoffman-white-number}*{theorem~2.1} for the proof.

If we replace the assumption of strict mean convexity by mean convexity, then $\Gamma$ may bound
a minimal surface in $\partial N$.   In that case, theorem~\ref{degree-theorem}
 remains true provided (i) we assume that
no two adjacent components of $\partial N\setminus \Gamma$ are both minimal surfaces, 
and (ii) we count minimal surfaces
in $\partial N$ only if they are stable.
  Theorem~\ref{degree-theorem} also generalizes to the case of curves and surfaces
invariant under a finite group $G$ of symmetries of $N$.  If one of those symmetries is $180^\circ$ rotation
about a geodesic $Y$, then the theorem also generalizes to $Y$-surfaces:

\begin{theorem}\label{Y-degree-theorem}
Let $N$ be a compact region in a smooth Riemannian $3$-manifold such that $N$ is homeomorphic to the $3$-ball.
Suppose that $N$ has piecewise smooth, weakly mean-convex boundary,
and that $N$ contains no closed minimal surfaces.
Suppose also that  $N$ admits a $180^\circ$ rotational symmetry $\rho_Y$ about a geodesic $Y$.

Let $C$ be a $\rho_Y$-invariant smooth closed curve in $(\partial N)\setminus Y$ such that $C/\rho_Y$ is connected,
 and such that no two adjacent components of $(\partial N)\setminus C$ are both smooth minimal surfaces.
Let $\Mm^*(C,n)$ be the collection of $G$-invariant, 
minimal embedded $Y$-surfaces $S$ in $N$ with boundary $C$ such that
(i) $S\cap Y$ has exactly $n$ points, and
(ii) if $S\subset \partial N$, then $S$ is stable.
Suppose $C$ is $(Y,n)$-bumpy in the following sense:
all the $Y$-surfaces in  $\Mm^*(C,n)$ are $\rho_Y$-nondegenerate (i.e., have no nontrivial $\rho_Y$-invariant
jacobi fields.)
Then:
\begin{enumerate}
\item\label{union-of-disks-case}
     If $C$ has two components and $n=0$, the number of surfaces in $\Mm^*(C,n)$ is odd.
\item\label{disk-case}
     If $C$ has one component and $n=1$, the number of surfaces in $\Mm^*(C,n)$ is odd.
\item\label{more-complicated-case}
     In all other cases, the number of surfaces in $\Mm^*(C,n)$ is even.
\end{enumerate}
\end{theorem}

We remark (see corollary~\ref{$Y$-corollary})
 that in case~\eqref{disk-case}, each surface in $\Mm^*(C,n)$ is a disk, 
 in case~\eqref{union-of-disks-case}, each surface
in $\Mm^*(C,n)$ is the union of two disks, 
and in case~\eqref{more-complicated-case}, each surface is $\Mm^*(C,n)$ 
has more complicated topology (it is connected but not simply connected).

Theorem~\ref{Y-degree-theorem} is proved 
in \cite{hoffman-white-number}*{\S4.7}.

In the proof of theorem~\ref{main-count-theorem}, 
we invoked the conclusion of theorem~\ref{Y-degree-theorem}.
We now justify that. 
  Let $\Gamma(t)$ be the one of the curves formed by rounding $\Gamma$ in section~\ref{roundings-subsection}.
Note that $\Gamma(t)$ bounds a unique minimal surface $\Omega(t)$ that lies in the helicoidal portion of 
 $\partial N$, i.e, that lies in $H\cap \{|z|\le h\}$.
(The surface $\Omega(t)$ is a topologically a disk, an annulus, or a pair of disks, depending 
on the signs of the rounding at $O$ and $O^*$.)
Note also that the complementary region $(\partial N)\setminus \Omega(t)$ is piecewise smooth, but
not smooth.
To apply theorem~\ref{Y-degree-theorem} as we did, we must check that:
\begin{enumerate}[\upshape (i)]
\item\label{no-closed-surface-item} $N$ contains no closed minimal surfaces.
\item\label{no-adjacent-item} No two adjacent components of $\partial N$ are smooth minimal surfaces.
\item\label{stable-item} The surface $\Omega(t)$ is strictly stable. 
    (We need this because in the proof of Theorem~\ref{main-count-theorem}, we counted $\Omega(t)$,
     whereas theorem~\ref{Y-degree-theorem} tells us to count it only if it is stable.)
\end{enumerate}
Now~\eqref{no-closed-surface-item} is true by hypothesis on the Riemannian metric on $N$:
see~theorem~\ref{main-bumpy-theorem}\eqref{no-closed-minimal-surface-hypothesis}.
Also,~\eqref{no-adjacent-item} is true because (as mentioned above) the surface $(\partial N)\setminus \Omega(t)$
is piecewise-smooth but not smooth.

On the other hand,~\eqref{stable-item} need not be true in general. 
However, in the proof of theorem~\ref{main-count-theorem}, 
we were allowed to choose $t>0$ as small as we like,
and~\eqref{stable-item} is true if $t$ is sufficiently small:

\begin{lemma}\label{strictly-stable-lemma}
Let $t\mapsto \Gamma(t)\subset H$ be a rounding as in theorem~\ref{bridged-approximations-theorem}.
Then the region $\Omega(t)$ in $\partial N$ bounded by $\Gamma(t)$ is strictly stable
provided $t$ is sufficiently small.
\end{lemma}

(We remark that is a special case of a more general principle: if two strictly stable minimal
surfaces are connected by suitable thin necks, the resulting surface is also strictly stable.)

\begin{proof}
Let $\lambda(t)$ be the lowest eigenvalue of the jacobi operator on $\Omega(t)$.
Note that $\lambda(t)$ is bounded. (It is bounded below by the lowest
eigenvalue of a domain in $H$ that contains all the $\Omega(t)$ and above by the lowest
eigenvalue of a domain that is contained in all the $\Omega(t)$.)  
It follows that any subsequential limit $\lambda$ as $t\to 0$ of the $\lambda(t)$ is an eigenvalue
of the jacobi operator on $\Omega$, where $\Omega$ is the region in $H$ bounded
by $\Gamma$.
(This is a special case of claim~\ref{eigenfunction-claim}.)

By hypothesis\footnote{This is the only place where that hypothesis is used.}~\eqref{strictly-stable-hypothesis}
of theorem~\ref{main-bumpy-theorem}, $\Omega$ is strictly stable, so $\lambda>0$ and therefore
 $\lambda(t)>0$ for all sufficiently small $t>0$.
\end{proof}

 \section{General results on existence of limits}\label{general-results-section}
    
At this point, we have completed the proof of theorem~\ref{theorem1} in the case $h<\infty$.
That is,
we have established the existence of periodic genus-$g$ helicoids in $\sS^2(R)\times \RR$.
During that proof (in sections~\ref{bumpy-section}--\ref{smooth-count-section}), 
we considered rather general Riemannian metrics on $\sS^2(R)\times\RR$.
However, from now on we will always use the standard product metric.
In the remainder of the paper, 
\begin{enumerate}
\item We prove existence of nonperiodic genus-$g$ helicoids in $\sS^2(R)\times \RR$ 
by taking limits of periodic examples as the period tends to $\infty$.
\item We prove existence of
helicoid-like surfaces in $\RR^3$ by taking suitable
limits of nonperiodic examples in $\sS^2(R)\times \RR$ as $R\to\infty$.
\end{enumerate}
(We remark that one can also get periodic genus $g$-helicoids in $\RR^3$
as limits of periodic examples in $\sS^2(R)\times \RR$ as $R\to\infty$ with the
period kept fixed.)

Of course one could take the limit as sets in the Gromov-Hausdorff sense.
But to get smooth limits, one needs curvature estimates and local area bounds:
without curvature estimates, the limit need not be smooth, 
whereas with curvature estimates but without local area bounds, limits might be minimal 
laminations rather than smooth, properly embedded surfaces.

In fact, local area bounds are the key, because such bounds allow
one to use the following compactness theorem 
(which extends similar results in~\cite{choi-schoen}, \cite{anderson}, 
and~\cite{white-curvature-estimates}):


\begin{theorem}[General Compactness Theorem]\label{general-compactness-theorem}
Let $\Omega$ be an open subset of a Riemannian $3$-manifold.
Let $g_n$ be a sequence of smooth Riemannian metrics on $\Omega$ converging smoothly
to a Riemannian metric $g$.
Let $M_n\subset \Omega$ be a sequence of properly embedded surfaces 
such that $M_n$ is minimal with respect to $g_n$.
Suppose also that  the area and the genus of $M_n$ are bounded independently of $n$.
Then (after passing to a subsequence) the $M_n$ converge to a  smooth, properly embedded
$g$-minimal surface $M'$.  For each connected component $\Sigma$ of $M'$,
either
\begin{enumerate}
\item the convergence to $\Sigma$ is smooth with multiplicity one, or
\item the convergence is smooth (with some multiplicity $>1$) away from a discrete set $S$.
\end{enumerate}
In the second case, if $\Sigma$ is two-sided, then it must be stable.

Now suppose $\Omega$ is an open subset of $\RR^3$.  (The metric $g$ need not be flat.)
If $p_n\in M_n$ converges to $p\in M$, then (after passing to a further subsequence)
either
\[
    \Tan(M_n,p_n)\to \Tan(M,p)
\]
or there exists constants $c_n>0$ tending to $0$ such that the surfaces
\[
    \frac{M_n - p_n}{c_n}
\]
converge to a non flat complete embedded minimal surface $M'\subset \RR^3$ of finite total 
curvature with ends parallel to $\Tan(M,p)$.

\end{theorem}


See~\cite{white-embedded} for the proof.

When we apply theorem~\ref{general-compactness-theorem}, 
in order to get smooth convergence everywhere (and not just away from
a discrete set), we will prove that the limit surface has no stable components.   For that,
 we will use the following theorem of Fischer-Colbrie and Schoen. (See theorem~3 on page 206
and paragraph 1 on page 210 of \cite{fischer-colbrie-schoen}.)

\begin{theorem*}
Let $M$ be an orientable, complete, stable minimal surface in a complete, orientable Riemannian $3$-manifold
of nonnegative Ricci curvature.  Then $M$ is totally geodesic, and its normal bundle is Ricci flat.
(In other words, if $\nu$ is a normal vector to $M$, then $\operatorname{Ricci}(\nu,\nu)=0$.)
\end{theorem*}

\begin{corollary}\label{stable-spheres-corollary}
If $M$ is a connected, stable, properly embedded, minimal surface in $\sS^2\times \RR$, then
$M$ is a horizontal sphere.
\end{corollary}

To prove the corollary, note that since $\sS^2\times\RR$ is orientable and simply connected
and since $M$ is properly embedded, $M$ is orientable.  
Note also that if $\operatorname{Ricci}(\nu,\nu)=0$,
then $\nu$ is a vertical vector.

In section~\ref{area-bounds-section}, 
we prove the area bounds we need to get nonperiodic examples in $\sS^2\times\RR$.
In section~\ref{nonperiodic-section}, 
we prove area and curvature bounds in $\sS^2(R)\times \RR$ as $R\to\infty$.
In section~\ref{R3-section}, we get examples in $\RR^3$ by letting $R\to\infty$.


\section{Uniform Local Area Bounds in $\sS^2\times\RR$}\label{area-bounds-section}

Let 
$
    \theta: \overline{H^+}\setminus (Z\cup Z^*)\to \RR
$
be the natural angle function which, if we identify $\sS^2\setminus\{O^*\}$ with $\RR^2$ by stereographic
projection, is given by $\theta(x,y,z)=\arg(x+iy)$.  Note that since $\overline{H^+}$ is simply connected,
we can let $\theta$ take values in $\RR$ rather than in $\RR$ modulo $2\pi$.

\begin{proposition}\label{flux-bounds-proposition}
Suppose $H$ is a helicoid in $\sS^2\times\RR$ with axes $Z$ and $Z^*$.
Let $M$ be a minimal surface in $\overline{H^+}$ with compact, piecewise-smooth boundary,
and let
\[
   S = M \cap \{a \le z \le b\} \cap \{\alpha\le \theta \le \beta\}.
\]
Then
\begin{align*}
  \area(S)
   \le
  (b-a)\int_{(\partial M)\cap\{z>a\}}|\vv_z \cdot \nupt |\,ds  
  +  
  (\beta-\alpha)\int_{(\partial M)\cap \{\theta>\alpha\}} |\vv_\theta\cdot \nupt | \,ds
\end{align*}
where $\vv_z=\pdf{}{z}$, $\vv_\theta=\pdf{}{\theta}$, and where $\nupt$
is the unit normal to $\partial M$ that points out of $M$.
\end{proposition}

\begin{proof}
Let $u:\sS^2\times \RR$ be the function $z(\cdot)$ or the function $\theta(\cdot)$.
In the second case, $u$ is well-defined as a single-valued function only on $H^+$.
But in both cases, $\vv=\vv_u:=\pdf{}{u}$ is well-defined Killing field on all of $\sS^2\times \RR$.
(Note that $\vv_\theta\equiv 0$ on $Z\cup Z^*$.)  

Now consider the vectorfield $w(u)\vv$, where $w:\RR\to \RR$ is given by
\[
  w(u) 
  =
  \begin{cases}
   0 &\text{if $u<a$}, \\
   u-a &\text{if $a \le u \le b$, and} \\
   b-a &\text{if $b< u$.}
  \end{cases}
\]
Then\footnote{The reader may find it helpful to note that in the proof, we are 
  expressing
 $\frac{d}{dt}\area(M_t)$ in two different ways (as a surface integral and as a boundary integral), where
 $M_t$ is a one-parameter family of surfaces $M_0=M$ and with initial velocity vectorfield $w(u)\vv$.}
\begin{align*}
   \int_M \Div_M(w\vv)\,dA
   &=
    \int_M\left( \grad_M(w(u))\cdot \vv + w(u) \Div_M\vv \right)\,dA  
   \\
   &=
    \int_M \left( w'(u) \grad_Mu \cdot \vv + 0 \right) \,dA  
    \\
    &=
    \int_{M\cap (u^{-1}[a,b]) } \nabla_Mu \cdot \vv \, dA
\end{align*}
since $\Div_M\vv\equiv 0$ (because $\vv$ is a Killing vectorfield.)

Let $\ee=\ee_u$ be a unit vectorfield in the direction of $\nabla u$.
Then $\nabla u = |\nabla u|\,\ee$ and $\vv = \pdf{}{u} = |\nabla u|^{-1}\ee$, so
\begin{align*}
 \nabla_Mu \cdot \vv
 &=
 (\nabla u)_M \cdot (\vv)_M \\
 &=
 (|\nabla u|\,\ee)_M \cdot (|\nabla u|^{-1}\ee)_M \\
 &=
 |(\ee)_M|^2  \\
 &=
 1 - (\ee\cdot \nu_M)^2
\end{align*}
where $(\cdot)_M$ denotes the component tangent to $M$ and
where $\nu_M$ is the unit normal to $M$.

Hence we have shown
\begin{equation}\label{e:StretchM}
    \int_M \Div_M(w\vv)\,dA = \int_{M \cap \{a\le u \le b \}} (1 - (\ee_u\cdot \nu_M)^2)\, dA.
\end{equation}

Since $M$ is a minimal surface, 
\[
  \Div_M(V) = \Div_M(V^{\rm tan})
\]
for any vectorfield $V$ (where $V^{\rm tan}$ is the component of $V$ tangent to $M$), so
\begin{equation}\label{e:FirstVariationM}
\begin{aligned}
   \int_M \Div_M(w\vv)\,dA
  &=
     \int_M \Div_M(w\vv)^{\rm tan}\,dA
  \\
  &=
  \int_{\partial M} (w\vv)\cdot \nupt  
  \\
  &\le
  (b-a)\int_{(\partial M) \cap \{u>a\}} \left|\vv_u\cdot \nupt \right|
 \end{aligned}
\end{equation}
Combining~\eqref{e:StretchM} and~\eqref{e:FirstVariationM} gives
\[
 \int_{M \cap \{a \le u \le b\}} (1 - (\ee_u\cdot \nu_M)^2)\, dA
 \le
 (b-a)\int_{(\partial M) \cap \{u>a\}} \left|\vv_u\cdot \nupt \right|
 \]
Adding this inequality
 for $u=z$ to the same inequality for $u=\theta$ (but with $\alpha$ and $\beta$ in place of $a$ and $b$)
 gives
\begin{equation}\label{almost-there}
\begin{aligned}
  &\int_S (2 - (\ee_z\cdot\nu_M)^2 - (\ee_\theta\cdot\nu_M)^2)\,dA
  \\
  &\quad \le
 (b-a)\int_{(\partial M)\cap\{z>a\}}|\vv_z \cdot \nupt|\,ds  \\
  &\quad\quad+  (\beta-\alpha)\int_{(\partial M)\cap \{\theta>\alpha\}} |\vv_\theta\cdot \nupt| \,ds
\end{aligned}
\end{equation}

Let $\ee_\rho$ be a unit vector orthogonal to $\ee_z$ and $\ee_\theta$.  Then
for any unit vector $\nu$, 
\[
   1 = (\ee_z\cdot\nu)^2 + (\ee_\theta\cdot\nu)^2 + (\ee_\rho\cdot \nu)^2,
\]
so the integrand in the left side of~\eqref{almost-there} is $\ge 1 + (\ee_\rho\cdot\nu_M)^2 \ge 1$.
\end{proof}

\newcommand{\diam}{\operatorname{diam}}
\begin{corollary}\label{flux-bounds-corollary}
Let $M$ be a compact minimal surface in $\overline{H^+}$ and let $L$ be the 
the length of $(\partial M)\setminus (Z\cup Z^*)$.
Then
\[
    \area(M\cap K)\le c_H L \diam(K)
\]
for every compact set $K$, where $\diam(K)$ is the diameter of $K$ and
where $c_H$ is a constant depending on the helicoid $H$.
\end{corollary}

The corollary follows immediately from Proposition~\ref{flux-bounds-proposition} because 
    $\vv_z\cdot\nu_M=0$ and $\vv_\theta=0$ on $(\partial M)\cap (Z\cup Z^*)$.


 \section{Nonperiodic genus-$g$ helicoids in $\sS^2\times\RR$: 
      theorem~\ref{theorem1} for $h=\infty$}
\label{nonperiodic-section}

Fix a helicoid $H$ in $\sS^2\times \RR$ with axes $Z$ and $Z^*$ and fix a genus $g$.
For each $h\in (0,\infty]$, consider the class $\Cc(h)=\Cc_g(h)$ of 
embedded, genus $g$ minimal surfaces $M$ in $\sS^2\times [-h,h]$ such that
\begin{enumerate}
\item If $h<\infty$, then $M$ is bounded by two great circles at heights $h$ and $-h$.
   If $h=\infty$, $M$ is properly embedded with no boundary.
\item $M\cap H \cap \{|z|<h\} = (X\cup Z\cup Z^*)\cap \{|z|<h\}$.
\item $M$ is a $Y$-surface.
\end{enumerate}

By the $h<\infty$ case of theorem~\ref{theorem1} (see section~\ref{construction-outline-section}),
the collection $\Cc(h)$ is nonempty for every $h<\infty$.
Here we prove the same is true for $h=\infty$:

\begin{theorem}\label{h-to-infinity-theorem}
Let $h_n$ be a sequence of positive numbers tending to infinity.
Let $M_n\in \Cc(h_n)$.
Then there is a subsequence that converges smoothly and with multiplicity one
to minimal surface $M\in \Cc(\infty)$.  The surface $M$ has bounded curvature,
and each of its two ends is congruent to a helicoid having the same pitch as $H$.
\end{theorem}

\begin{proof}[Proof of theorem~\ref{h-to-infinity-theorem}]
Note that $M_n\cap H^+$ is bounded by two vertical line segments, by the horizontal great circle $X$,
and by a pair of great semicircles at heights $h_n$ and $-h_n$.  It follows that vertical flux is uniformly
bounded.  Thus by corollary~\ref{flux-bounds-corollary},
 for any ball $B$, the area of
\[
   M_n\cap H^+ \cap B
\]
is bounded by a constant depending only on the radius of the ball.
Therefore
the areas of the $M_n$ (which are obtained from the $M_n\cap H^+$ by Schwarz reflection) are also
uniformly bounded on compact sets.
 By the compactness theorem~\ref{general-compactness-theorem},
we can (by passing to a subsequence) assume that the $M_n$ converge as sets to a smooth, properly embedded limit minimal surface $M$.   
According to~\cite{rosenberg2002}*{Theorem~4.3}, every properly embedded minimal surface in $\sS^2\times\RR$
is connected unless it is a union of horizontal spheres.  
Since $M$ contains $Z\cup Z^*$, it is not a union of horizontal spheres, and thus it is connected.
By corollary~\ref{stable-spheres-corollary}, $M$ is unstable.
Hence by the general compactness theorem~\ref{general-compactness-theorem}, the convergence
$M_n\to M$ is smooth with multiplicity one.

Now suppose that each $M_n$ is a $Y$-surface, i.e., that
\begin{enumerate}
\item $\rho_Y$ is an orientation-preserving involution of $M_n$,
\item $M_n/\rho_Y$ is connected, and
\item Each $1$-cycle $\Gamma$ in $M_n$ is homologous (in $M_n$) to $-\rho_Y\Gamma$.
\end{enumerate}

The smooth convergence implies that $\rho_Y$ is also an orientation-preserving involution of $M$.
Since $M$ is connected, so is $M/\rho_Y$.  Also, if $\Gamma$ is a cycle in $M$,
then the smooth, multiplicity one convergence implies that $\Gamma$ is a limit of cycles $\Gamma_n$
in $M_n$.  Thus $\Gamma_n$ together with $\rho_Y\Gamma_n$ bound a region, call it $A_n$,
in $M_n$.  Note that the $\Gamma_n \cup \rho_Y\Gamma_n$ lie in a bounded region in $\sS^2\times \RR$.
Therefore so do the $A_n$ (by, for example, the maximum principle applied to the minimal surfaces $A_n$.)
Thus the $A_n$ converge to a region $A$ in $M$ with boundary $\Gamma + \rho_Y\Gamma$.
This completes the proof that $M$ is a $Y$-surface.

Recall that $Y$ intersects any $Y$-surface transversely, and that twice the number of intersection points is equal
to the genus.  It follows immediately from the smooth convergence (and the compactness of $Y$)
that $M$ has genus $g$.

The fact that $M\cap H=X\cup Z\cup Z^*$ follows immediately from smooth convergence together
with the corresponding property of the $M_n$.

Next we show that $M$ has bounded curvature.
Let $p_k\in M$ be a sequence of points such that the curvature of $M$ at $p_k$ tends
to the supremum.  Let $f_k$ be a screw motion such that $f_k(H)=H$ and such that $z(f(p_k))=0$.
The surfaces $f_k(M)$ have areas that are uniformly bounded on compact sets. (They inherit those
bounds from the surfaces $M_n$.)  Thus exactly as above, by passing to a subsequence, we get
smooth convergence to a limit surface.  It follows immediately that $M$ has bounded curvature.

Since $M$ is a minimal embedded surface of finite topology containing $Z\cup Z^*$, each of its
two ends is asymptotic to a helicoid by~\cite{hoffman-white-axial}.
Since $M\cap H=Z\cup Z^*$, those limiting helicoids must have the same pitch as $H$.
(If this is not clear, observe that the intersection of two helicoids with the same axes but different pitch
contains an infinite collection of equally spaced great circles.)
\end{proof}

\addtocontents{toc}{\SkipTocEntry}
\subsection{Proof of theorem~\ref{theorem1} for $h=\infty$}
The non-periodic case of theorem~\ref{theorem1} follows immediately from the 
periodic case together with theorem~\ref{h-to-infinity-theorem}.
The various asserted properties of the non-periodic examples follow from the corresponding
properties of the periodic examples together with the smooth convergence in
 theorem~\ref{h-to-infinity-theorem}, except for the noncongruence properties, which are
 proved in appendix~\ref{noncongruence-appendix}. \qed


\section{Convergence to Helicoidal Surfaces in $\RR^3$}\label{R3-section}

In the section, we study the behavior of genus-$g$ helicoidal surfaces
 $\sS^2(R)\times\RR$ as $R\rightarrow\infty$.
 The results in this section will be used in section~\ref{Proof_of_theorem_2} 
 to prove theorem~\ref{theorem2} of section~\ref{section:main-theorems}
(Theorem~\ref{theorem2}  is restated in section~\ref{Proof_of_theorem_2}  as theorem~\ref{R3limits2}.)

 We will identify $\sS^2(R)$ with $\RR^2 \cup\{\infty\}$ by stereographic projection, and therefore
  $\sS^2(R)\times \RR$ with 
  \[
     (\RR^2\cup \{\infty\})\times \RR = \RR^3 \cup (\{\infty\}\times\RR) = \RR^3\cup Z^*.
  \]

Thus we are working with $\RR^3$ together with a vertical axis $Z^*$ at infinity.
The Riemannian metric is
\begin{equation}\label{metric}
    \left(\frac{4R^2}{4R^2+ x^2 + y^2}\right)^2(dx^2+dy^2) + dz^2.
\end{equation}
In particular, the metric coincides with the Euclidean metric along the $Z$ axis.
Inversion in the cylinder
\begin{equation} \label{C2r}
C_{2R}=\{(x,y,z)\,:\, x^2+y^2=(2R)^2\}
\end{equation}
 is an isometry of \eqref{metric}. 
Indeed, $C_{2R}$ corresponds to $E\times\RR$, 
where $E$ is the equator of $\sS^2\times\{0\}$ with respect 
to the antipodal points $O$ and $O^*$.
 We also note for further use that \begin{equation} \label{Cyldist}
\dist_R(C_{2R}\,,\,Z)= \pi R/2,
\end{equation}
where $\dist_R(\cdot,\cdot)$ is the distance function associated to the metric \eqref{metric}.

We fix a genus $g$ and choose a helicoid
 $H\subset \RR^3$  with axis $Z$ and containing $X$.
(Note that it is a helicoid for all choices of $R$.)   As usual, let $H^+$ is the component of $\RR^3\setminus H$
containing $Y^+$, the positive part of the $y$-axis.
Let $M$ be one of the nonperiodic, genus-$g$ examples constructed described in 
 theorem~\ref{theorem1}.  Let 
 \[
     S=\interior (M\cap H^+).
 \]
According to theorem~\ref{theorem1}, $M$ and $S$ have the following properties:
\begin{enumerate}
\item\label{example-def-one} $S$  is a smooth, embedded $Y$-surface in $H^+$
that intersects $Y^+$ in exactly $g$ points, 
\item\label{example-def-two} The boundary\footnote{Here we are regarding $M$ and $S$ as  subsets of $\RR^3$
with the metric~\eqref{metric}, so $\partial S$ is $X\cup Z$ and not $X\cup Z\cup Z^*$.}
of $S$ is $X\cup Z$.
\item\label{example-def-three} $M=\overline{S}\cup \rho_Z\overline{S} \cup Z^*$ is a 
  smooth surface that is minimal
  with respect to the metric~\eqref{metric}.
\end{enumerate}

\newcommand{\hh}{\eta}   
\begin{definition}\label{example-definition}
An {\em example} is a triple $(S,\hh,R)$ with $\hh>0$ and $0<R<\infty$
such that $S$ satisfies~\eqref{example-def-one}, 
\eqref{example-def-two}, and~\eqref{example-def-three}, where
$H$ is the helicoid in $\RR^3$ that has axis $Z$, that contains $X$, and
that has vertical distance between successive sheets equal to $\hh$.
In the terminology of the previous sections, $H$ is the helicoid of pitch $2\hh$.
\end{definition}

\stepcounter{theorem}
\addtocontents{toc}{\SkipTocEntry}
\subsection{Convergence Away from the Axes}

Until \ref{nearZ} it will  be convenient to work not in $\RR^3$ but rather
in the universal cover of $\RR^3\setminus Z$,  still with the Riemannian metric~\eqref{metric}.
Thus the angle function $\theta(\cdot)$ will be well-defined and single valued. 
However, we normalize the angle function so that $\theta(\cdot)=0$ on $Y^+$.
(In the usual convention for cylindrical coordinates, $\theta(\cdot)$ would be $\pi/2$ on $Y^+$.)
Thus $\theta=-\pi/2$ on $X^+$ and $\theta=\pi/2$ on $X^-$.

Of course $Z$ and $Z^*$ are not in the universal cover, but $\dist(\cdot, Z)$
and $\dist(\cdot, Z^*)$ still make sense.

Since we are working in the universal cover, each vertical line intersects $H^+$ in a single segment
of length $\hh$.

\begin{theorem}[First Compactness Theorem]\label{first-compactness-theorem} 
Consider a sequence $(S_n,\hh_n,R_n)$ of examples with $R_n$ bounded away from $0$ 
and with $\hh_n\to 0$.
Suppose that 

\begin{enumerate}[\upshape (*)]
\item each $S_n$ is graphical in some nonempty, open cylindrical region $U\times \RR$ such that
$\theta(\cdot)>\pi/2$ on $U\times\RR$. In other words, every vertical line in $U\times\RR$ intersects
$M_n$ exactly once.
\end{enumerate}
Then after passing to a subsequence, the $S_n$ converge smoothly away from a discrete set
$K$ to the surface $z=0$.  The convergence is with multiplicity one where $|\theta(\cdot)|>\pi/2$ and with
multiplicity two where $|\theta(\cdot)|<\pi/2$.  

Furthermore, the singular set $K$ lies in the region $|\theta(\cdot)|\le \pi/2$.
\end{theorem}

\begin{remark}\label{compactness-remark}
Later (Corollary~\ref{hypothesis-satisfied-corollary} and Corollary~\ref{singularities-in-Y-corollary}) 
we will show the hypothesis \thetag{*} is not needed 
and  that the singular set $K$ lies in $Y^+$.
\end{remark}

\begin{proof}
By passing to a subsequence and scaling,  we can assume that the $R_n$ converge to a limit $R\in [1,\infty]$.
Note that the $H_n^+$ converge as sets to the surface $\{z=0\}$ in the universal cover of $\RR^3\setminus Z$.
Thus, after passing to a subsequence, the $S_n$ converge as sets to a closed subset 
of the surface $\{z=0\}$.    By standard estimates for minimal graphs, the convergence is smooth
(and multiplicity one) in $U\times\RR$.   Thus the area blowup set
\[
  Q: = \{ q: \text{$\limsup_n \area(S_n\cap \BB(q,r)) =\infty$ for all $r>0$} \}
\]
is contained in $\{z=0\}\setminus U$ and is therefore a proper subset of $\{z=0\}$.
The constancy theorem for area blow up sets~\cite{white-controlling-area}*{theorem~4.1} 
states that the area blowup set of a sequence
of minimal surfaces cannot be a nonempty proper subset of a smooth, connected two-manifold, provided
the lengths of the boundaries are uniformly bounded on compact sets.   Hence $Q$ is empty.
   That is, the areas
of the $S_n$ are uniformly bounded on compact sets.

Thus by the general compactness theorem~\ref{general-compactness-theorem},
after passing
to a subsequence, the $S_n$ converge smoothly away from a discrete set $K$
to a limit surface $S'$ lying in $\{z=0\}$.
The surface $S'$ has some constant multiplicity  in the region where $\theta(\cdot)>\pi/2$.  
Since the $S_n\cap (U\times\RR)$ are graphs,
that multiplicity must be $1$.  By $\rho_Y$ symmetry, the multiplicity is also $1$
where $\theta<-\pi/2$.
  Since each $S_n$ has boundary $X$, the multiplicity
of $S'$ where $|\theta(\cdot)| <\pi/2$ must be $0$ or $2$.

Note that $\tilde S_n:=S_n\cap \{|\theta(\cdot)|<\pi/2\}$ is nonempty and lies in the solid cylindrical region
\begin{equation}\label{the-region}
    \{|\theta(\cdot)|\le \pi/2\} \cap \{ |z|\le 2\hh_n\}
\end{equation} 
and that $\partial \tilde S_n$ lies on the cylindrical, vertical edge of that region.  
It follows (by theorem~\ref{all-or-nothing-theorem} in appendix~\ref{hemisphere-appendix}) 
that for $\hh_n/r_n$ sufficiently small,
every vertical line that intersects the region~\eqref{the-region} is at distance at most $4\hh_n$
from $S_n$.
Thus the limit of the $\tilde S_n$ as sets is all of $\{z=0\}\cap \{|\theta(\cdot)|\le \pi/2\}$,
and so the multiplicity there is two, not zero.

Since the convergence $S_n\to S'$ is smooth wherever $S'$ has multiplicity $1$ (either
by the General Compactness Theorem~\ref{general-compactness-theorem} 
or by the Allard Regularity Theorem), 
     $|\theta(\cdot)|$ must be $\le \pi/2$ at
each point of $K$. 
\end{proof}


\begin{theorem}\label{second-compactness-theorem}
Let $(S_n, \hh_n, R_n)$ be a sequence of examples with $R_n\ge 1$ and with $\hh_n\to 0$.
Let $f_n$ be the screw motion through angle $\alpha_n$ that maps $H_n^+$ to itself, 
and assume that $|\alpha_n|\to \infty$.
Let $S_n' = f_n(S_n)$.   Suppose each $S_n'$ is graphical in some nonempty open cylinder $U\times\RR$.
Then the $S_n'$ converge smoothly (on compact sets) with multiplicity one to the surface $\{z=0\}$.
\end{theorem}

The proof is almost identical to the proof of theorem~\ref{first-compactness-theorem}, so we omit it.

\begin{theorem}\label{graphical-theorem}
For every genus $g$ and  angle $\alpha>\pi/2$, there is a $\lambda<\infty$ with the following
property.  If $(S,\hh,R)$ is a genus-$g$ example (in the sense of definition~\ref{example-definition}) with 
\[
    \dist(Z,Z^*) = \pi R > 4\lambda \hh,
\]
then $S$ is graphical in the region
\[
   Q(\lambda \hh, \alpha):= \{|\theta(\cdot)| \ge \alpha, \, \dist(\cdot, Z\cup Z^*) \ge \lambda \hh \}.
\]
\end{theorem}

\begin{proof}
Suppose the result is false for some $\alpha>\pi/2$, and let $\lambda_n\to\infty$.
Then for each $n$, there is an example $(S_n,\hh_n,r_n)$ such that
\begin{equation}
\dist_n(Z,Z^*) > 4\lambda_n \hh_n
\end{equation}
and such that $S_n$  is not a graphical in $Q(\lambda_n\hh_n,\alpha)$.

Here $\dist_n(\cdot,\cdot)$ denotes distance with 
respect the metric that comes from $\sS^2(R_n)\times \RR$.
However, henceforth we will write $\dist(\cdot, \cdot)$ instead of $\dist_n(\cdot,\cdot)$ to reduce
notational clutter.

Since the ends of $M_n=\overline{S_n\cup\rho_ZS_n}$ 
are asymptotic to helicoids as $z\to\pm\infty$, 
note that $S_n$ is graphical in $Q(\lambda_n \hh_n, \beta)$ for all sufficiently large $\beta$.
Let $\alpha_n\ge \alpha$ be the largest angle such that $S_n$ is 
not graphical in $Q(\lambda_n\hh_n, \alpha_n)$.
Note that there must be a point $p_n\in S_n$
such that
\begin{align}
 \theta(p_n)&=\alpha_n, \\ 
 \dist(p_n,Z\cup Z^*) &> \lambda_n \hh_n, 
 \end{align}
 and such that $\Tan(S_n,p_n)$ is vertical.
 Without loss of generality, we may assume (by scaling) that $\dist(p_n, Z\cup Z^*)=1$.
 In fact, by symmetry of $Z$ and $Z^*$, we may assume that 
 \[
   1 = \dist(p_n,Z) \le \dist(p_n, Z^*),
 \]
 which of course implies that $\pi R_n = \dist(Z,Z^*) \ge 2$, and therefore that
\[
  \lambda_n\hh_n \le \frac12.
\]
By passing to a further subsequence, we can assume that
\[
  \lambda_n \hh_n \to \mu\in [0,\frac12].
\]
Since $\lambda_n\to\infty$, this forces
\[
  \hh_n\to 0.
\]
We can also assume that
\[
\alpha_n\to \talpha \in [\alpha, \infty].
\]

Case 1:  $\talpha<\infty$.
Then the $p_n$ converge to a point $p$ with $\theta(p)=\talpha$
and with $\dist(p,Z)=\dist(p,Z\cup Z^*)=1$.

Note that $M_n$ is graphical in the region $Q(\lambda_n \hh_n, \alpha_n)$,
and that those regions converge to $Q(\mu, \alpha)$.  Thus by the Compactness
Theorem~\ref{first-compactness-theorem}, 
the $S_n$ converge smoothly and with multiplicity one to $\{z=0\}$
in the region $|\theta(\cdot)|>\pi/2$.   But this is a contradiction since $p_n\to p$,
which is in that region, and since $\Tan(S_n,p_n)$ is vertical.

Case 2: Exactly as in case 1, except that we apply a screw motion $f_n$ to $M_n$
such that $\theta(f_n(p_n))=0$. (We then use theorem~\ref{second-compactness-theorem}
 rather than Theorem~\ref{first-compactness-theorem}.)
\end{proof}

\begin{corollary}\label{hypothesis-satisfied-corollary}
The hypothesis~(*) in Theorems~\ref{first-compactness-theorem} and~\ref{second-compactness-theorem}
 is always satisfied provided $n$ is sufficiently large.
\end{corollary}

\stepcounter{theorem}
\addtocontents{toc}{\SkipTocEntry}
\subsection{Catenoidal Necks} The next theorem shows that, in the compactness theorem~\ref{first-compactness-theorem}, any point away from $Z\cup Z^*$
where the convergence is not smooth must lie on $Y$, and that near such a point,
the  $S_n$ have small catenoidal necks.

\begin{theorem}\label{limit-is-a-catenoid-theorem}
Let $(S_n,\hh_n,R_n)$ be a sequence of examples and
 $p_n\in S_n$ be a sequence of points  such that
\begin{equation}\label{slope-bigger-than-delta-hypothesis}
\slope(S_n,p_n)\ge \delta>0, 
\end{equation}
(where $\slope(S_n,p_n)$ is the slope of the tangent plane to $S_n$ at $p_n$) and such that
\begin{equation}\label{away-from-axes-hypothesis}
\frac{\dist(p_n,{Z\cup Z^*})}{\hh_n}  \to \infty.
\end{equation}
Then there exist positive numbers $c_n$ such that (after passing to a subsequence) the surfaces
\begin{equation}\label{rescaled-catenoidish}
   \frac{S_n-p_n}{c_n}
\end{equation}
converge to a catenoid in $\RR^3$.  
The waist of the catenoid is a horizontal circle, and the line $(Y^+-p_n)/c_n$ converges to a line
that intersects the waist  in two diametrically opposite points.

Furthermore, 
\begin{equation}\label{h_n-relatively-large}
\frac{\hh_n}{c_n} \to \infty. 
\end{equation}
and
\begin{equation}\label{h_n-relatively-large2}
\frac{\dist(p_n, S_n\cap Y^+)}{\hh_n} \to 0.  
\end{equation}
\end{theorem}

\begin{proof}
By scaling and passing to a subsequence, we may assume that
\begin{equation}\label{distance-to-Z-normalized}
   1 = \dist(p_n, Z) \le \dist(p_n, Z^*).
\end{equation}
and that $\theta(p_n)$ converges to a limit $\alpha\in [-\infty,\infty]$.
By~\eqref{away-from-axes-hypothesis} (a statement that is scale invariant)
 and by~\eqref{distance-to-Z-normalized}, $\hh_n\to 0$.
Thus by~theorem~\ref{graphical-theorem} (and standard estimates for minimal graphs), $|\alpha|\le \pi/2$.

First we prove that there exist $c_n\to 0$ such that the surfaces $(S_n-p_n)/c_n$
converge subsequentially to a catenoid with horizontal ends.

{\bf Case 1}: $|\alpha|=\pi/2$.  By symmetry, it suffices to consider the case $\alpha=\pi/2$.
Let $\widetilde S_n$ be obtained from $S_n$ by Schwartz reflection about $X^-$.

By the last sentence of the general compactness theorem~\ref{general-compactness-theorem}, 
there exist numbers $c_n\to 0$
such that (after passing to a subsequence) the surfaces
\[
  \frac{\widetilde{S}_n - p_n}{c_n}
\]
converge smoothly to a complete, non-flat, properly embedded minimal surface $\widetilde{S}\subset \RR^3$ of finite
total curvature whose ends are horizontal.   
By proposition~\ref{genus-0-proposition}, $\widetilde{S}$ has genus $0$.  
By  a theorem of Lopez and Ros\cite{lopez-ros}, the only 
nonflat, properly embedded minimal surfaces in $\RR^3$ with genus zero and finite total curvature are the catenoids.
Thus $\tilde S$ is a catenoid.   Note that 
\[
   \frac{S_n-p_n}{c_n}
\]
converges to a portion $S$ of $\widetilde{S}$.  Furthermore, $S$ is either all of $\widetilde{S}$, or 
it is a portion of $\widetilde S$ bounded by a horizontal line $\widetilde{X}= \lim_n ((X^- - p_n)/c_n$ 
in $\widetilde{S}$.  
Since catenoids contain no lines, 
in fact $S=\widetilde{S}$ is a catenoid.

{\bf Case 2}:  $|\alpha|<\pi/2$.
By the last statement of the general compactness theorem~\ref{general-compactness-theorem}, 
there are $c_n>0$ tending to $0$ such
that (after passing to a subsequence) the surfaces
\[
 \frac{S_n-p_n}{c_n}
\]
converge smoothly to a complete, nonflat, embedded minimal surface $S\subset \RR^3$ of finite
total curvature with ends parallel to horizontal planes.  By monotonicity, 
\begin{equation}\label{monotonicity}
  \limsup_n \frac{\area\left(\frac{S_n-p_n}{c_n} \cap \BB(0,\rho)\right)}{\pi \rho^2} \le 2
\end{equation}
for all $\rho>0$.  Thus $S'$ has density at infinity $\le 2$, so it has at most two ends.  
If it had just one end, it would be a plane. But it is not flat, so that is impossible. Hence it has two ends. 
By a theorem of Schoen~\cite{schoen-uniqueness}, 
a properly embedded minimal surface in $\RR^3$ of 
with finite total curvature and two ends must be  a catenoid.

This completes the proof that (after passing to a subsequence) the surfaces $(S_n-p_n)/c_n$
converge to a catenoid $S$ with horizontal ends.

Note that for large $n$, there is a simple closed geodesic $\gamma_n$ in $S_n$
such that $(\gamma_n-p_n)/c_n$ converges to the waist of the catenoid $S$.
Furthermore, $\gamma_n$ is unique in the following sense: if $\gamma_n'$ is a
simple closed geodesic in $S_n$ that converges to the waist of the catenoid $S$, then
$\gamma_n'=\gamma_n$ for all sufficiently large $n$.  (This follows from the implicit
function theorem and the fact that the waist $\gamma$ of the catenoid is non-degenerate
as a critical point of the length function.)

\begin{claim} $\rho_Y\gamma_n=\gamma_n$ for all sufficiently large $n$.
\end{claim}

\begin{proof}[Proof of claim]
 Suppose not.  
Then (by passing to a subsequence) we can assume that $\gamma_n\ne\rho_Y\gamma_n$ for all $n$.
Thus (passing to a further subsequence) the curves $(\rho_Y\gamma_n-p_n)/c_n$ do one of the following:
(i) they converge to $\gamma$, (ii) they converge to another simple closed geodesic in $S$ having the same
length as $\gamma$, or (iii) they diverge to infinity.  
Now (i) is impossible by the uniqueness of the $\gamma_n$.  Also, (ii) is impossible because the waist
$\gamma$ is the only simple closed geodesic in the catenoid $S$.
Thus (iii) must hold: the curves $(\rho_Y\gamma_n - p_n)/c_n$ diverge to infinity.

Since $S_n$ is a $Y$-surface, $\gamma_n$ together with $\rho_Y\gamma_n$ bound a region $A_n$
in $S_n$.  By the maximum principle, $\theta(\cdot)$ restricted to $A_n$ has its maximum on one of the
two boundary curves $\gamma_n$ and $\rho_Y\gamma_n$ and (by symmetry) its minimum on the other.
(Note that the level sets of $\theta$ are totally geodesic and therefore minimal.)

\newcommand{\hA}{\hat{A}}
By passing to a subsequence, we can assume that the regions $(A_n-p_n)/c_n$ converges to a subset $\hA$
of the catenoid $S$.
Note that $\hA$ is the closure of one of the components of $S\setminus \Gamma$.  (This is because one of the two boundary components of $(A_n-p_n)/c_n$,
namely $(\gamma_n-p_n)/c_n$, converges to the waist of the catenoid, 
whereas the other boundary component, namely $(\rho_Y\gamma_n-p_n)/c_n$, diverges to infinity.)
The fact that $\theta | \overline{A_n}$ attains it maximum on $\gamma_n$ implies that
there is a linear function $L$ on $\RR^3$ with horizontal gradient such that $L|  \hA$
attains its maximum on the waist $\gamma=\partial \hA$.  But that is impossible since the catenoid $S$
has a horizontal waist.  This proves the claim.
\end{proof}

Since each $\gamma_n$ is $\rho_Y$-invariant (by the claim), it follows that 
the waist $\gamma$ is invariant under $180^\circ$ rotation about the line $Y'$,
where $Y'$ is a subsequential limit of the curves $(Y_n-p_n)/c_n$.
Since $\gamma$ is a horizontal circle, $Y'$ must be a line that bisects the circle.
Thus
\begin{equation}\label{normalized-distance-to-Y}
   \frac{\dist(p_n, S_n\cap Y)}{c_n} \to \dist(O, S\cap Y') < \infty.
\end{equation}

Note that $\hh_n/c_n \to \infty$, since if it converged to a finite limit, 
then the regions $(H_n^+ - p_n) /c_n$ would converge to a horizontal slab of finite
thickness and the catenoid $S$ would be contained in that slab, a contradiction.
This completes the proof of~\eqref{h_n-relatively-large}.

Finally,~\eqref{h_n-relatively-large2} follows immediately 
from~\eqref{h_n-relatively-large} and~\eqref{normalized-distance-to-Y}.
\end{proof}

\begin{corollary}\label{singularities-in-Y-corollary}
The singular set $K$ in Theorem~\ref{first-compactness-theorem} is a finite subset of $Y^+$.
In fact (after passing to a subsequence), $p\in K$ if and only if there is a sequence $p_n\in Y^+\cap S_n$
such that $p_n\to p$.
\end{corollary}

The following definition is suggested by theorem~\ref{limit-is-a-catenoid-theorem}:

\begin{definition}\label{neck-definition}
Let $(S,\hh,R)$ be an example (as in definition~\ref{example-definition}).  
Consider the set of points of $S$ at which the tangent plane is vertical.
A {\em neck} of $S$ is a connected component of that set consisting of a simple closed
curve that intersects $Y^+$ in exactly two points.
The {\em radius}
of the neck is half the distance between those two points, and the {\em axis} of the neck
is the vertical line that passes through the midpoint of those two points.
\end{definition}


\begin{theorem}\label{graph-form-theorem}
Suppose that $(S,\hh,R)$ is an example (as in definition~\ref{example-definition}) and that $V$ is a vertical line.
If $V$ is not too close to $Z\cup Z^*$ and also not too close to any neck
axis, then $V$ intersects $M$ in at most two points, and the tangent planes
to $M$ at those points are nearly horizontal.   
Specifically, 
for every $\eps>0$, there is a $\lambda$ (depending only on genus and $\eps$) 
with the following properties.
Suppose that
\[
   \frac{\dist(V,Z\cup Z^*)}{\hh} \ge \lambda,
\]
and that for every neck axis $A$, either
\[
    \frac{\dist(V,A)}{r(A)} \ge \lambda
\]
(where $r(A)$ is the neck radius)
or
\[
    \frac{\dist(V,A)}{\hh} \ge 1.
\]
Then
\begin{enumerate}[\upshape (i)]
\item\label{slope-small-item}The slope of the tangent plane at each point in $V\cap M$ is $<\eps$, and
\item\label{one-or-two-item} $V$ intersects $\overline{S}$ in exactly one point if $\theta(V)>\pi/2$ and in exactly
two points if $\theta(V)\le \pi/2$.
\end{enumerate}
\end{theorem}

\begin{proof} Let us first prove that there is a value $\Lambda<\infty$  of $\lambda$ such that 
assertion~\eqref{slope-small-item} holds.  Suppose not.   
Then there exist examples $(S_n,\hh_n,R_n)$
and vertical lines $V_n$ such that
\begin{equation}\label{thing-one}
   \frac{\dist(V_n,Z\cup Z^*)}{\hh_n} \ge \lambda_n\to \infty,
\end{equation}
and such that
\begin{equation}\label{thing-two}
   \frac{\dist(V_n,A)}{r(A)} \ge \lambda_n \quad\text{or}\quad  \frac{\dist(V_n,A)}{\hh_n} \ge 1
\end{equation}
for every neck axis $A$ of $S_n$, 
but such that $V_n\cap S_n$ contains a point $p_n$ at which the slope of $M_n$ is $\ge \eps$.

Note that~\eqref{thing-one} and~\eqref{thing-two} are scale invariant.  We can can choose coordinates
so that $p_n$ is at the origin and, by Theorem~\ref{limit-is-a-catenoid-theorem}, 
we can choose scalings so that the $S_n$
converge smoothly to a catenoid in $\RR^3$.  Let $A'$ be the axis of the catenoid, $r(A')$ be the
radius of the waist of the catenoid, and $V'$ be the vertical line through the origin.
Then $\dist(A',V')$ is finite, $r(A')$ is finite and nonzero, and $\hh_n\to\infty$ 
by~\eqref{h_n-relatively-large}.  Thus if $A_n$ is the neck axis of $S_n$ that
converges to $A'$, then
\[
   \lim_n \frac{\dist(V_n,A_n)}{r(A_n)}  <\infty
   \quad\text{and}\quad
   \lim_n\frac{\dist(V_n,A_n)}{\hh_n}  =  0,
\]
contradicting~\eqref{thing-two}.  This proves
that there is a value
of $\lambda$, call it $\Lambda$, that makes assertion~\eqref{slope-small-item} of the theorem true.

Now suppose that there is no $\lambda$ that makes assertion~\eqref{one-or-two-item} true.
Then there is a sequence $\lambda_n\to\infty$, a sequence of examples $(S_n,\hh_n,R_n)$,
and a sequence of vertical lines $V_n$ such that~\eqref{thing-one} and~\eqref{thing-two}
hold, but such that $V_n$ does not intersect $M_n$ in the indicated number of points.
By scaling, we may assume that 
\[
   1 = \dist(V_n, Z)\le \dist(V_n, Z^*),
\]
which implies that $R_n$ is bounded below and (by~\eqref{thing-one}) that $\hh_n\to 0$.
We may also assume that $\theta(V_n)\ge 0$, and
that each $\lambda_n$ is greater than $\Lambda$.
Thus $V_n$ intersects $S_n$ transversely.
For each fixed $n$, if we move $V_n$ in such a way that $\dist(V_n,Z)=1$ stays constant and
that $\theta(V_n)$ increases, then~\eqref{thing-one} and~\eqref{thing-two} remain true, 
so $V_n$ continues to be transverse to $M_n$.
Thus as we move $V_n$ in that way, the number of points in 
   $V_n\cap \overline{S_n}$ does not change unless 
$V_n$ crosses $X$, so we may assume that $\theta(V_n)\ge \pi/4$.
But now Theorem~\ref{first-compactness-theorem} and Remark~\ref{compactness-remark} 
imply that $V_n\cap \overline{S_n}$ has
the indicated number of intersections, contrary to our assumption that it did not.
\end{proof}

\begin{corollary}\label{form-of-graph-corollary}
Let $\eps>0$ and $\lambda>1$ be as in Theorem~\ref{graph-form-theorem}, and 
let $(S,\hh,R)$ be an example.
Consider the following cylinders: 
vertical solid cylinders of radius $\lambda \hh$ about $Z$ and $Z^*$, 
and for each neck 
axis\,\footnote{See~\ref{neck-definition} for the definition of ``neck axis".} $A$ of $S$
with $\dist(A,Z\cup Z^*)>(\lambda-1) \hh$, a vertical solid cylinder with axis $A$ 
and radius $\lambda r(A)$.
 Let $J$ be the union of those 
cylinders.
Then $S\setminus J$ consists of two components, one
of which can be parametrized as
\[
   \{ (r\cos\theta, r\sin\theta, f(r,\theta)):   r>0,\, \theta\ge -\pi/2\}  \setminus J
\]
where $f(r,-\pi/2) \equiv 0$ and where
\begin{equation}\label{trapped}
       \hh\left( \frac{\theta}{\pi} - \frac{1}2\right) \le f(r,\theta) \le \hh\left( \frac{\theta}{\pi} + \frac{1}2\right).
\end{equation}
\end{corollary}

Of course, by $\rho_Y$ symmetry, the other component of $S\setminus J$ can be written
\[
   \{ (r\cos\theta, r\sin\theta, -f(r,-\theta)):   r>0,\quad \theta\le \pi/2\}  \setminus J.
\]

The inequality~\eqref{trapped} expresses the fact that $S$ lies in $H^+$.
Note that in Corollary~\ref{form-of-graph-corollary}, because we are working in the universal cover of
 $\RR^3\setminus Z$,
each vertical cylinder about a neck axis in the collection $J$ 
intersects $H^+$ in a single connected component. (If we were working in $\RR^3$, it would intersect
$H^+$ in infinitely many components.)   Thus the portion of $S$ that lies in such a cylinder is a single
catenoid-like annulus.  If we were working in $\RR^3$, the portion of $S$ in such a cylinder would
be that annulus together with countably many disks above and below it.

\begin{remark}\label{form-of-graph-remark}
In Corollary~\ref{form-of-graph-corollary}, 
the function $f(r,\theta)$ is only defined for $\theta\ge -\pi/2$. 
It is positive
for $\theta> -\pi/2$ and it vanishes where $\theta=\pi/2$.   Note that we can extend $f$ by Schwarz reflection
to get a function $f(r,\theta)$ defined for all $\theta$:
\begin{align*}
  f(r,\theta) &= f(r,\theta) \qquad\text{for $\theta\ge -\pi/2$, and} \\
  f(r,\theta) &= -f(r, -\pi-\theta)  \qquad\text{for $\theta < -\pi/2$.}
\end{align*}
Corollary~\ref{form-of-graph-corollary} 
states that, after removing the indicated cylinders, we can express $S$ (the portion of $M$
in $H^+$) as the union of two multigraphs: the graph of the original, unextended $f$ together with the image of that graph under
$\rho_Y$.  Suppose we remove from $M$ those cylinders together with their images under $\rho_Z$.
Then the remaining portion of $M$ can be expressed as the the union of two multigraphs:
the graph of the extended function $f$ (with $-\infty<\theta<\infty$) together with the image of that graph 
under $\rho_Y$.
\end{remark}

\begin{remark}
Note that $H\setminus (Z\cup X)$ consists of four quarter-helicoids, two of which are
described in the universal cover of $\RR^3\setminus Z$ by
\[
    z =   \frac{\hh}{\pi} \left( \theta + \frac{\pi}2\right), \quad (\theta\ge  -\pi/2)
\]
and
\[
   z =  \frac{\hh}{\pi}\left( \theta - \frac{\pi}2\right), \quad (\theta \le \pi/2).
\]
(As in the rest of this section, we are measuring $\theta$ from $Y^+$ rather than from $X^+$.)
These two quarter-helicoids overlap only in the region $-\pi/2<\theta<\pi/2$:
a vertical line in that region intersects both quarter-helicoids in points that are
distance $\hh$ apart, whereas any other vertical line intersects only one
of the two quarter-helicoids.
Roughly speaking, theorem~\ref{graph-form-theorem} 
and corollary~\ref{form-of-graph-corollary} say
that if $(S,\hh,R)$ is an example with $R/\hh$ large, then $S$ must be obtained from these two
quarter-helicoids by joining them by catenoidal necks away from $Z$ and 
in some possibly more complicated way near $Z$.   The catenoidal necks lie
along the $Y$-axis.

Figure~\ref{unrolled-figure} illustrates the intersection of $M=S\cup\rho_Z S$ with a vertical cylinder with axis  $Z$.  The shaded region is the intersection of the cylinder with $H^+$. The intersections of the cylinder with the quarter-helicoids are represented by halflines  on the boundary of  the shaded region:  $\theta\geq - \pi/$2 on top of the shaded region, and  $\theta\leq \pi/2$ on the bottom. The radius of the cylinder is chosen so that the cylinder passes though a catenoidal neck of S that can be thought of as joining the quarter-helicoids, allowing $S$ to make a transition from approximating one quarter helicoid to approximating to the other. The transition takes place in the region $-\pi/2\leq\theta\leq \pi/2$.
\end{remark}


\begfig
\begin{center}
\includegraphics[height=40mm]{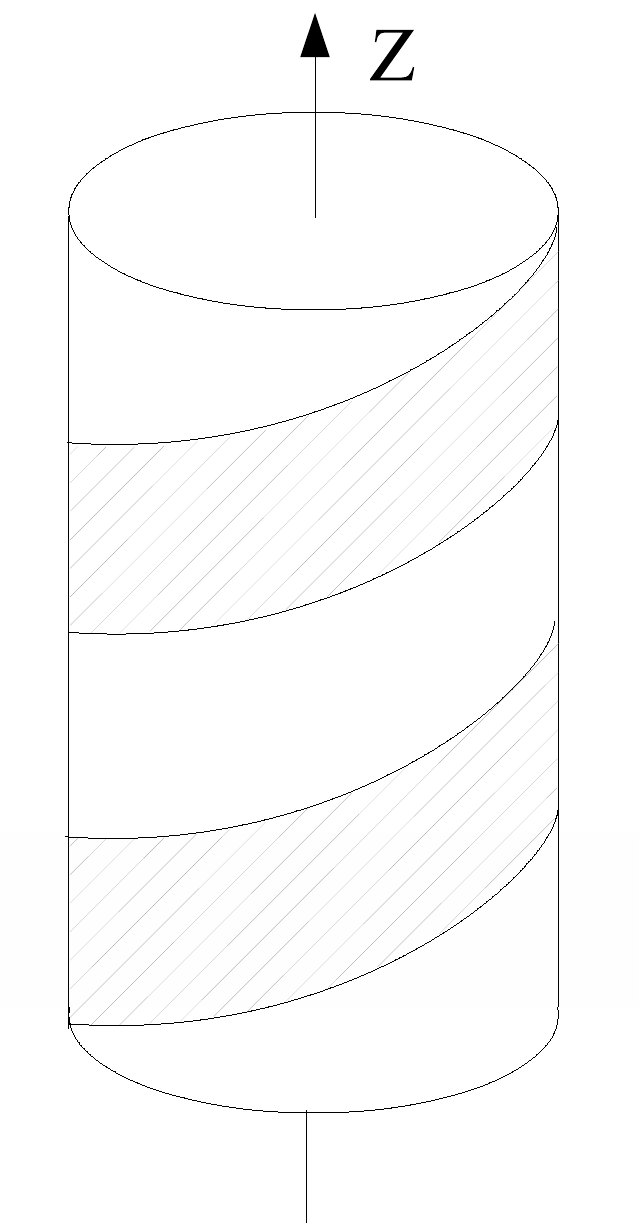}
\hspace{1cm}
\includegraphics[height=40mm]{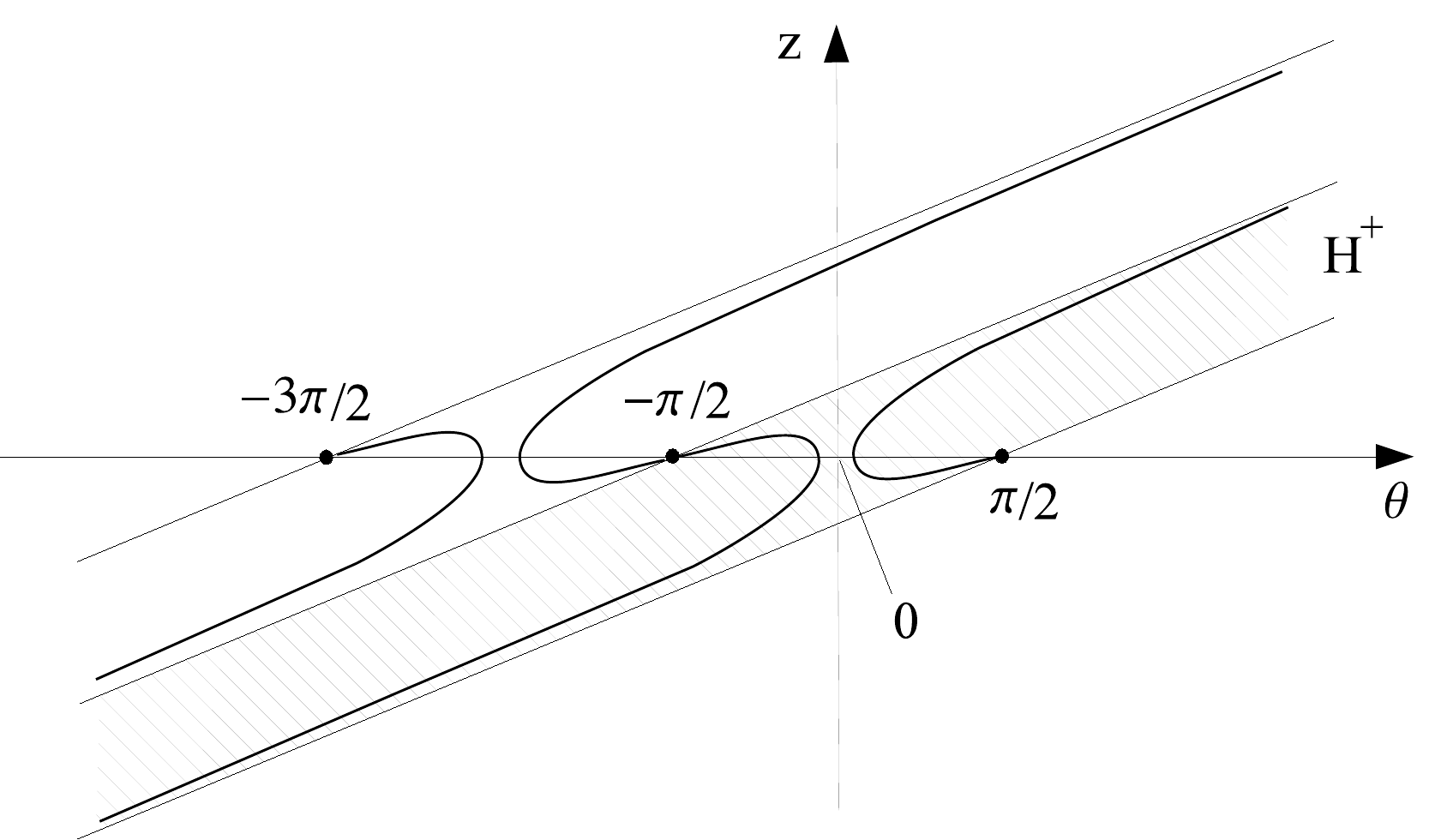}
\end{center}
\caption{{\bf Left}: the shaded region is the intersection of $H^+$ with the vertical cylinder of axis $Z$ and
radius $r$. {\bf Right}: intersection of $M$ with the same cylinder, unrolled in the plane. We use cylindrical coordinates $(r,\theta,z)$, with $\theta=0$ being the positive $Y$-axis.
The radius $r$ is chosen
so that the cylinder intersects a catenoidal neck. The positive $X$-axis intersects the cylinder at the point $(\theta,z)=(-\pi/2,0)$. The negative $X$-axis intersects the cylinder at the point $(\theta,z)=(\pi/2,0)$, which is the same as the point $(-3\pi/2,0)$ on the cylinder.}
\label{unrolled-figure}
\endfig


\stepcounter{theorem}
\addtocontents{toc}{\SkipTocEntry}
\subsection{Behavior near $Z$}\label{nearZ}

In this section, 
we consider examples (see definition~\ref{example-definition}) $(S_n, 1, R_n)$ 
with $\eta=1$ fixed and with $R_n\to\infty$.
We will work in $\RR^3$ (identified with $(\sS^2(R_n)\times \RR)\setminus Z^*$ by stereographic
projection as described at the beginning of section~\ref{R3-section}), 
rather than in the universal cover of $\RR^3\setminus Z$.

\newcommand{\tM}{\tilde M}
\begin{theorem}\label{R-to-infinity-theorem}
Let $(S_n,1, R_n)$ be a sequence of examples with $R_n\to\infty$.
Let $\sigma_n$ be a sequence of screw motions of $\RR^3$ that map $H$ to itself.
Let
\[
  M_n = \sigma_n( \overline{S_n \cup \rho_Z S_n}).
\]
In other words, $M_n$ is the  full genus-$g$ example (of which  $S_n$ is the subset in the interior of $H^+$)  followed by 
the screw motion $\sigma_n$.
Then (after passing to a subsequence), the $M_n$ converge
smoothly on compact sets to a properly embedded, multiplicity-one minimal surface $M$ in $\RR^3$.
 Furthermore, there is a solid cylinder $C$ about $Z$ such that $M\setminus C$ is the union
of two multigraphs.
\end{theorem}

Thus the family $\FF$ of all such subsequential limits~$M$ (corresponding to arbitrary sequences of $M_n$ and 
$\sigma_n$) is compact with respect to smooth
convergence.  It is also closed under screw motions that leave $H$ invariant.  Those two
facts immediately imply the following corollary:

\begin{corollary}\label{R-to-infinity-corollary} 
Let $\FF$ be the family of all such subsequential limits.
For each solid cylinder $C$ around $Z$, each $M\in \FF$, and each $p\in C\cap M$, the
curvature of $M$ at $p$ is bounded by a constant $k(C)<\infty$ depending
only on $C$ (and genus).
\end{corollary}

\begin{proof}[Proof of theorem~\ref{R-to-infinity-theorem}]
Let $0<d_1^n < d_2^n < \dots < d_g^n$ be the distances of the points in $S_n\cap Y^+$ to the origin.
By passing to a subsequence, we may assume that the limit
\[
 d_k:= \lim_{n\to\infty} d_k^n  \in [0,\infty]
\]
exists for each $k$.  Let $d$ be the largest finite element of $\{d_k: k=1,\dots, g\}$.
By passing to a further subsequence, we may assume that the $\sigma_nM_n$ converge
as sets to a limit set $M$.

Let $C$ be a solid cylinder of radius $\ge (\lambda+1)(1+d)$ around 
  $Z$ where $\lambda$ is as in
Corollary~\ref{form-of-graph-corollary}
for $\eps=1$.   Let $\hat C$ be any larger solid cylinder around $Z$.
By corollary~\ref{form-of-graph-corollary} (see also remark~\ref{form-of-graph-remark}),
for all sufficiently large $n$,  $M_n\cap (\hat C\setminus C)$ is the union of two
 smooth multigraphs, and for each vertical line $V$
in $\hat C\setminus C$, each connected component of $V\setminus H$ intersects $\tM_n$ at most twice.
In fact, all but one such component must intersect $M_n$ exactly once.

By standard estimates for minimal graphs, the convergence $M_n\to M$ is smooth
and multiplicity $1$ in the region $\RR^3\setminus C$.  
It follows immediately that 
$M\setminus C$
is the union of two multigraphs, and that
 the area blowup set
\[
    Q:= \{ p:\text{$ \limsup_{n\to\infty} \area(\sigma_nM_n\cap \BB(p,r))=\infty$ for every $r>0$} \}
\]
is contained in $C$.

The halfspace theorem for area blowup sets \cite{white-controlling-area}*{7.5} says that if an area blowup set is contained in a halfspace
of $\RR^3$, then that blowup set must contain a plane.  Since $Q$ is contained in the cylinder $C$,
it is contained in a halfspace but does not contain a plane.
Thus $Q$ must be empty.
Consequently, the areas of the $M_n$ are uniformly bounded locally.  Since the genus is also bounded,
we have, by the General Compactness Theorem~\ref{general-compactness-theorem}, that $M$ is a smooth embedded minimal hypersurface, and that 
either
\begin{enumerate}
\item the convergence $M_n\to M$ is smooth and multiplicity $1$, or
\item the convergence $M_n\to M$ is smooth with some multiplicity $m>1$ away from a discrete set.
   In this case, $M$ must be stable.
\end{enumerate}

Since the multiplicity is $1$ outside of the solid cylinder $C$, it follows that the convergence
$M_n\to M$ is everywhere smooth with multiplicity $1$.
\end{proof}

\begin{theorem}\label{asymptotic-to-H-theorem}
Suppose that in Theorem~\ref{R-to-infinity-theorem}, the screw motions $\sigma_n$ are all the identity map.
Let $M$ be a subsequential limit of the $M_n$, and suppose that $M\ne H$.
Then $M\cap H=X\cup Z$ and $M$ is asymptotic to $H$ at infinity.
\end{theorem}

\begin{proof}
Since $M_n\cap H=X\cup Z$ for each $n$, the smooth convergence implies that $M$ cannot
intersect $H$ transversely at any point not in $X\cup Z$.  It follows from the strong maximum principle
that $M$ cannot touch $H\setminus (X\cup Z)$.

Since $M$ is embedded, has finite topology, and has infinite total curvature, it follows
from work by Bernstein and Breiner~\cite{bernstein-breiner-conformal} 
 or by Meeks and Perez~\cite{meeks-perez-end} that $M$ is asymptotic to some helicoid $H'$
at infinity.  The fact that $M\cap H=X\cup Z$ implies that $H'$ must be $H$.

The works of Bernstein-Breiner and Meeks-Perez quoted in the previous
paragraph rely on many deep results of Colding and Minicozzi.
We now give a more elementary proof that $M$ is asymptotic to a helicoid at infinity.

According to theorem~4.1 of~\cite{hoffman-white-geometry}, a properly immersed nonplanar
minimal surface in $\RR^3$ with finite genus, one end, and bounded curvature must be asymptotic
to a helicoid and must be conformally a once-punctured Riemann surface provided it contains $X\cup Z$
and provided it intersects some horizontal plane $\{x_3=c\}$ in a set that, outside of a compact region
in that plane, consists of two disjoint smooth embedded curves tending to $\infty$.
Now $M$ contains $X\cup Z$ and has 
bounded curvature (by corollary~\ref{R-to-infinity-corollary}).   
Thus to prove theorem~\ref{asymptotic-to-H-theorem}, it suffices to prove lemmas~\ref{one-end-lemma} 
and~\ref{z=0-level-lemma} below.
\end{proof}

\begin{lemma}\label{one-end-lemma}
Let $M$ be as in theorem~\ref{asymptotic-to-H-theorem}.  Then $M$ has exactly one end.
\end{lemma}

\begin{proof}
Let $Z(R)$ denote the solid cylinder with axis $Z$ and radius $R$.
By theorem~\ref{R-to-infinity-theorem}, for all sufficiently large $R$, the set 
\[
   M \setminus Z(R)
\]
is the union of two connected components (namely multigraphs) that are related to each other by $\rho_Z$.
We claim that for any such $R$, the set
\[
   M\setminus (Z(R)\cap \{|z|\le R\}) \tag{*}
\]
contains exactly one connected component. 
To see that is has exactly one component, 
let $\mathcal{E}$ be the component of \thetag{*} containing $Z^+\cap\{z>R\}$. 
 Note that $\mathcal{E}$ is invariant under
$\rho_Z$.  Now $\mathcal{E}$ cannot be contained in $Z(R)$ by the maximum principle (consider
catenoidal barriers).    Thus $\mathcal{E}$ contains one of the two connected components of 
  $M\setminus Z(R)$.
By $\rho_Z$ symmetry, it must then contain both components of $M\setminus Z(R)$. 
 It follows that if the set~\thetag{*} had
a connected component other than $\mathcal{E}$, that component would have to lie
in $Z(R) \cap \{z<-R\}$.  But such a component would violate
the maximum principle.
\end{proof}

\begin{lemma}\label{z=0-level-lemma}
Let $M$ be as in theorem~\ref{asymptotic-to-H-theorem}.  Then $M\cap \{z=0\}$ is the union
of $X$ and a compact set.
\end{lemma}

\begin{proof}
In the following argument, it is convenient to choose the angle function $\theta$
on $H^+$ so that $\theta=0$ on $X^+$, $\theta=\pi/2$ on $Y^+$, and $\theta=\pi$
on $X^-$.  

By theorem~\ref{R-to-infinity-theorem}, for all sufficiently large $R$, the set 
\[
   M \setminus Z(R)
\]
is the union of two multigraphs that are related to each other by $\rho_Z$.

By the smooth convergence $M_n\to M$ together with
 corollary~\ref{form-of-graph-corollary} and remark~\ref{form-of-graph-remark}, 
 one of the components of $M\setminus Z(R)$ 
 can be parametrized
as
\[
   (r\cos\theta, r\sin\theta, f(r,\theta)) \qquad (r\ge R,\,  \theta \in \RR)
\]
where
\begin{equation}\label{boundary-condition}
f(r,0)\equiv 0
\end{equation}
and
\begin{equation}\label{new-bound}
    \theta-\pi < f(r,\theta) < \theta + \pi.
\end{equation}
(The bound~\eqref{new-bound} looks different from the bound~\eqref{trapped}
 in corollary~\ref{form-of-graph-corollary} 
because there we were measuring $\theta$ from $Y^+$ whereas here we are measuring it from $X^+$.)

Of course $f$ solves the minimal surface equation in polar coordinates.

For $0<s<\infty$, define a function $f_s$ by
\[
    f_s(r,\theta) =  \frac1{s} f( s r, \theta).
\]
Going from $f$ to $f_s$ corresponds to dilating $S$ by $1/s$.  
(To be more precise, $(r,\theta) \mapsto (r\cos\theta, r\sin\theta, f_s(r,\theta))$
parametrizes the dilated surface.)
Thus the function $f_s$ will also solve the polar-coordinate minimal surface equation.
By~\eqref{new-bound},
\begin{equation}\label{f_s-bound}
   \theta - \pi \le  s f_s(r,\theta) \le \theta+ \pi.
\end{equation}
By the Schauder estimates for $f_s$ and by the bounds~\eqref{f_s-bound},
the functions $s f_s$ converge smoothly (after passing to a subsequence) as $s\to 0$ to a harmonic
function $g(r,\theta)$ defined for all $r>0$ and satisfying
\begin{equation}\label{g-bound}
   \theta - \pi \le g \le \theta + \pi.
\end{equation}
Here ``harmonic'' is with respect to the standard conformal structure on $\sS^2$ (or equivalently
on $\RR^2$), so $g$ satisfies the equation
\[
  g_{rr} + \frac1r g_{r} + \frac1{r^2}g_{\theta\theta} = 0.
\]
Now define $G: \RR^2\to \RR$ by
\[
  G(t,\theta) = g(e^t, \theta).
\]
Then $G$ is harmonic in the usual sense: $G_{tt} + G_{\theta\theta}=0$.

By~\eqref{g-bound}, $G(t,\theta)-\theta$ is a bounded, entire harmonic function, and therefore is constant.
Also, $G-\theta$ vanishes where $\theta=0$, so it vanishes everywhere.  Thus
\[
  g(r,\theta)\equiv \theta,
\]
and therefore $\pdf{}{\theta}g \equiv 1$.

The smooth convergence of $s f_s$ to $g$ implies that
\[
  \lim_{r\to\infty} r \pdf{}{\theta}f(r,\theta) = 1,
\]
where the convergence is uniform given bounds on $\theta$.
Thus there is a $\rho<\infty$ such that for each $r\ge \rho$,
the function
\[
    \theta\in [-2\pi, 2\pi] \mapsto f(r,\theta)
\]
is strictly increasing.  Thus it has exactly one zero in this interval,
namely $\theta=0$.  But by the bounds~\eqref{new-bound}, $f(r,\theta)$ never
vanishes outside this interval.   Hence for $r\ge \rho$, $f(r,\theta)$
vanishes if and only if $\theta=0$.

So far we have only accounted for one component of $M\setminus Z(R)$.
But the behavior of the other component follows by $\rho_Z$ symmetry.
\end{proof}


\section{The proof of theorem~\ref{theorem2}} \label{Proof_of_theorem_2}

For the reader's convenience, we restate theorem~\ref{theorem2}
before proving it:

\begin{theorem}\label{R3limits2}
Let $R_n$ be a sequence of radii tending to
infinity.   For each $n$, let $M_{+}(R_n)$ and $M_{-}(R_n)$ be 
genus-$g$ surfaces in $\sS^2(R_n)\times \RR$ satisfying the list
of properties in theorem~\ref{theorem1}, where $H$ is the helicoid of pitch $1$ 
and $h=\infty$.
 Then, after passing to
a subsequence, the $M_+(R_n)$ and $M_{-}(R_n)$ converge smoothly on
compact sets to limits $M_+$ and $M_{-}$ with the
following properties:
\begin{enumerate}
\item\label{helicoidlike} $M_{+}$ and $M_{-}$ are complete, properly
embedded minimal surfaces in $\RR^3$ that are asymptotic to the
standard helicoid $H\subset\RR^3$. 
\item\label{intersection-property} If  $M_s \ne H$, then $M_s\cap H=X\cup Z$ and
   $M$ has sign $s$ at $O$ with respect to $H$.
\item\label{Y-surface-property} $M_s$ is a $Y$-surface.
\item\label{point-count} 
  $\|M_s\cap Y\| = 2\,\|M_s\cap Y^+\| +1 = 2\,\genus(M_s)+1$.
 \item\label{genus-bound} If $g$ is even, then
      $M_{+}$ and $M_{-}$ each have genus at most $g/2$.
      If $g$ is odd, then $\genus(M_{+})+\genus(M_{-})$
       is at most $g$.
 \item\label{genus-parity} The genus of $M_{+}$ is even. 
          The genus of $M_{-}$ is odd.
\end{enumerate}
\end{theorem}

\begin{proof}
Smooth convergence to a surface asymptotic to $H$ was proved in
theorems~\ref{R-to-infinity-theorem}
and~\ref{asymptotic-to-H-theorem}.
Statements~\eqref{helicoidlike} and \eqref{intersection-property}
 follow immediately from the smooth convergence
 and the corresponding properties of the the surfaces $M_s(R_n)$.
 
Next we prove Statement~\eqref{Y-surface-property}.
Note that $\rho_Y$ invariance of $M_s$ follows immediately
from the smooth convergence and the $\rho_Y$ invariance of the $M_s(R_n)$.
To show that $M_s$ is a $Y$-surface, we must show that $\rho_Y$
acts on the first homology group of $M_s$ by multiplication by $-1$.  
Let $\gamma$  be a closed curve in $M_s$.  We must  show
that $\gamma\cup \rho_Y \gamma$ bounds a region of
$M^\infty _s$.  The curve $\gamma$ is approximated by curves
$\gamma_n \subset M_s(R_n)$. Since each $M_s(R_n)$ is a
$Y$-surface, $\gamma_n \cup \rho_Y \gamma_n$ bounds a compact
region in $W_n\subset M_s(R_n)$. These regions converge uniformly
on compact sets to a region $W\subset M_s$ with boundary
$\gamma \cup \rho_Y \gamma$ but a priori that region might not be compact.
By the maximum principle, each $W_n$ is contained in the smallest slab
of the form $\{|z|\le a\}$ containing $\gamma_n\cup\rho_Y\gamma_n$.
Thus $W$ is also contained in such a slab.  Hence $W$ is compact, since
 $M_s$ contains only one end and that end is helicoidal (and therefore is not
 contained in a slab.)  This completes the proof of 
 Statement~\eqref{Y-surface-property}.


Next we prove Statement~\eqref{point-count}.
Because
$M_s$ is a $Y$-surface,  
\[
   \|  M_s \cap Y \| =2 -\chi(M_s)
\]
by Statement~3 of Proposition~\ref{Y-surface-topology-propostion}.
Also,
\[
   \chi(M_s)= 2-2\genus(M_s)-1
\]
since $M_s$ has exactly one end.  Combining the
last two identities gives Statement~\eqref{point-count}.
(Note that $M_s\cap Y$ consists of the points of $M_s\cap Y^+$,
the corresponding points in $M_s\cap Y^-$, and the origin.)

To prove Statement~\eqref{genus-bound}, note
that $M_{+}(R_n)\cap Y^+$ contains exactly $g$ points.
By passing to a subsequence, we can assume (as $n\to\infty$)
that $a$ of those points stay a bounded distance from $Z$,
that $b$ of those points stay a bounded distance from $Z^*$, and
that for each of the remaining $g-a-b$ points, the distance from the point to $Z\cup Z^*$
  tends to infinity.

By smooth convergence,
\[
  \|M_{+}\cap Y^+\| = a.
\]
so the genus of $M_{+}$ is $a$ by Statement~\eqref{point-count}.

If $g$ is even, then $M_{+}(R_n)$ is symmetric by reflection in the totally
geodesic cylinder
\begin{equation}\label{cylinder-again}
   \dist(\cdot, Z) = \dist(\cdot, Z^*).
\end{equation}
It follows that $a=b$, so $\genus(M_{+})\le a \le g/2$.
The proof for $M_{-}$ and $g$ even is identical.

If $g$ is odd, then $M_{-}(R_n)$ is obtained from $M_{+}(R_n)$ by reflection
in the cylinder~\eqref{cylinder-again}.  Hence exactly $b$ of the points
of $M_{-}(R_n)\cap Y^+$ stay a bounded distance from $Z$.
It follows that $M_{-}\cap Y^+$ has exactly $b$ points, and therefore
that $M_{-}$ has genus $b$.  Hence
\[
 \genus(M_{+}) + \genus(M_{-}) = a + b \le g,
\]
which completes the proof of statement~\eqref{genus-bound}.

It remain only to prove statement~\eqref{genus-parity}: the genus of $M_{+}$
is even and the genus of $M_{-}$ is odd. 
By statement~\eqref{point-count}, this is equivalent to showing that $\|M_s\cap Y^+\|$ is even or odd
according to whether $s$ is $+$ or $-$.
Let
\[
   S = S_s = M_s\cap H^+.
\]
Then $M_s\cap Y^+= S\cap Y$, so it suffices to show that $\|S\cap Y\|$
is even or odd according to whether $s$ is $+1$ or $-1$.
By Proposition~\ref{Y-surface-topology-propostion}, 
this is equivalent to showing that $S$ has two ends
if $s$ is $+$ and one end if $s$ is $-$.

Let $Z(R)$ be the solid cylinder of radius $R$ about $Z$.  
By~Corollary~\ref{form-of-graph-corollary} (see also the first three paragraphs 
of the proof of Lemma~\ref{z=0-level-lemma}), we can choose
$R$ sufficiently large that $S\setminus Z(R)$ has two components.
One component is a multigraph on which $\theta$ goes from $\theta(X^+)$ to $\infty$, and on which $z$ is unbounded above.  The other component is a multigraph
on which $\theta$ goes from $\theta(X^-)$ to $-\infty$ and on which $z$ is unbounded
below.    In particular, if we remove a sufficiently large finite solid cylinder
\[
   C: = Z(R) \cap \{|z|\le A\}  
\]
from $S$, then the resulting surface $S\setminus C$ has two components.
(We choose $A$ large enough that $C$ contains all points of $H\setminus Z$
at which the tangent plane is vertical.)
One component has in its closure $X^+\setminus C$ and $Z^+\setminus C$,
and the other component has in its closure $X^-\setminus C$ 
and $Z^-\setminus C$.   Consequently $Z^+$ and $X^+$ belong to an end
of $S$, and $X^-$ and $Z^-$ also belong to an end of $S$.

Thus $S$ has one or two ends according to whether $Z^+$ and $X^-$ belong
to the same end of $S$ or different ends of $S$.  Note that they belong to the same
end of $S$ if and only if every neighborhood of $O$ contains a path in $S$
with one endpoint in $Z^+$ and the other endpoint in $X^-$, i.e., if and only
if $M_s$ is negative at $O$.  
By the smooth convergence, $M_s$ is negative at $O$ if and only if the $M_s(R_n)$
are negative at $O$, i.e., if and only if $s=-$.
\end{proof}

\begin{corollary}
If $g$ is $1$ or $2$, then $M_{+}$ has genus $0$ and $M_{-}$ has genus $1$.
\end{corollary}

The corollary follows immediately from statements~\eqref{genus-bound}
and~\eqref{genus-parity} of theorem~\ref{R3limits2}.

\appendix 

\section{Minimal surfaces in a thin hemispherical shell}\label{hemisphere-appendix}

If $D$ is a hemisphere of radius $R$ and $J$ is an interval with 
 $|J|/R$ is sufficiently small, we show that the projection of a minimal surface $M\subset D\times J$ with 
$\partial M\subset \partial D\times J$ onto $D$ covers most of $D$. This result is used in the proof of
theorem~\ref{first-compactness-theorem} to show that a certain
multiplicity is $2$, not $0$.

\begin{theorem}\label{all-or-nothing-theorem}
Let $D=D_R$ be an open hemisphere in the sphere $\sS^2(R)$
and let $J$ be an interval with length $|J|$.  If $|J|/R$ is sufficiently small,
the following holds.
Suppose that $M$
is a nonempty minimal surface in $D\times J$ with $\partial M$ in $(\partial D)\times J$.
Then every point $p\in D$ is within distance $4 |J|$ of $\Pi(M)$,
where $\Pi: D\times I \to D$ is the projection map.
\end{theorem}

\begin{proof}
By scaling, it suffices to prove that if $J=[0,1]$ is an interval of length $1$, then every point
$p$ of $D_R$ is within distance $4$ of $\Pi(M)$ provided $R$ is sufficiently large.

For the Euclidean cylinder $\BB^2(0,2)\times [0,1]$, 
the ratio of the area of the cylindrical side (namely $4\pi$) to the sum of the areas
of top and bottom (namely $8\pi$) is less than $1$.  Thus the same is true
for all such cylinders of radius $2$ and height $1$ in $\sS^2(R)\times \RR$, provided $R$ is sufficiently large.
It follows that there is a catenoid in $\sS^2(R)\times \RR$ whose boundary consists of
the two circular edges of the cylinder.

Let $\Delta\subset D_R$ be a disk of radius $2$ containing the point $p$.
By the previous paragraph, we may suppose that $R$ is sufficiently large
that there is a catenoid $C$ in $\Delta\times[0,1]$ whose boundary consists of the two circular
edges of $\Delta\times [0,1]$.  In other words, $\partial C= (\partial \Delta)\times(\partial [0,1])$.

Now suppose $\dist(p, \Pi(M))> 4$.  Since $\Delta$ has diameter $4$,
it follows that $\Pi(M)$ is disjoint from $\Delta$ and therefore that $M$ is disjoint from $C$.
 If we slide $C$ around in $D_R\times [0,1]$, it can never
bump into $M$ by the maximum principle.  That is, $M$ is disjoint from the union $K$ of
all the catenoids in $D_R\times [0,1]$ that are congruent to $C$.  Consider all surfaces
of rotation with boundary $(\partial D_R)\times (\partial [0,1])$ that are disjoint from $M$ and from $K$.
The surface $S$ of least area in that collection is a catenoid, and (because it is disjoint from $K$)
it lies with distance $2$ of $(\partial D_R)\times [0,1]$.

We have shown: if $R$ is sufficiently large and if the theorem is false for $D_R$ and $J=[0,1]$,
then there is a catenoid $S$ such that $S$ and $\Sigma:=(\partial D_R)\times [0,1]$ have the same boundary
and such that $S$ lies within distance $2$ of $\Sigma$.

Thus if the theorem were false, there would be a sequence of radii $R_n\to \infty$
and a sequence of catenoids $S_n$ in $D_{R_n}\times[0,1]$ such that
\begin{equation*}
   \partial S_n = \partial \Sigma_n
\end{equation*}
and such that $S_n$ lies within distance $2$ of $\Sigma_n$, where
\[
 \Sigma_n = (\partial D_{R_n})\times [0,1].
\]
We may use coordinates in which the origin is in $(\partial D_{R_n})\times \{0\}$.
As $n\to \infty$, the $\Sigma_n$ converge smoothly to an infinite flat strip $\Sigma=L\times[0,1]$ in $\RR^3$
(where $L$ is a straight line in $\RR^2$),
and the $S_n$ converges smoothly to a minimal surface $S$ such that (i) $\partial S=\partial \Sigma$,
(ii) $S$ has the translational invariance that $\Sigma$ does, and (iii) $S$ lies within a bounded
distance of $\Sigma$.   It follows that $S=\Sigma$.

Now both $\Sigma_n$ and $S_n$ are minimal surfaces.  
(Note that $\Sigma_n$ is minimal, and indeed totally
geodesic, since $\partial D$ is a great circle.)
Because $\Sigma_n$ and $S_n$ have the boundary and converge smoothly to the same limit $S$,
it follows that there is a nonzero Jacobi field $f$ on $S$ that vanishes at the boundary. Since $S$ is flat,
$f$ is in fact a harmonic function.
The common rotational symmetry of $\Sigma_n'$ and $S_n'$ implies that the Jacobi field is translationally
invariant along $S$, and thus that it achieves its maximum somewhere.  (Indeed, it attains its maximum
on the entire line $L\times\{1/2\}$.)  But then $f$ must be constant,
which is impossible since $f$ is nonzero and vanishes on $\partial S$.
\end{proof}


\section{Noncongruence Results}\label{noncongruence-appendix}

In this section, we prove the noncongruence results in theorem~\ref{theorem1}:
that $M_+$ and $M_-$ are not congruent by any orientation-preserving isometry of $\sS^2\times\RR$,
and that if the genus is even, they are not congruent by any isometry of $\sS^2\times\RR$.
For these results, we have to assume that $H$ does not have infinite pitch, i.e., that $H\ne X\times \RR$.
See remark~\ref{extra-symmetry-remark}.

We first consider the periodic case: $M_+$ and $M_-$ are properly embedded
in $\sS^2\times(-h,h)$ where $0<h<\infty$.  
We will assume that the height $h$ is not an integer multiple of the vertical separation between sheets of $H$,
i.e, that the boundary circles of $M_s$ are not vertical translates of $X$.  This assumption makes the proof
simpler, and we never actually use the noncongruence results.
Suppose $\phi: \sS^2\times\RR\to \sS^2\times\RR$ is an isometry
such that $\phi(M_{+})=M_{-}$. 
We must show that $\phi$ is orientation-reversing on $\sS^2\times\RR$
and that the genus of $M_{+}$ is odd.
 Let $G= \partial M_{+}=\partial M_{-}$.
Then $\phi(G)=G$.  Note that $G$ consists of two horizontal circles, and that
the only two vertical lines that intersect both circles are $Z$ and $Z^*$ (by the hypothesis
on $h$.)  It follows that the symmetry group of $G$ is the same
as the symmetry group of $H\cap \{|z|<h\}$.
In particular, from $\phi(G)=G$ we see that
\begin{enumerate}[\upshape (i)]
\item $\phi$ belongs to
$
  \mathcal{S}:= \{I, \rho_X, \rho_Y, \rho_Z(=\rho_{Z^*}) \} 
$
if $\phi$ is orientation-preserving on $\sS^2\times\RR$, and
\item $\phi =\mu_E\circ \psi$ for 
  some $\psi$ in $\mathcal{S}$ if $\phi$ is 
  orientation-reversing on $\sS^2\times\RR$.
\end{enumerate}
Now $M_{+}\ne M_{-}$ since their signs at $O$ are different.
Thus $\phi \notin \mathcal{S}$, 
 since $M_{+}$ is invariant under the maps in $\mathcal{S}$.

This proves that $\phi$ is orientation-reversing on $\sS^2\times\RR$.
 Thus $\phi=\mu_E\circ \psi$
for some $\psi\in \mathcal{S}$, so
\begin{equation}
  M_{-} =\mu_E\circ \psi M_{+} = \mu_E(M_{+}) 
\end{equation}
since $M_{+}$ is invariant under the maps in $\mathcal{S}$.
If the genus were even, then $\mu_E(M_{+})=M_{+}$,
contradicting the fact that $M_{+}\ne M_{-}$. Thus the genus is odd.

This completes the proof of the noncongruence assertions in the case $h<\infty$
  (assuming that   $h$ not a multiple of the vertical separation between sheets of $H$).
The proof of the noncongruence assertions when $h=\infty$ is very similar. 
Suppose $\phi: \sS^2\times\RR\to \sS^2\times\RR$ is an isometry
such that $\phi(M_{+})=M_{-}$. 
We must show that $\phi$ is orientation-reversing on $\sS^2\times\RR$
and that the genus of $M_{+}$ is odd.
Note that $M_s$ contains $Z\cup Z^*$, and (since $M_s\cup H = X\cup Z\cup Z^*$) that those are the
only vertical lines contained in $M_s$.  Also, $M_s$ contains $X$, but no other horizontal great circle,
since if it did, it would be invariant under a vertical screw motion and therefore would have infinite genus.
Consequently $\phi$ maps $X\cup Z\cup Z^*$ to itself.  The rest of the proof is almost identical to the 
proof when $h<\infty$.

\begin{remark}\label{extra-symmetry-remark}
If $H$ is the infinite-pitch helicoid $X\times\RR$, 
then $H$ has an extra symmetry, namely the reflection $\mu_Y$ in the totally geodesic
cylinder $Y\times\RR$.  This extra symmetry allows us to choose $M_+$ and $M_-$ to violate the noncongruence
assertions of theorem~\ref{theorem1}: specifically, once we have chosen $M_+$, we can let
 $M_-$ be $\mu_Y M_+$.  Then $M_+$ and $M_-$ will satisfy all the assertions of theorem~\ref{theorem1}
 except for the noncongruence assertions, which they will violate.
 \end{remark}

\nocite{pedrosa-ritore}
\nocite{hoffman-wei}
\newcommand{\hide}[1]{}

\end{document}